\documentclass[11pt,twoside,a4paper]{amsart}

\usepackage{amssymb}
\usepackage{vmargin}
\usepackage{amscd}
\usepackage{stmaryrd}
\usepackage{mathrsfs}
\usepackage[all]{xy}
\usepackage{xr-hyper}
\usepackage{hyperref}
\hypersetup{colorlinks,
linkcolor=black,
citecolor=black,
urlcolor=black}

\setmargins{32mm}{20mm}{14.6cm}{22cm}{1cm}{1cm}{1cm}{1cm}

\newcommand\la{\leftarrow}

\newcommand\id{\mathrm{id}}

\newcommand\ten{\otimes}

\newcommand\eps{\epsilon}

\newcommand\CC{\mathrm{C}}

\renewcommand\H{\mathrm{H}}
\newcommand\z{\mathrm{Z}}

\newcommand\HH{\mathrm{HH}}
\newcommand\HC{\mathrm{HC}}

\newcommand\HN{\mathrm{HN}}

\newcommand\Z{\mathbb{Z}}
\newcommand\Q{\mathbb{Q}}

\newcommand\bG{\mathbb{G}}

\newcommand\bL{\mathbb{L}}

\newcommand\C{\mathcal{C}}

\newcommand\cA{\mathcal{A}}
\newcommand\cB{\mathcal{B}}
\newcommand\cC{\mathcal{C}}
\newcommand\cD{\mathcal{D}}

\newcommand\cF{\mathcal{F}}

\newcommand\cL{\mathcal{L}}
\newcommand\cM{\mathcal{M}}

\newcommand\cP{\mathcal{P}}

\newcommand\cS{\mathcal{S}}

\newcommand\F{\mathscr{F}}

\newcommand\sF{\mathscr{F}}

\newcommand\sO{\mathscr{O}}
\newcommand\sP{\mathscr{P}}

\newcommand\sX{\mathscr{X}}

\newcommand\fA{\mathfrak{A}}

\newcommand\g{\mathfrak{g}}

\newcommand\cHom{\mathcal{H}\!\mathit{om}}

\newcommand\Ho{\mathrm{Ho}}

\newcommand\Alg{\mathrm{Alg}}

\newcommand\comm{\mathrm{com}}

\newcommand\Hom{\mathrm{Hom}}
\newcommand\Map{\mathrm{Map}}
\newcommand\map{\mathrm{map}}

\newcommand\cyc{\mathrm{cyc}}

\newcommand\HHom{\underline{\mathrm{Hom}}}

\newcommand\DDer{\underline{\mathrm{Der}}}

\newcommand\Ext{\mathrm{Ext}}

\newcommand\End{\mathrm{End}}

\newcommand\Iso{\mathrm{Iso}}
\newcommand\Coiso{\mathrm{CoIso}}

\newcommand\cone{\mathrm{cone}}
\newcommand\cocone{\mathrm{cocone}}

\newcommand\per{\mathrm{per}}

\newcommand{\brh}{\llbracket \hbar \rrbracket}

\newcommand{\llb}{\llbracket}
\newcommand{\rrb}{\rrbracket}

\newcommand\coker{\mathrm{coker\,}}

\newcommand\ch{\mathrm{ch}}

\newcommand\At{\mathrm{At}}

\newcommand\Spec{\mathrm{Spec}\,}

\newcommand\Set{\mathrm{Set}}

\newcommand\Mor{\mathrm{Mor}\,}

\newcommand\Aff{\mathrm{Aff}}

\newcommand\Sp{\mathrm{Sp}}
\newcommand\PreSp{\mathrm{PreSp}}
\newcommand\PreBiSp{\mathrm{PreBiSp}}
\newcommand\BiSp{\mathrm{BiSp}}
\newcommand\BiIso{\mathrm{BiIso}}
\newcommand\BiLag{\mathrm{BiLag}}

\newcommand\Pol{\mathrm{Pol}}

\newcommand\Lag{\mathrm{Lag}}

\newcommand\Comp{\mathrm{Comp}}

\newcommand\nondeg{\mathrm{nondeg}}

\newcommand\ad{\mathrm{ad}}

\newcommand\<{\langle}
\renewcommand\>{\rangle}
\newcommand\Lim{\varprojlim}
\newcommand\LLim{\varinjlim}

\newcommand\ho{\mathrm{ho}\!}
\newcommand\into{\hookrightarrow}
\newcommand\onto{\twoheadrightarrow}
\newcommand\abuts{\implies}
\newcommand\xra{\xrightarrow}

\newcommand\pr{\mathrm{pr}}

\newcommand\dmd{\diamond}
\newcommand\bt{\bullet}
\newcommand\by{\times}

\newcommand\mc{\mathrm{MC}}
\newcommand\mmc{\underline{\mathrm{MC}}}

\newcommand\Perf{\mathrm{Perf}}

\newcommand\Symm{\mathrm{Symm}}

\newcommand\GL{\mathrm{GL}}

\newcommand\Mat{\mathrm{Mat}}
\newcommand\et{\acute{\mathrm{e}}\mathrm{t}}

\newcommand\Tot{\mathrm{Tot}\,}
\newcommand\tot{\mathrm{tot}}

\newcommand\tr{\mathrm{tr}}
\newcommand\ev{\mathrm{ev}}

\newcommand\pd{\partial}

\newcommand\half{\frac{1}{2}}

\newcommand\Tor{\mathrm{Tor}}

\newcommand\rig{\mathrm{rig}}

\newcommand\gr{\mathrm{gr}}

\newcommand\Fil{\mathrm{Fil}}

\newcommand\DR{\mathrm{DR}}

\newcommand\op{\mathrm{opp}}

\newcommand\co{\colon\thinspace}

\newcommand\oR{\mathbf{R}}

\newcommand\oL{\mathbf{L}}

\newcommand\uleft\underleftarrow
\newcommand\uline\underline
\newcommand\uright\underrightarrow

\newcommand{\tps}{\texorpdfstring}

\DeclareMathOperator{\hatHHom}{\uline{\mathrm{H}\widehat{\mathrm{o}}\mathrm{m}}}
\DeclareMathOperator{\hatTot}{\mathrm{T}\widehat{\mathrm{o}}\mathrm{t}}

\newtheorem{theorem}{Theorem}[section]
\newtheorem{proposition}[theorem]{Proposition}
\newtheorem{corollary}[theorem]{Corollary}

\newtheorem{lemma}[theorem]{Lemma}
\newtheorem*{theorem*}{Theorem}
\newtheorem*{proposition*}{Proposition}
\newtheorem*{corollary*}{Corollary}
\newtheorem*{lemma*}{Lemma}
\newtheorem*{conjecture*}{Conjecture}

\theoremstyle{definition}
\newtheorem{definition}[theorem]{Definition}
\newtheorem*{definition*}{Definition}

\newtheorem*{notation*}{Notation}

\theoremstyle{remark}
\newtheorem{example}[theorem]{Example}

\newtheorem{remark}[theorem]{Remark}
\newtheorem{remarks}[theorem]{Remarks}

\newtheorem{properties}[theorem]{Properties}

\newtheorem*{example*}{Example}
\newtheorem*{examples*}{Examples}
\newtheorem*{remark*}{Remark}
\newtheorem*{remarks*}{Remarks}
\newtheorem*{exercise*}{Exercise}
\newtheorem*{property*}{Property}
\newtheorem*{properties*}{Properties}

\begin{document}

\begin{abstract}
 We introduce the notions of shifted bisymplectic and shifted double Poisson structures on differential graded associative algebras, and more generally on non-commutative derived moduli functors with well-behaved cotangent complexes. For smooth algebras concentrated in degree $0$, these structures recover the classical notions of bisymplectic and double Poisson structures, but in general they involve an infinite hierarchy of higher homotopical data, ensuring that they are invariant under quasi-isomorphism.  The structures induce shifted symplectic and shifted Poisson structures on the underlying commutative derived moduli functors, and also on underlying representation functors. 
 
 We show that there  are canonical equivalences between the spaces of  shifted  bisymplectic structures and of non-degenerate $n$-shifted double Poisson structures. We also give canonical shifted bisymplectic and bi-Lagrangian structures on various derived non-commutative moduli functors of modules over Calabi--Yau dg categories. 
 
 Unlike their commutative counterparts, these structures enjoy a formal integration property, which we exploit to  show that Calabi--Yau and pre-Calabi--Yau structures on a dg algebra correspond respectively to bisymplectic and double Poisson structures on its quotient prestacks by the adjoint $\bG_m$-action.
\end{abstract}

\title[Shifted bisymplectic and  double Poisson structures]{Shifted bisymplectic and  double Poisson structures on non-commutative derived prestacks}
\author{J.P.Pridham}

\maketitle


\section*{Introduction}

In derived algebraic geometry, most examples of the shifted symplectic and Lagrangian structures of \cite{KhudaverdianVoronov,BruceGeomObjects,PTVV} 
  on derived moduli stacks arise for moduli problems which are linear in nature, such as moduli of perfect complexes on Calabi--Yau varieties.  
The idea behind this paper is that for such problems,  the shifted symplectic and Lagrangian structures are just commutative shadows of structures living on a non-commutative moduli functor.

For smooth algebras, the non-commutative counterpart of a  $0$-shifted symplectic structure was introduced in \cite{CrawleyBoeveyEtingofGinzburgNCGeomQuivers} as  a bisymplectic structure, which is a closed $2$-form in the cyclic (or Karoubi) de Rham double complex satisfying a non-degeneracy condition. The straightforward generalisation of this to a quasi-isomorphism-invariant definition on quasi-free differential graded associative algebras (DGAAs)  just replaces the kernel $\ker(d \co \Omega^2_{\cyc} \to \Omega^3_{\cyc})$ with the product total complex $\Tot^{\Pi}\Omega^{\ge 2}_{\cyc}$, so an $n$-shifted bisymplectic structure on $(A,\delta)$  consists of a sequence $(\omega_i \in (\Omega^i_{\cyc}(A))_{i+n})_{i \ge 2}$ with $d\omega_i=\delta\omega_{i+1}$, together with a non-degeneracy condition on $\omega_2$ phrased in terms of quasi-isomorphism rather than isomorphism.  

One source of motivation for constructing shifted symplectic and Lagrangian structures is that they give rise to shifted Poisson structures, by \cite{KhudaverdianVoronov,poisson,CPTVV} in the symplectic case and \cite{MelaniSafronovII} in the Lagrangian case. A shifted Poisson structure is essentially an $L_{\infty}$ structure in which all of the operations are multiderivations; in the  $0$-shifted non-singular underived setting, this just reduces to a Poisson bracket, and then the non-commutative counterpart is given by the double Poisson brackets of \cite{vdBerghDoublePoisson}, which on an associative algebra $A$ map from  $A \by A$ to  $A \ten A$. 
Formulating a quasi-isomorphism-invariant  generalisation of this to DGAAs is much more subtle than for shifted symplectic structures, and involves the complex of non-commutative cyclic polyvectors,  with the $L_{\infty}$-operations lifting to maps from  $A^i$ to  $A^{\ten i}$; there is an equivalent characterisation given by protoperads in \cite{LerayProperadDoublePoisson}.  

Our main results, Corollary \ref{compatcor2} and its generalisation Theorem \ref{Artinthm}, establish the equivalence between $n$-shifted bisymplectic structures  and non-degenerate $n$-shifted double Poisson structures on non-commutative derived moduli functors with well-behaved cotangent complexes, and Corollary \ref{coisocor} gives a similar result for $n$-shifted bi-Lagrangians. At this level of generality,  there is a further subtlety in defining shifted double Poisson structures because of their limited functoriality, similar to  the difficulties with shifted Poisson structures on (commutative) derived moduli stacks. The solution is to take \'etale approximations of the functors  using the stacky DGAAs of \cite{NCstacks}, which act as a derived non-commutative analogue of Lie algebroids. 
A stacky DGAA is a bidifferential bigraded associative algebra, where we think of one grading as stacky and the other as derived,  only defining equivalences relative to the derived structure, and the formulation of shifted double Poisson structures extends to them in an \'etale functorial way.  

Given an $n$-shifted bisymplectic non-commutative derived moduli functor $F$, and a Frobenius algebra $M$,    Propositions \ref{weilArtinprop} and \ref{weilhgsprop} show that the non-commutative Weil restriction of scalars functor $F(M\ten-)$ is also $n$-shifted bisymplectic. When applied to matrix algebras $M$,   this amounts to saying that the functors of representations are  $n$-shifted bisymplectic, which in turn implies  that the  underlying representation functors on commutative objects are all $n$-shifted symplectic. 

From this perspective, functors of representations of an algebra $A$ arise by Weil restriction of the quotient prestack $[\Spec^{nc} A/\bG_m]$ of the adjoint action by units (rather than  of $\Spec^{nc} A$ itself). It then turns out that Calabi--Yau and pre-Calabi--Yau structures on $A$ correspond precisely to shifted bisymplectic and double Poisson structures on $[\Spec^{nc} A/\bG_m]$ (Propositions \ref{CYprop} and \ref{preCYprop}). These calculations are simplified by  a formal integration property unique to  this non-commutative setting: that when a derived moduli functor admits an Artin atlas by an affine object, its   shifted bisymplectic and double Poisson structures are uniquely determined by those on the formal completion of the atlas (Corollaries \ref{integratecorBiSp} and \ref{integratecorDP}). 


Other important results are Theorems \ref{Perfthm1}, \ref{Perfthm2}, \ref{Perfthm2Lag}, \ref{Morthm} and \ref{MorthmLag},  which construct natural shifted bisymplectic and bi-Lagrangian structures on various non-commutative derived moduli functors associated to Calabi--Yau dg algebras and dg categories.  As a consequence of our main theorem, we then show in  Corollaries \ref{Perfpoissoncor} and \ref{Morpoissoncor} and  Remark \ref{CostelloRozenblyumRmk}  that there are natural shifted double Poisson structures on various derived moduli functors of modules associated to relative Calabi--Yau  dg categories.

\smallskip
The structure of the paper is as follows.
\smallskip

In Section \ref{spsn}, we introduce shifted bisymplectic structures, defining them first for non-commutative affine objects (DGAAs concentrated in non-negative homological degrees), then for  derived non-commutative Artin prestacks, and finally for derived non-commutative  prestacks with well-behaved cotangent complexes. The main results are the existence and comparison theorems  listed above, constructing shifted bisymplectic structures on various functors parametrising moduli of modules, and extending shifted bisymplectic structures to   
 Weil restriction of scalars by Frobenius algebras, and hence to representation functors.
 
Section \ref{poisssn}  then introduces  shifted double Poisson structures and establishes the equivalence between non-degenerate shifted double Poisson structures and shifted bisymplectic structures. A shifted double Poisson structure $\pi$  on a DGAA $A$ consists of shifted cyclic $i$-derivations $\pi_i$ for $i \ge 2$, where a shifted $i$-derivation is a map $\pi_i \co  A^i \to A^{\ten i}$   which satisfies  cyclic symmetry or antisymmetry depending on the shift, and acts as a derivation on its final input with respect to the outer $A$-module structure on $A^{\ten i}$. To be a double Poisson structure, $\pi$ must satisfy a higher analogue of the double Jacobi identity, which can be phrased as a Maurer--Cartan condition with respect to a double Schouten--Nijenhuis bracket on the complex $\widehat{\Pol}^{nc}_{\cyc}(A,n)$ of cyclic multiderivations.

In order to compare shifted bisymplectic and double Poisson structures, we adapt the method of \cite{poisson}, which turns out to generalise the Legendre transformation used in \cite{KhudaverdianVoronov}.  The observation which enables us to do this is that there is an  $(n+1)$-shifted double Poisson algebra  $\widehat{\Pol}^{nc}(A,n)$ of $n$-shifted multiderivations whose quotient by commutators recovers the complex $\widehat{\Pol}^{nc}_{\cyc}(A,n)$ of shifted cyclic multiderivations. For any $n$-shifted  double Poisson structure $\pi$, contraction then defines  a multiplicative map $\mu(-,\pi)$ from the associative de Rham complex to $\widehat{\Pol}^{nc}(A,n)$, so taking quotients by commutators gives us a map from the  cyclic de Rham complex to the complex $\widehat{\Pol}^{nc}_{\cyc}(A,n)$ with differential twisted by $\pi$, whose cohomology we can think of as non-commutative Poisson cohomology. 

The condition for an  $n$-shifted bisymplectic structure $\omega$ to be compatible with an $n$-shifted double Poisson structure $\pi$ is then that
\[
 \mu(\omega,\pi) \sim \sum_{i\ge 2} (i-1)\pi_i.
\]
When $\pi$ is non-degenerate, $\mu(-,\pi)$ is a quasi-isomorphism, so there is an essentially unique shifted bisymplectic form $\omega$ compatible with $\pi$. Conversely, for each shifted bisymplectic form $\omega$ there is an essentially unique compatible non-degenerate $n$-shifted double Poisson structure $\pi$. The proof of this is much harder, and proceeds by solving inductively for the terms $\pi_i$ by using towers of obstructions associated to filtered differential graded Lie algebras.

The comparison is  first done (Corollary \ref{compatcor2}) in the affine case and then  (Theorem \ref{Artinthm}) for more general functors, making use of stacky DGAAs and \'etale functoriality. We also discuss shifted double co-isotropic structures and their relation with shifted bi-Lagrangians in \S \ref{coisosn}.

\subsubsection*{Notation}

For a chain (resp. cochain) complex $M$, we write $M_{[i]}$ (resp. $M^{[j]}$) for the complex $(M_{[i]})_m= M_{i+m}$ (resp. $(M^{[j]})^m = M^{j+m}$). We often work with double complexes, in the form of cochain chain complexes, in which case $M_{[i]}^{[j]}$ is the double complex $(M_{[i]}^{[j]})^n_m= M_{i+m}^{j+n}$.  When we have a single grading and need to compare chain and cochain complexes, we silently  make use of the equivalence  $u$ from chain complexes to cochain complexes given by $(uV)^i := V_{-i}$, and refer to this as rewriting the chain complex as a cochain complex (or vice versa). On suspensions, this has the effect that $u(V_{[n]}) = (uV)^{[-n]}$. We also occasionally write $M[i]:=M^{[i]} =M_{[-i]}$ when there is only one grading.

For chain complexes, by default we denote differentials by $\delta$. When we work with cochain chain complexes, the cochain differential is usually denoted by $\pd$. We use the subscripts and superscripts $\bt$ to emphasise that chain and cochain complexes incorporate differentials, with $\#$ used instead when we are working  with the underlying graded objects.

Given $A$-modules $M,N$ in chain complexes, we write $\HHom_A(M,N)$ for the cochain complex given by
\[
 \HHom_A(M,N)^i= \Hom_{A_{\#}}(M_{\#},N_{\#[-i]}),
\]
with differential $\delta f= \delta_N \circ f \pm f \circ \delta_M$,
where $V_{\#}$ denotes the graded vector space underlying a chain complex $V$.

We write $s\Set$ for the category of simplicial sets, and $\map$ for derived mapping spaces, i.e. the right-derived functor of simplicial $\Hom$. (For dg categories, $\map$ corresponds via the Dold--Kan correspondence to the good truncation of $\oR\HHom$.)

\tableofcontents

\section{Shifted bisymplectic structures}\label{spsn}

From now on, we will fix a  graded-commutative algebra $R=R_{\bt}$ in chain complexes over $\Q$. As in \cite[Definition \ref{NCstacks-dgalgdef}]{NCstacks}, 
 we write $dg\Alg(R)$ for the category of associative  $R$-algebras $A_{\bt}$ in  chain complexes, and $dg_+\Alg(R)$ for its full subcategory consisting of objects concentrated in non-negative chain degrees. We refer to objects of $dg\Alg(R)$ as differential graded associative $R$-algebras ($R$-DGAAs).

There is a model structure on $dg_+\Alg(R)$ (resp. $dg\Alg(R)$) in which weak equivalences are quasi-isomorphisms and fibrations are surjections in strictly positive degrees (resp. surjections); any $R$-DGAAs which are freely generated as graded associative $R_{\#}$-algebras (forgetting the differential)  are cofibrant in the model structure on $dg_+\Alg(R)$, as are their retracts.

\subsection{Shifted bisymplectic structures on non-commutative derived affines}\label{spaffsn}

For the purposes of this section, we fix  a cofibrant  $R$-DGAA $A=A_{\bt}$. We will denote the differentials on $A$, $R$ and related constructions by $\delta$. 

\begin{remark} \label{weakerassumptionrmk}
 The condition that $A$ be cofibrant over $R$ is a little stronger than we need. It will suffice to assume that for all left $A$-modules $M$ and right $A$-modules $N$, the natural map $M\ten^{\oL}_A\oL\Omega^1_{A/R}\ten^{\oL}_AN \to M\ten_A\Omega^1_A\ten_AN$ is a quasi-isomorphism. In particular, this applies if $R$ and $A$ are concentrated in non-negative degrees, with  $A_{\#}$ flat as an   $R_{\#}$-module and flat  as a  $A_{\#}$-bimodule. All formulae will then remain valid, except for the term $\gr_F^0\DR_{\cyc}(A/R)=A/[A,A]$ below, which does not affect bisymplectic structures. 
 \end{remark}

 The following is \cite[Definition \ref{NCstacks-Omegadef}]{NCstacks}:
\begin{definition}\label{Omegadefhere}
 Given  a morphism $C \to A$ in $dg\Alg(R)$, we  write $A^{e}:=A\ten_RA^{\op}$ and define the $A^{e}$-module $\Omega^1_{A/C}$ in complexes to be the kernel of the multiplication map $A\ten_C A \to A$, denoting its differential (inherited from $A$) by $\delta$. 

We denote by $\oL\Omega^1_{A/C}$ the cotangent complex, given by 
\[
 \cocone(A\ten^{\oL}_CA \to A),
\]
again regarded as an $A^{e}$-module.

We simply write $\Omega^1_A:=\Omega^1_{A/R}$ and $\oL\Omega^1_A:=\oL\Omega^1_{A/R}$.
\end{definition}

\begin{definition} 
We define $\Omega^{\bt}_{A/R}$ (a double complex) to be the tensor algebra of $\Omega^1_{A/R}$
over $A$, equipped with the de Rham differentials $d \co \Omega^p_{A/R} \to \Omega^{p+1}_{A/R}$ (usually denoted $B$ in the cyclic homology literature) in addition to the structural differentials; the de Rham differential $d$ is determined by the properties that it is an algebra derivation, that $d\circ d=0$ and that $da=a\ten 1 - 1\ten a$ for $a \in A$.

Following \cite[1.24]{karoubiHomologieCyclique} and \cite[7.2.2]{KontsevichSoibelmanAinftyalgcats}, 
 we also form the cyclic (or Karoubi) de Rham double complex $\Omega^{\bt}_{A/R,\cyc}$ as the quotient $\Omega^{\bt}_{A/R}/[ \Omega^{\bt}_{A/R},\Omega^{\bt}_{A/R}]$ by the commutator  $[\omega,\nu]= \omega \cdot \nu -(-1)^{(\deg \omega)(\deg \nu)}\nu \cdot \omega$, where $\deg$ denotes total degree. 
\end{definition}
Explicitly, for $p>0$ elements of $\Omega^p_{\cyc}$ are given by equivalence classes under $C_p$-symmetries of elements of the form $a_1db_1\ldots a_pdb_p$, for $a_i,b_i \in A$.

\begin{definition}\label{DRdef}
Define the de Rham complex $\DR(A/R)$ and   cyclic de Rham complex $\DR_{\cyc}(A/R)$ to be the product total cochain complexes of the double complexes 
\begin{align*}
\Omega^{\bt}_{A/R}&= ( A \xra{d} \Omega^1_{A/R} \xra{d} \Omega^2_{A/R}\xra{d} \ldots),\\
 \Omega^{\bt}_{A/R,\cyc}&= ( A/[A,A] \xra{d} \Omega^1_{A/R,\cyc} \xra{d} \Omega^2_{A/R,\cyc}\xra{d} \ldots),
\end{align*}
so the total differential is $d \pm \delta$.

We define the Hodge filtration $F$ on  $\DR(A/R)$ (resp. $\DR_{\cyc}(A/R)$) by setting    $F^p\DR(A/R) \subset \DR(A/R)$ (resp. $F^p\DR_{\cyc}(A/R) \subset \DR_{\cyc}(A/R)$) to consist of terms $\Omega^i_{A/R}$ (resp. $\Omega^i_{A/R,\cyc}$) with $i \ge p$.
\end{definition}

Properties of the product total complex ensure that a map $f \co A \to B$ induces a quasi-isomorphism $\DR_{\cyc}(A/R) \to \DR_{\cyc}(B/R)$ whenever the maps $\Omega^p_{A/R,\cyc} \to \Omega^p_{B/R,\cyc}  $ are quasi-isomorphisms, which will happen whenever $f$ is a weak equivalence between cofibrant $R$-algebras.

The complex $\DR(A/R)$ has the natural structure of an associative DG algebra over $R$, filtered in the sense that $F^iF^j \subset F^{i+j}$. 

\begin{definition}\label{presymplecticdef}
Define an $n$-shifted pre-bisymplectic structure $\omega$ on $A/R$ to be a cocycle
\[
 \omega \in \z^{n+2}F^2\DR_{\cyc}(A/R).
\]
\end{definition}

Explicitly, this means that $\omega$ is given by an infinite sum $\omega = \sum_{i \ge 2} \omega_i$, with $\omega_i \in (\Omega^i_{A/R,\cyc})_{i-n-2}$ and $d\omega_i = \delta \omega_{i+1}$.

\begin{definition}\label{contractiondef}
 Given  $\phi \in \DDer_R(A,A^{e}):= \HHom_{A^{e}}(\Omega^1_{A/R},A^{e})$, the complex of double derivations, recall from 
\cite[\S 2.6]{CrawleyBoeveyEtingofGinzburgNCGeomQuivers}
that the contraction map $i_{\phi}$ is defined to be  the double derivation $i_{\phi} \co \Omega^*_{A/R} \to \Omega^*_{A/R}\ten \Omega^*_{A/R}$ determined by the properties that $i_{\phi}(a)=0$ and $i_{\phi}(da)= \phi(a)$, for $a \in A$.

Writing $(\alpha \ten \beta)^{\dmd}:= \pm \beta \alpha$, the reduced contraction $\iota_{\phi} \co  \Omega^n_{A/R} \to \Omega^{n-1}_{A/R}$  is then defined in \cite[\S 2.6]{CrawleyBoeveyEtingofGinzburgNCGeomQuivers} by $\iota_{\phi}(\omega):=i_{\phi}(\omega)^{\dmd}$.
\end{definition}

As in \cite[Lemma 2.8.6]{CrawleyBoeveyEtingofGinzburgNCGeomQuivers}, the reduced contraction kills commutators, so descends to a map $\iota_{\phi} \co \Omega^n_{\cyc} \to \Omega^{n-1}$.

\begin{definition}\label{bisymplecticdef}
 If $\Omega^1_{A/R} $ is perfect as an $A^{e}$-module,  define an $n$-shifted bisymplectic structure $\omega$ on $A/R$ to be an $n$-shifted pre-bisymplectic structure $\omega$ for which the component $\omega_2 \in \z^n\Omega^2_{A/R,\cyc}$ induces a quasi-isomorphism
\[
 \omega_2^{\sharp} \co \DDer_R(A,A^{e}) \to (\Omega^1_{A/R})_{[-n]}
\]
by $\omega_2^{\sharp}(\phi)= \iota_{\phi}(\omega_2)$. 
\end{definition}

\begin{remark}
 When $n=0$ and $A$ is concentrated in degree $0$, this recovers the notion of bi-symplectic structures from \cite[Definition 4.2.5]{CrawleyBoeveyEtingofGinzburgNCGeomQuivers}, since the higher terms are then necessarily $0$, giving  $\omega=\omega_2$ with $d\omega_2=0$.
\end{remark}



\begin{definition}\label{PreSpdef}
 Define the space $\PreBiSp(A,n)= \Lim_{p\ge 2}\PreBiSp(A,n)/F^p $ of $n$-shifted pre-bisymplectic structures on $A/R$ to be the simplicial set given in degree $k$ by setting
 \[
(\PreBiSp(A,n)/F^p)_k:= \z^{n+2}((F^2\DR_{\cyc}(A/R)/F^p)\ten_{\Q} \Omega^{\bt}(\Delta^k)),
\]
where 
\[
\Omega^{\bt}(\Delta^k)=\Q[t_0, t_1, \ldots, t_k,\delta t_0, \delta t_1, \ldots, \delta t_k ]/(\sum t_i -1, \sum \delta t_i)
\]
is the commutative dg algebra of de Rham polynomial forms on the $k$-simplex, with the $t_i$ of degree $0$.
 

Let $\BiSp(A,n) \subset \PreBiSp(A,n)$ consist of the bisymplectic structures --- this is a union of path-components.
\end{definition}
Note that  $\PreBiSp(A,n)/F^p$  is canonically weakly  equivalent to the Dold--Kan denormalisation of the complex $\tau^{\le 0}(F^2\DR_{\cyc}(A)^{[n+2]}/F^p)$, where $\tau$ denotes good truncation, and similarly for the limit  $ \PreBiSp(A,n)$. However,   our definition in terms of de Rham polynomial forms will simplify the comparison with double Poisson structures.
 
In particular, the components of $\PreBiSp(A,n)$ are just elements in 
\[
 \H^{n+2}F^2\DR_{\cyc}(A/R), 
\]
while equivalence classes of $k$-morphisms in $\PreBiSp(A,n)$ are  given by elements in $\H^{n+2-k}F^2\DR_{\cyc}(A/R)$.

\begin{example}[Shifted cotangent spaces and their twists]\label{shiftedcotex}
The simplest examples of $n$-shifted pre-bisymplectic structures are given by shifted cotangent spaces of non-commutative affine spaces. Explicitly, if we start from a cofibrant  DGAA $B:=(R\<x_1, \ldots, x_r\>,\delta)$ for some differential $\delta$,  then there is an $n$-shifted cotangent space of $B$  given by $R\<x_1, \ldots, x_r, \xi_1, \ldots, \xi_r\>$ for $\xi_i$ of degree $n-\deg x_i$, with
\[
\delta \xi_j := \sum_i \pm \frac{\pd \delta x_i}{\pd x_j}'' \xi_i\frac{\pd \delta x_i}{\pd x_j}'  ,\quad \text{  where }\quad d\delta x_i= \sum_j \frac{\pd \delta x_i}{\pd x_j}'(dx_j)\frac{\pd \delta x_i}{\pd x_j}''.
\]
This then carries a $(-n)$-shifted symplectic structure explicitly given in co-ordinates by $\omega = \sum_i dx_i d\xi_i$. 

There are also twisted variants of this construction, where we start from $B$ as above, choose an element $f \in B_{\cyc}$ of degree $n-1$ with $\delta f =0$, and then modify $\delta \xi_j$ by  adding $(\frac{\pd f}{\pd x_j})_{\cyc}$, where  $df= \sum_j (\frac{\pd f}{\pd x_j})_{\cyc} dx_j \in \Omega^1_{A,\cyc}$.  We can think of this as the derived critical locus of $f$; if we forced all variables to commute, this would recover the usual derived critical locus of the image of $f$ in $B^{\comm}= (R[x_1, \ldots, x_r],\delta)$.

The simplest examples of this type take $\deg x_i=0$ for all $i$, forcing $\delta x_i=0$, and then  we can easily see that $\delta \omega =0$, because 
\[
 \delta \omega = \pm d\delta(\sum_i\xi_i dx_i )=  \pm d \sum_i (\frac{\pd f}{\pd x_i})_{\cyc} dx_i  \sim  \pm d (df)=0,
\]
where $\sim$ denotes cyclic equivalence;
this recovers the algebra $\fA(B,f)$ from \cite[\S 1.3]{ginzburgCYAlg} as its $0$th homology.  
 \end{example}
 
\begin{definition}\label{Lagdef}
Take a morphism $f \co A \to B$ of  cofibrant  $R$-DGAAs, with an  $n$-shifted pre-bisymplectic structure $\omega$ on $A/R$. We then define a bi-isotropic structure on $B$ relative to $\omega$ to be an element $(\omega, \lambda)$ of 
\[
 \z^{n+1}\cone(F^2\DR_{\cyc}(A/R)\to F^2\DR_{\cyc}(B/R)) 
\]
lifting $\omega$.

 This structure is called bi-Lagrangian if $\Omega^1_{A/R} $ and $\Omega^1_{B/R}$ are perfect as an $A^{e}$-module and a $B^{e}$-module respectively, if $\omega$ is bi-symplectic, and if contraction with the image $(\bar{\omega_2},\bar{\lambda}_2)$ of $(\omega,\lambda)$ in  $\z^{n-1}\cone(\Omega^2_{A/R,\cyc} \to \Omega^2_{B/R,\cyc} )$ 
 induces a quasi-isomorphism
\[
 (f \circ \omega_2^{\sharp}, \lambda_2^{\sharp}) \co \cone(\DDer(B,B^{e}) \to   \DDer_R(A,B^{e})) \to   (\Omega^1_{B/R})_{[-n]}.
\]  
\end{definition}

 \begin{definition}\label{Isodef}
  Given a morphism $A \to B$ of  cofibrant  $R$-DGAAs, define 
  the space $\BiIso(A,B;n)=\Lim_{p\ge 2}\BiIso(A,B;n)/F^p$ of  $n$-shifted bi-isotropic structures on the pair $(A,B)$ over $R$ to be the simplicial set given in degree $k$ by setting
\[
 (\Iso(A,B;n)/F^p)_k:= \z^{n+1}(\cone(F^2\DR_{\cyc}(A/R)/F^p  \to F^2\DR_{\cyc}(B/R)/F^p)\ten_{\Q} \Omega^{\bt}(\Delta^k)).
\]

Set $\BiLag(A,B;n) \subset \BiIso(A,B;n)$ to consist of the   bi-Lagrangians on bi-symplectic structures --- this is a union of path-components.
\end{definition}
Thus the components of $\BiIso(A,B;n)$ are just elements in 
\[
 \H^{n+1}\cone(F^2\DR_{\cyc}(A/R) \to F^2\DR_{\cyc}(B/R)), 
\]
and equivalence classes of $k$-morphisms in $\Iso(A,B;n)$ are  given by elements in $\H^{n+1-k}$ of the same complex. 

\begin{definition}\label{commdefhere}
If $\Q \subset R$, there is a natural model structure on   differential graded commutative $R$-algebras for which the inclusion functor to $dg\Alg(R)$ is right Quillen. The left adjoint functor $B \mapsto B^{\comm}$ is given by taking the quotient $B/([B,B])$ by the commutator ideal, and we denote its left-derived functor by $B \mapsto B^{\oL,\comm}$. 
\end{definition}

With the notation of \cite{poisson}, we then have:
\begin{lemma} \label{commSplemma}
 There is a natural map $\PreBiSp(A,n) \to \PreSp(A^{\oL,\comm},n)$ which sends the space $\BiSp(A,n)$ of $n$-shifted bisymplectic structures on $A$ to the space   $\Sp(A^{\oL,\comm},n)$ of $n$-shifted symplectic structures on its commutative quotient. 
 
 Similarly, there is a natural map $\BiIso(A,B;n) \to \Iso(A^{\oL,\comm},B^{\oL,\comm}, n)$ which sends the space $\BiLag(A,B;n)$ of $n$-shifted bi-Lagrangian structures on $B$ over  $A$ to the space   $\Lag(A^{\oL,\comm},B^{\oL,\comm};n)$ of $n$-shifted Lagrangian structures on their commutative quotients. 
\end{lemma}
\begin{proof}
 We may assume that $A$ is cofibrant, and then we just observe that there is a natural map from the filtered DGAA $\DR(A/R)$ to the commutative de Rham algebra of $A^{\comm}$, given by killing the commutator ideal. Since the latter contains the commutator, the map factors through  $\DR_{\cyc}(A/R)$, giving our required map $\PreBiSp(A,n) \to \PreSp(A^{\comm},n)$. In particular, the commutative cotangent module $\Omega^1_{A^{\comm}/R}$ is given by $(\Omega^1_A)\ten_{A^{e}}A^{\comm}$, so the non-degeneracy condition of Definition \ref{bisymplecticdef} induces the required quasi-isomorphism between $\Omega^1_{A^{\comm}/R}$ and its shifted $A^{\comm}$-linear dual. 
 
 The isotropic statements then follow in exactly the same way.
\end{proof}

\subsection{Shifted bisymplectic structures on derived Artin NC prestacks}\label{spartinsn} 
 
The following is \cite[Definition \ref{NCstacks-NCprestdef}]{NCstacks}:
\begin{definition}
  The model category $DG^+\Aff^{nc}(R)^{\wedge}$ of derived NC prestacks  is defined to be the category of simplicial set-valued functors on $dg_+\Alg(R)$, equipped with the left Bousfield localisation of the projective model structure at morphisms of the form $f^* \co \oR \Spec^{nc} B \to \oR\Spec^{nc} A$, for quasi-isomorphisms $A \to B$, where  $\oR\Spec^{nc} A= \map_{dg_+\Alg(R)}(A,-)$.
\end{definition}
As in \cite[\S 2.3.2]{hag1},
the homotopy category $\Ho(DG\Aff^{nc}(R)^{\wedge})$ is equivalent to the category of weak equivalence classes of weak equivalence-preserving simplicial functors on $dg_+\Alg(R)$, since fibrant objects are those prestacks which preserve weak equivalences and are objectwise fibrant.

For the purposes of this section, we will   restrict to the smaller class of derived Artin NC prestacks introduced in \cite{NCstacks}. 
By \cite[Definition \ref{NCstacks-geomdef2}]{NCstacks}, a derived  NC prestack $F$ is said to be to be  strongly quasi-compact (sqc) $N$-geometric   NC Artin 
if it arises as the geometric realisation of simplicial diagram $X_{\bt}$ of derived NC affines whose homotopy partial matching maps satisfy various submersiveness and surjectivity conditions. A filtered colimit of open morphisms of sqc $N$-geometric   NC derived Artin prestacks, for varying $N$, is then called \cite[Definition \ref{NCstacks-inftygeomdef2}]{NCstacks} an $\infty$-geometric  NC derived Artin prestack.

\begin{definition}\label{PreBiSpArtindef}
 Given an $\infty$-geometric  NC derived Artin prestack   $F \in DG^+\Aff^{nc}(R)^{\wedge}$, we define the space $\PreBiSp(F,n)$ of $n$-shifted pre-bisymplectic structures on $F$ to be the mapping space
\[
 \PreBiSp(F,n):= \map_{ DG^+\Aff^{nc}(R)^{\wedge}}(F, \PreBiSp(-,n)),
\]
 noting that the construction of Definition \ref{PreSpdef} is clearly functorial.
 
 Similarly, given a morphism $\eta \co G \to F$ of  $\infty$-geometric  NC derived Artin prestacks, we generalise Definition \ref{Isodef} to define the space  $\BiIso(F,G;n)$ of $n$-shifted bi-isotropic structures on $G$ over $F$ to be the homotopy fibre product
\[
 \BiIso(F,G;n):=\PreBiSp(F,n)\by^h_{\PreBiSp(G,n)}\{0\},
\]
so that  bi-isotropic structures on $G$ over a fixed $n$-shifted pre-symplectic structure $\omega$ on $F$ are given by the homotopy fibre $\{\eta^*\omega\}\by^h_{\PreBiSp(G,n)}\{0\}$ of $\BiIso(F,G;n)\to \PreBiSp(F,n)$ over $\omega$.
 \end{definition}

In other words, an $n$-shifted pre-bisymplectic structure on $F$ consists of compatible   $n$-shifted pre-bisymplectic structures on $A$ for all points $x \in F(A)$ and all $A$. Note in particular that this implies $ \PreBiSp(\oR\Spec^{nc} A,n) \simeq\PreBiSp(A,n)$. More generally, if $F$ is sqc, it has a simplicial resolution $F \simeq \ho\LLim X_{\bt}$ by derived NC affines, then calculating the relevant homotopy limits shows that  $\PreBiSp(F,n)$ can be described by the formulae of Definition \ref{PreSpdef}, replacing $F^p\DR_{\cyc}(A/R)$ with the product total complex $\Tot^{\Pi}(i \mapsto F^p\DR_{\cyc}(X_i/R))$.

Extending the definition of bisymplectic structures to derived  NC prestacks is not so straightforward, because the construction $A \mapsto \DDer_A(A,A^{e})$  has very limited functoriality, with knock-on effects on non-degeneracy conditions, which is why have restricted to Artin prestacks in order to make the following work. Although Definition \ref{PreBiSpArtindef} can naturally be extended to arbitrary derived  NC prestacks, beware that it can be too crude; we will introduce a subtler definition in \S \ref{hgsBiSpsn} which agrees with it in the Artin case.

\begin{definition}
 Given an $A^{e}$-module $M$ and a strictly positive integer  $p$, write $(M^{\ten p})_{\cyc}$ for the image of $M^{\ten p}$ in the cyclic quotient $T_A(M)/[T_AM,T_AM]$ of the tensor algebra.
 \end{definition}
Thus $\Omega^p_{A,\cyc}\cong((\Omega^1_A)^{\ten p})_{\cyc}$; as in the proof of \cite[Proposition 4.1.2]{vdBerghDoublePoisson}, $(M^{\ten p})_{\cyc}$ can be obtained by writing $p$ copies of $-\ten_AM\ten_A-$ in a circle, then taking the quotient by the obvious action of $C_p$.
 
The following is now an immediate consequence of the description of the cotangent complex $\bL_F$ in the proof of \cite[Lemma \ref{NCstacks-cotgood}]{NCstacks}: 
\begin{lemma}
For an $\infty$-geometric  NC derived Artin prestack $F$, the natural maps
\[
 \oR\Gamma(F, (\bL_F^{\ten p})_{\cyc}) \to  \oR\Gamma(F, \oL\Omega^p_{\sO_F,\cyc} ), 
\]
induced by the maps $x^*\bL_F \to \oL\Omega^1_A$ for $x \in F(A)$, are quasi-isomorphisms, where $\sO_F(\oR\Spec^{nc} B \to F):= B$.
 \end{lemma}

That equivalence  then gives us  a natural map $\pi_0\PreBiSp(F,n) \to \H^{n+2}\oR\Gamma(F, (\bL_F^{\ten p})_{\cyc})$ 
by sending $\omega$ to $\omega_2$, allowing us to generalise Definition \ref{bisymplecticdef} as follows:

\begin{definition}\label{bisymplecticdef2}
Given an $\infty$-geometric  NC derived Artin prestack $F$,    we say that an $n$-shifted pre-bisymplectic structure $\omega$ on $F$ is $n$-shifted bisymplectic if the cotangent complex $\bL_F$ is perfect as an $\sO_F^{\oL,e}$-module, and the induced map
\[
 \omega_2^{\sharp}\co \co \oR\HHom_{A^{\oL,e}}(\bL_{F,x} ,A^{\oL,e}) \to (\bL_{F,x} )_{[-n]}
\]
is a quasi-isomorphism for all $A$ and all $x \in F(A)$.

We then let $\BiSp(F,n) \subset \PreBiSp(F,n)$ consist of the $n$-shifted bisymplectic structures --- this is a union of path components.

Similarly, given a morphism $\eta \co G \to F$ of  $\infty$-geometric  NC derived Artin prestacks, we say that a bi-isotropic structure $\lambda$ on $G$ relative to an  $n$-shifted bisymplectic structure $\omega$ on $F$ is bi-Lagrangian if 
$\bL_G$ is perfect as an $\sO_G^{\oL,e}$-module, and the induced map
\[
 (\eta \circ \omega_2^{\sharp}, \lambda_2^{\sharp}) \co \cone( \oR\HHom_{A^{\oL,e}}(\bL_{G,x} ,A^{\oL,e})  \to \oR\HHom_{A^{\oL,e}}(\bL_{F,\eta(x)} ,A^{\oL,e})) \to   (\bL_{G,x})_{[-n]}.
\] 
is a quasi-isomorphism for all $A$ and all $x \in G(A)$.

We then let $\BiLag(F,G;n) \subset \BiIso(F,G;n)$ consist of the $n$-shifted bi-Lagrangian structures --- this is a union of path components.

\end{definition}

Note that if $F^{\comm}$ denotes the restriction of $F$ to CDGAs, then  Lemma \ref{commSplemma} combines with a comparison of commutative and non-commutative cotangent complexes to  imply that we have a natural map  $\BiSp(F,n) \to \Sp(F^{\comm, \sharp},n)$ to the space of symplectic structures from \cite{poisson}, where $(-)^{\sharp}$ denotes \'etale hypersheafification. The following proposition implies that we also have natural maps $\BiSp(F,n) \to \Sp( (\Pi_{\Mat_r}F)^{\comm ,\sharp},n)$ giving  symplectic structures on all of the representation spaces, for the Weil restriction functor $\Pi_{\Mat_r}$ of \cite[\S \ref{NCstacks-weilrestrn}]{NCstacks}.

\begin{proposition}\label{weilArtinprop}
 If  $S$ is a finite flat Frobenius $R$-algebra, then for any $\infty$-geometric  NC derived Artin prestack $F$, there is a natural map
\[
 \BiSp(F,n) \to \BiSp(\Pi_{S/R}F,n)
\]
from the space of $n$-shifted bisymplectic structures on $F$ to the space of $n$-shifted bisymplectic structures on $\Pi_{S/R}F$, where $\Pi_{S/R}F(B):=F(B\ten_RS)$.

For any morphism $G \to F$ of $\infty$-geometric  NC derived Artin prestacks, there is similarly a natural map
\[
 \BiLag(F,G;n) \to \BiLag(\Pi_{S/R}F,\Pi_{S/R}G;n)
\]
from the space of $n$-shifted bi-Lagrangian structures on $G$ over $F$ to the space of $n$-shifted bi-Lagrangian structures on $\Pi_{S/R}G$ over $\Pi_{S/R}F$.
 \end{proposition}
\begin{proof}
We will prove the first statement; the proof of the second is very similar.

By \cite[Lemma \ref{NCstacks-stweillemma2}]{NCstacks}, we know that $\Pi_{S/R}F$ is $\infty$-geometric, so the space of bisymplectic structures thereon is well-defined. For any $B \in dg_+\Alg(R)$, there is a natural $(B\ten_RS)^{\oL,e}$-linear map
\[
 \Omega^1_{(B\ten_RS)/R}\to \Omega^1_{B/R}\ten_RS,
\]
corresponding to the derivation $d_B \ten \id_S$ on $B\ten_RS$. Combined with the multiplication on $S$, this gives morphisms $\Omega^p_{(B\ten_RS)/R}\to \Omega^p_{B/R}\ten_RS $ which combine with the trace $\tr \co S \to R$ to give  morphisms $\Omega^p_{(B\ten_RS)/R}\to \Omega^p_{B/R}$. Since the trace of a Frobenius algebra kills commutators, this descends to a filtered morphism $\DR_{\cyc}(B\ten_RS  /R)\to \DR_{\cyc}(B/R)$.

Applying this functorially to elements $x \in  \Pi_{S/R}F(B)$, this gives us a map $\rho \co \PreBiSp(F,n) \to \PreBiSp(\Pi_{S/R}F,n)$ of shifted pre-bisymplectic structures, and we need only show that it preserves the bisymplectic structures. Writing $S^*:=\Hom_R(S,R)$, We begin by observing that the cotangent complex of  $\Pi_{S/R}F$ is given at $x$ by 
\[
 \bL_{\Pi_{S/R}F,x} \simeq \bL_{F,x}\ten^{\oL}_{S^{e}}S^*. 
\]
since for any $B^{\oL,e}$-module $M$, we have
\begin{align*}
 \oR\HHom_{B^{\oL,e}}(\bL_{\Pi_{S/R}F,x},M) &\simeq \oR\HHom_{B^{\oL,e}\ten_R S^{e}}(\bL_{F,x},M\ten_RS)\\
 &\simeq \oR\HHom_{B^{\oL,e}\ten_R S^{e}}(\bL_{F,x},\HHom_R(S^*,M))\\
 &\simeq \oR\HHom_{B^{\oL,e}}(\bL_{F,x}\ten^{\oL}_{S^{e}}S^*,M).
\end{align*}

A shifted bisymplectic structure $\omega$  on $F$ then gives a quasi-isomorphism $\omega_2^{\sharp} \co \bL_{F,x} \to \oR\HHom_{B^{\oL,e}\ten_R S^{e}}(\bL_{F,x}, B^{\oL,e}\ten_R S^{e})$, and hence
\[
\bL_{\Pi_{S/R}F,x}\to \oR\HHom_{B^{\oL,e}\ten_R S^{e}}(\bL_{F,x}, B^{\oL,e}\ten_R S^*)\simeq \oR\HHom_{B^{\oL,e}}(\bL_{F,x}\ten^{\oL}_{S^{e}}S, B^{\oL,e}).
\]
The map $\rho(\omega)_2^{\sharp}$ then composes this with the isomorphism $S^* \to S$ induced by the trace ($S$ being Frobenius) 
to give the required  quasi-isomorphism
\[
 \bL_{\Pi_{S/R}F,x}\to \oR\HHom_{B^{\oL,e}\ten_R S^{e}}(\bL_{F,x}, B^{\oL,e}\ten_R S)\simeq \oR\HHom_{B^{\oL,e}}(\bL_{\Pi_{S/R}F,x}, B^{\oL,e}).\qedhere
\]
\end{proof}

\subsection{The \tps{$2$}{2}-shifted bisymplectic structure on \tps{$\Perf$}{Perf}\ }

As in \cite[\S \ref{NCstacks-dmodperfsn}]{NCstacks}, an example of an $\infty$-geometric derived Artin  NC prestack is given by the moduli prestack $\Perf\co dg_+\Alg(R) \to s\Set$ of perfect complexes, which sends an $R$-DGAA $B$ to the nerve of the $\infty$-category of perfect right $B$-modules in complexes.

\begin{definition}\label{HCdef}
For $B \in dg_+\Alg(R)$ flat over $R$, we denote the Hochschild differential by $\mathsf{b} \co \Omega^p_{B/R} \to \Omega^{p-1}_{B/R}$. As in \cite[2.1.2]{GinzburgSchedlerNewCyclic}, this can be written as $\mathsf{b}(\alpha da)= (-1)^p [\alpha,a]$, for $\alpha \in \Omega^{p-1}_{B/R} $ and $\alpha \in B$.  We also denote Connes' differential by $\mathsf{B} \co \Omega^p_{B/R} \to \Omega^{p+1}_{B/R}$. 

For $u$ of cohomological degree $-2$, 
  we then write  $\mathbf{HN}_{R,\bt}(B)$ for the complex 
\[
 (\prod_{p \in \Z} u^{-p}F^p\DR(B/R), \mathsf{B}\pm u^{-1}\mathsf{b}\pm \delta) = ((\bigoplus_{p \ge 0} u^{-p}\Omega^p_{B/R})[-p]\llb u \rrb, \mathsf{B}\pm u^{-1}\mathsf{b}\pm \delta), 
\]
where we set $F^p\DR:=\DR$ for $p < 0$.
Adapting \cite[Application 9.8.4]{W} to dg algebras, this is Connes' construction of the complex calculating negative cyclic homology $\HN_*(B)$ of $B$ over $R$, once we note that $( u^{-*}\Omega^*_{B/R}[-*], u^{-1}\mathsf{b}\pm \delta)$ is the normalised Hochschild complex $\bar{\CC}_{R,\bt}(B)$.
%
%
\end{definition}

Note that on taking the quotient by $u$, we have $\mathbf{HN}_{R,\bt}(B)/u\mathbf{HN}_{R,\bt}(B) \cong (\prod_{p \ge 0} u^{-p}\Omega^p_{B/R}, u^{-1}\mathsf{b}\pm \delta)$, which is the normalised Hochschild homology complex $\bar{\CC}_{R,\bt}(B)$, 
the non-negativity of the grading on   $B$ meaning that this infinite product is just a direct sum. The correct generalisation of Definition \ref{HCdef} to arbitrary   $B \in dg\Alg(R)$ involves a more subtle mixture of direct sums and products.

\begin{lemma}\label{goodwillielemma}
Up to coherent homotopy,  the Goodwillie--Jones Chern character gives  a natural map 
\[
\ch^-\co K(B)_{\Q} \to \mathbf{HN}_{R,\bt}(B)
\]
lifting the Dennis trace,
where $K(-)_{\Q}$ denotes the $\Q$-linearisation of the $K$-theory spectrum.
 \end{lemma}
\begin{proof}
Via the Dold--Kan and Eilenberg--Zilber correspondences, we may replace $dg_+\Alg$ with the category of simplicial $R$-algebras, and the desired map is then the Goodwillie--Jones Chern character of \cite{goodwillieChern}.
\end{proof}

Combined with Lemma \ref{goodwillielemma}, the following gives a refinement of the Karoubi Chern character of \cite[\S\S1.17--1.24]{karoubiHomologieCyclique}: 
\begin{lemma}\label{KaroubiChernLemma}
Passing to the quotient $\DR_{\cyc}$ of $\DR$ gives a natural map
\[
\xi \co  \mathbf{HN}_{R,\bt}(B) \to \prod_{p \in \Z} F^p\DR_{\cyc}(B/R)[2p].
\]
\end{lemma}
\begin{proof}
Since $\mathsf{b}$ is defined as a commutator, it vanishes on $\DR_{\cyc}$. Moreover, $\mathsf{B}$ descends to a map $ \Omega^p_{\cyc}(B/R) \to \Omega^{p+1}_{\cyc}(B/R)$ on the cyclic quotient, where it equals $(p+1)d$.  Thus the differential $\mathsf{B}+u^{-1}\mathsf{b}\pm \delta$ becomes $(p+1)d \pm \delta$ on $\Omega^p_{\cyc}(B/R)$, giving us the required map
 $ \mathbf{HN}_{R,\bt}(B)\to \prod_{p \in \Z} u^{-p}F^p\DR_{\cyc}(B/R)$ if we divide $\Omega^p_{B/R}$ by $(p!)$.  
\end{proof}

\begin{proposition}\label{HNsympprop1}
 There is a natural map from the Dold--Kan denormalisation of the good truncation of $ \mathbf{HN}_{R,\bt}(A)_{[d]}$ to the space of $(2-d)$-shifted pre-bisymplectic structures on $A$, and hence from $\HN_d(A)$ to equivalence classes of $(2-d)$-shifted pre-bisymplectic structures on $\oR \Spec^{nc} A$. 
 
The pre-bisymplectic structure thus associated to  an element $\alpha \in \HN_d(A)$ is bisymplectic if and only if its image $\bar{\alpha} \in \HH_d(A) = \H_d(A\ten^{\oL}_{A^{\oL,e}}A)$ in Hochschild homology is non-degenerate in the sense that
the composite map
 \[
\DDer_R(A,A^{\oL,e})\xra{u} \oR\HHom_{A^{\oL,e}}(A,A^{\oL,e})[1] \xra{\ev_{\bar{\alpha}}}  A[1-d] \xra{u^*} \Omega^1_A[2-d], 
\]
is a quasi-isomorphism, where the  maps $u, u^*$ come from the resolution $\cone(\Omega^1_A \to A^{\oL,e})$ of $A$ as a bimodule, and the middle map is given by allowing $\oR\HHom_{A^{\oL,e}}(A,A^{\oL,e})$ to act on the first factor of  $\bar{\alpha} \in \H_d(A\ten^{\oL}_{A^{\oL,e}}A)$.
\end{proposition}
\begin{proof}
 The map from the negative cyclic homology space to the space of shifted pre-bisymplectic structures is given by the map $\xi_2$ of Lemma \ref{KaroubiChernLemma}, which  projects $\xi$ onto the $p=2$ factor. 
 
 It follows by construction that the composite map
 \[
   \mathbf{HN}_{R,\bt}(A)\to u^{-p}F^p \DR_{\cyc}(A/R) \to u^{-p}\Omega^p_{\cyc}(A/R)
 \]
factors through the Hochschild homology complex $ \mathbf{HN}_{R,\bt}(A)/u \mathbf{HN}_{R,\bt}(A)$, so non-degeneracy of $\alpha$ depends only on its image $\bar{\alpha} \in \HH_d(A)$. The description of the composite above then follows from an analysis of the contraction map, with closure of $\alpha$ under the Hochschild differential ensuring that $ \iota_{\phi}(\bar{\alpha})=0$ for any $\phi$ in the image of the natural map $A^{\oL,e} \to \DDer_R(A,A^{\oL,e})$. 
 \end{proof}

 \begin{remark}\label{EHCrmk}
  By functoriality, Proposition \ref{HNsympprop1} gives a morphism $\oR\Gamma(F, \mathbf{HN}_{\bt}(\sO_F)[-d]) \to \PreBiSp(F,2-d)$ for any derived NC prestack $F$, with the non-degenerate condition then phrased by replacing $\oL\Omega^1_A$ with the cotangent complex $L_{F,x}$ of $F$, for each $x \in F(A)$. 
  
  For instance, the adjoint action of units $B^{\by}$ in $B$ on $\Hom_{dg_+\Alg(R)}(A,B)$ gives rise to a derived  NC prestack $\oR[\Spec^{nc} A/\bG_m] $ as the right-derived functor, and then for the canonical map $u \co [\oR\Spec^{nc} A \to \oR\Spec^{nc} A/\oR \bG_m] $, we have $\bL_{[\oR\Spec^{nc} A/\oR \bG_m],u} \simeq A_{[1]}$, regarded as the homotopy fibre of $\oL\Omega^1_A \to A^{\oL,e}$, so the non-degeneracy condition in this case would be a quasi-isomorphism between $\oR\HHom_{A^{\oL,e}}(A,A^{\oL,e})[d-2]$ and $A$. The relevant negative cyclic homology complex  $\oR\Gamma( [\oR\Spec^{nc} A/\oR \bG_m], \mathbf{HN}_{\bt}(\sO_{[\oR\Spec^{nc} A/\oR \bG_m]}))$ is an integrated form of the negative extended cyclic homology of \cite{GinzburgSchedlerNewCyclic}, with variables $x^{\pm}$ in place of a formal variable $t=\log x$.
 \end{remark}

 \begin{example}\label{exlocsys}
  Given a reduced simplicial set $(X,x)$, we can form the simplicial loop group $G(X,x)$ as in \cite{loopgp}, linearise to give a simplicial algebra $R.G(X,x)$, and then apply Dold--Kan normalisation and the Eilenberg--Zilber correspondence to give $NR.G(X,x) \in dg_+\Alg(R)$. Then $\oR\Spec^{nc} NR.G(X,x)$ is the derived moduli functor of framed rank $1$ local systems on $(X,x)$, sending $B \in dg_+\Alg(R)$ to the $\infty$-groupoid of locally constant $B$-modules $M$ over $X$ with framings $M_x \simeq B$. Similarly, $\oR[\Spec^{nc} NR.G(X,x)/\bG_m]$ is a derived moduli functor of unframed rank $1$ local systems, sending $B$ to the $\infty$-groupoid of those locally constant $B$-modules $M$ over $X$ which are locally isomorphic to $B$.
  
 Reasoning as in \cite[\S 5.1]{BravDyckerhoff}, a non-degenerate negative cyclic homology class of degree $d$ for the derived NC prestack $\oR[\Spec^{nc} NR.G(X,x)/\bG_m]$  then arises from a class $[X] \in \H_d(X)$ inducing a Poincar\'e duality isomorphism $\H^*(X,R) \cong \H_{d-*}(X,R)$, thus giving us a  $(2-d)$-shifted pre-bisymplectic structure on derived NC moduli of rank $1$ local systems.
   \end{example}

The following is a non-commutative analogue of \cite[Theorem 0.3]{PTVV}, though the bulk of the proof, establishing the non-degeneracy condition, differs substantially because the formula for the Chern character asserted in the proof of  \cite[Theorem 0.3]{PTVV} does not make sense in non-commutative cases.
\begin{theorem}\label{Perfthm1}
 There is a canonical $2$-shifted bisymplectic structure on the derived Artin NC prestack $\Perf$.
 \end{theorem}
\begin{proof}
We have a natural map $\Perf(B) \to K(B)$ from  perfect complexes to $K$-theory. Composing it with the Karoubi and Goodwillie--Jones Chern characters of 
 Lemma \ref{KaroubiChernLemma} and Proposition \ref{HNsympprop1} then gives us  a natural map
\[
 \xi_2 \circ \ch^-\co \Perf(B) \to \PreBiSp(B,2)
\]
 to the space of  $2$-shifted pre-bisymplectic structures. Since this is functorial, it defines an element $\omega$ of   $\PreBiSp(\Perf,2)$ in the sense of Definition \ref{PreBiSpArtindef}.

It therefore remains only to show that this pre-bisymplectic structure is bisymplectic. If we take  a point $[E] \in \Perf(B)$ corresponding to a perfect complex $E$ of right $B$-modules, then as in \cite[Remark \ref{NCstacks-Perfnicermk}]{NCstacks}, the  tangent complex is given by
\[
 \oR\HHom_{B^{\oL,e}}(\bL_{\Perf,[E]},B^{\oL,e})\simeq\oR\HHom_B(E, E\ten_RB)[1]\simeq E\ten_R^{\oL}E^*[1],
\]
 where $B$ acts on both terms on the right and $E^*=\oR\HHom_B(E,B)$. Thus, $E$ being perfect, $\bL_{\Perf,[E]}$ is the predual $E^*\ten_R^{\oL}E[-1]$,  so there is  a natural quasi-isomorphism  $\sigma \co \oR\HHom_{B^{\oL,e}}(\bL_{\Perf,[E]},B^{\oL,e})\to \bL_{\Perf,[E]}[2]$.

The map $[E] \co \oR \Spec^{nc} B \to \Perf$ induces a natural map 
\[
 \At_E \co \oR\HHom_{B^{\oL,e}}(\oL\Omega^1_B,B^{\oL,e}) \to \oR\HHom_{B^{\oL,e}}(\bL_{\Perf,[E]},B^{\oL,e})
\]
  on tangent complexes, which we think of as the non-commutative Atiyah class, and this has a corresponding  adjoint map $\At^*_E \co\bL_{\Perf,[E]} \to  \oL\Omega^1_B$. In order to show that $\omega$ is bisymplectic, it suffices to show that 
\[
 \omega_2^{\sharp}\co  \oR\HHom_{B^{\oL,e}}(\oL\Omega^1_B,B^{\oL,e}) \to \oL\Omega^1_B[2]
\]
is naturally homotopic to the composition $\At^*_E \circ \sigma \circ \At_E$, since $\sigma$ is a quasi-isomorphism.

We next observe that $\At_E$ factorises as the composite
\[
 \oR\HHom_{B^{\oL,e}}(\oL\Omega^1_B,B^{\oL,e}) \xra{u} \oR\HHom_{B^{\oL,e}}(B,B^{\oL,e})[1] \xra{i_E} 
 \oR\HHom_B(E, E\ten_RB)[1],
\]
for $u$ as in Proposition \ref{HNsympprop1} and $i_E$ given by $E\ten_B-$. Similarly, we have $\At_E^* \simeq u^*\circ i_E^*$, where $i_E^* \co E^*\ten_R^{\oL}E[-1] \to B[-1]$ is given by evaluating $E^*$ on $E$. Thus
\[
 \At^*_E \circ \sigma \circ \At_E \simeq u^*\circ (i_E^* \circ \sigma \circ i_E) \circ u.
\]

 Since the image  $\bar{\omega}_2$ of $\omega_2$ in Hochschild homology is just given by the Dennis trace of $E$, it follows from 
\cite[\S 4.3]{BresslerNestTsygan} that   $\bar{\omega}_2=\cL_{B/R}(\id_E)$ for  the \emph{Lefschetz map} $\cL_{B/R}$,  given by the composite
\[
 \oR\HHom_B(E, E) \simeq (E\ten^{\oL}_R E^*)\ten^{\oL}_{B^{\oL,e}} B \xra{i_E^* \ten \id} B \ten^{\oL}_{B^{\oL,e}} B.     
\]         
Thus $\ev_{\bar{\omega}} \co \oR\HHom_{B^{\oL,e}}(B,B^{\oL,e}) \to B$ is given by $i_E^* \circ \sigma \circ i_E$, so 
$ 
\omega_2^{\sharp} \simeq \At^*_E \circ \sigma \circ \At_E
$, 
as required.
\end{proof}

\subsection{Shifted bisymplectic structures on derived NC prestacks with good obstruction theory} \label{spcotsn}

Proving that a given derived NC prestack is Artin is far harder than in the commutative setting, because of the lack of descent, so we now show how to extend the formulation of bi-symplectic structures to any derived NC prestacks $F$  with well-behaved obstruction theory and cotangent complexes, so as to  include all natural derived NC moduli functors. This involves a small \'etale site of stacky derived NC affines over $F$, and is formulated in such a way as to be consistent with \S \ref{spartinsn}.

We begin by recalling some definitions from \cite[\S \ref{NCstacks-stackysn}]{NCstacks}

\begin{definition}
A stacky DGAA over a CDGA $R_{\bt}$ is  an associative  $R$-algebra $A^{\bt}_{\bt}$ in  chain cochain complexes. We write $DGdg\Alg(R)$ for the category of  stacky DGAAs over $R$, and $DG^+dg\Alg(R)$ (resp. $DG^+dg_+\Alg(R)$) for the full subcategory consisting of objects $A$ concentrated in non-negative cochain degrees (resp. non-negative bidegrees).
\end{definition}
When working with  cochain chain complexes $V^{\bt}_{\bt}$, recall that we  usually denote the chain differential by $\delta \co V^i_j \to V^i_{j-1}$, and the cochain differential by $\pd \co V^i_j \to V^{i+1}_j$.

\begin{example}\label{BGex1}
As in \cite[Example \ref{NCstacks-DstarBGex}]{NCstacks}, stacky DGAAs can arise as non-commutative analogues of Lie algebroids. For instance, for $B \in dg_+\Alg(R)$,  completing the prestack $\oR [\Spec^{nc} B/\bG_m] $ of Remark \ref{EHCrmk} along $\oR \Spec^{nc} B$ corresponds to looking at the stacky DGAA $O([\Spec^{nc} B/\g_m]) := (B\<s\>, \pd)$, where $s$ has cochain degree $1$, with cochain differential $\pd b = sb-bs$ for $b \in B$, and $\pd s = s^2$. 

In general, there is a similar bar construction $[\Spec^{nc} B/\g]$  whenever a finite  non-unital associative algebra $\g$ acts on $A$ via double derivations in a suitable way. In the example above, $\g_m$ is just the base ring $R$, whereas the infinitesimal form of $\bG_a$ in the non-commutative context is the $R$-module $R$ with zero multiplication.
\end{example}

We say that a morphism  $A \to B$ of stacky DGAAs is a levelwise quasi-isomorphism if the maps $A^i \to B^i$ are quasi-isomorphisms of chain complexes for all $i \in \Z$. 

There is a Dold--Kan denormalisation functor $D$ from non-negatively graded DGAAs to cosimplicial associative algebras; it has a left adjoint, which we denote by $D^*$.  For most practical purposes, the functor $D^*$ can be understood by remembering that it sends the tensor algebra on a cosimplicial space $V$ to the tensor algebra on the cosimplicial normalisation $N_cV$. 

\begin{definition}
 Given a chain cochain complex $V$, define the cochain complex $\hatTot V \subset \Tot^{\Pi}V$ as a subcomplex of the product total complex by
\[
(\hatTot V)^m := (\bigoplus_{i < 0} V^i_{i-m}) \oplus (\prod_{i\ge 0}   V^i_{i-m})
\]
with differential $\pd \pm \delta$. 
\end{definition}
The key property of the semi-infinite total complex $\hatTot$ is that it sends levelwise quasi-isomorphisms in the chain direction to quasi-isomorphisms; the same is not true in general of the sum and product total complexes $\Tot, \Tot^{\Pi}$, cf. \cite[\S 5.6]{W}.
The functor $\hatTot$ is referred to as Tate realisation in \cite{CPTVV}.

\begin{definition}\label{hatHomdef}
 Given a stacky DGAA $A$ and $A$-modules $M,N$ in chain cochain complexes, we define the chain cochain complex 
$\cHom_A(M,N)$  by 
\[
 \cHom_A(M,N)^i_j=  \Hom_{A^{\#}_{\#}}(M^{\#}_{\#},N^{\#[i]}_{\#[j]}),
\]
with differentials  $\pd f:= \pd_N \circ f \pm f \circ \pd_M$ and  $\delta f:= \delta_N \circ f \pm f \circ \delta_M$,
where $V^{\#}_{\#}$ denotes the bigraded vector space underlying a chain cochain complex $V$. 

We then define the  $\Hom$ complex $\hatHHom_A(M,N)$ by
\[
 \hatHHom_A(M,N):= \hatTot \cHom_A(M,N).
\]
\end{definition}
Note that $\hatTot$ is lax monoidal with respect to tensor products, which means in particular that   there is a multiplication $\hatHHom_A(M,N)\ten_R \hatHHom_A(N,P)\to \hatHHom_A(M,P)$  (the same is not true for $\Tot^{\Pi} \cHom_A(M,N)$ in general).

\subsubsection{Structures on stacky affines}\label{stackyDRsn}

Given a stacky $R$-DGAA $A$, there is an $A$-bimodule $\Omega^1_A:= \ker(A\ten_RA \to A)$ exactly as for DGAAs, except that we now have double complexes instead of chain complexes. There are then triple complexes $\Omega^{\bt}_A$ and $\Omega^{\bt}_{\cyc}$ defined exactly as in \S \ref{spaffsn}, but with an additional grading, and we now adapt the definition of bisymplectic structures similarly.

For the purposes of this section, we now fix a stacky DGAA $A \in DG^+dg\Alg(R)$ which is  cofibrant in the model structure of \cite[Lemma \ref{NCstacks-bidgaamodel}]{NCstacks}, though this can be relaxed  along similar lines  to  Remark \ref{weakerassumptionrmk}.


%

\begin{definition}\label{biDRdef}
Define the de Rham complex $\DR(A/R)$ and   cyclic de Rham complex $\DR_{\cyc}(A/R)$ to be the product total cochain complexes of the double cochain complexes 
\begin{align*}
& \Tot^{\Pi}  A \xra{d} \Tot^{\Pi} \Omega^1_{A/R} \xra{d} \Tot^{\Pi} \Omega^2_{A/R}\xra{d} \ldots,\\
 & \Tot^{\Pi}(A/[A,A]) \xra{d} \Tot^{\Pi}\Omega^1_{A/R,\cyc} \xra{d} \Tot^{\Pi}\Omega^2_{A/R,\cyc}\xra{d} \ldots,
\end{align*}
so the total differential is $d \pm \delta \pm \pd$.

We define the Hodge filtration $F$ on  $\DR(A/R)$ (resp. $\DR_{\cyc}(A/R)$) by setting    $F^p\DR(A/R) \subset \DR(A/R)$ (resp. $F^p\DR_{\cyc}(A/R) \subset \DR_{\cyc}(A/R)$) to consist of terms $\Omega^i_{A/R}$ (resp. $\Omega^i_{A/R,\cyc}$) with $i \ge p$.
\end{definition}

We use the notation $\sigma^{\le q}$ to denote the brutal cotruncation in the cochain direction
\[
 (\sigma^{\le q}M)^i := \begin{cases} 
                         M^i & i \ge q, \\ 0 & i<q.
                        \end{cases}
\]
\begin{definition}\label{bisymplecticdefstacky}
Define an $n$-shifted pre-bisymplectic structure $\omega$ on $A/R$ to be an element
\[
 \omega \in \z^{n+2}F^2\DR_{\cyc}(A/R).
\]

 If the chain complexes    $ (\Omega^1_{A/R}\ten_{A^{\oL,e}}(A^0)^{\oL,e})^i$ are acyclic for all $i >q$ and $\Tot \sigma^{\le q} (\Omega^1_{A/R}\ten_{A^{\oL,e}}(A^0)^{\oL,e})$ is quasi-isomorphic  to a  perfect  $(A^0)^{\oL,e}$-module, then we say that an $n$-shifted pre-bisymplectic structure $\omega$ is bisymplectic if contraction with $\omega_2 \in \z^n\Tot^{\Pi}\Omega^2_{A/R,\cyc}$
 induces a quasi-isomorphism
\[
\omega_2^{\sharp}\co \hatHHom_{A^{\oL,e}}(\Omega^1_A, (A^0)^{\oL,e})[-n]\to  \Tot^{\Pi} (\Omega_{A/R}^1\ten_{A^{\oL,e}}(A^0)^{\oL,e}) _{[-n]}. 
\]
\end{definition}

\begin{example}\label{BGex2}
 As in Example \ref{BGex1}, the non-commutative analogue of a Lie algebra is a finite flat  non-unital associative $R$-algebra $\g$. We then think of the bar construction of $\g$ as being functions $O(B\g)$ on the nerve of $\g$, so  $O(B\g)^n :=(\g^*)^{\ten n}$, with the obvious multiplication and with differential $\pd \co \g^* \to \g^*\ten \g^*$ dual to the multiplication.
 
Then $\DR(O(B\g))$ is a tensor algebra generated by copies of $\g^*$ in cochain degrees $1$ and $2$, so if $R=\H_0R$ then $\Tot^{\Pi} \Omega^p_{O(B\g)} $ is concentrated in degrees $\ge 2p$, meaning that a $2$-shifted pre-bisymplectic structure is just an element $\omega \in \Omega^2_{O(B\g),\cyc}$ of degree $2$ with $d\omega =\pd \omega =0$.  Expanding out the definitions, this is an element  $\omega \in \Symm^2\g^*$, with the property that $[v, \omega]=0 \in (\g^*)^{\ten 2}$ for all $v \in \g$, where
\[
 [v, a\ten b]:= (v\cdot a)\ten b - a\ten (b\cdot v) +(v\cdot b)\ten a - b\ten(a\cdot v),
\]
for the natural left and right actions of $\g$ on $\g^*$ induced by multiplication in $\g$.

In other words, a  pre-bisymplectic structure on $B\g$ is a symmetric pairing $\<-,-\>\co \g \ten \g \to R$ which has cyclic symmetry in the sense that $\<ab,c\>= \<b,ca\>$.

Looking at tangent and cotangent complexes, it then follows that  $2$-shifted pre-bisymplectic structure $\omega$ as above is bisymplectic if and only if the pairing is non-degenerate in the sense that it induces an isomorphism $\g \to \g^*$, i.e. if $\g$ is a non-unital Frobenius algebra.

\end{example}

\begin{definition}\label{PreSpdef2}
 Define the space $\PreBiSp(A,n)= \Lim_{p\ge 2}\PreBiSp(A,n)/F^p $ of $n$-shifted pre-bisymplectic structures on $A/R$ to be the simplicial set given in degree $k$ by setting
 \[
(\PreBiSp(A,n)/F^p)_k:= \z^{n+2}((F^2\DR_{\cyc}(A/R)/F^p)\ten_{\Q} \Omega^{\bt}(\Delta^k)),
\]
and let $\BiSp(A,n) \subset \PreBiSp(A,n)$ to consist of the bisymplectic structures --- this is a union of path-components.
\end{definition}

\begin{definition}\label{Lagdef2}
Take a morphism $f \co A \to B$ of  cofibrant stacky  $R$-DGAAs, with an  $n$-shifted pre-bisymplectic structure $\omega$ on $A/R$. We then define an bi-isotropic structure on $B$ relative to $\omega$ to be an element $(\omega, \lambda)$ of 
\[
 \z^{n+1}\cone(F^2\DR_{\cyc}(A/R)\to F^2\DR_{\cyc}(B/R)) 
\]
lifting $\omega$.

If the chain complexes    $ (\Omega^1_{A/R}\ten_{A^{\oL,e}}(A^0)^{\oL,e})^i$  and  $ (\Omega^1_{B/R}\ten_{B^{\oL,e}}(B^0)^{\oL,e})^i$  are acyclic for all $i >q$, with  $\Tot \sigma^{\le q} (\Omega^1_{A/R}\ten_{A^{\oL,e}}(A^0)^{\oL,e})$ and $\Tot \sigma^{\le q} (\Omega^1_{B/R}\ten_{B^{\oL,e}}(B^0)^{\oL,e})$ quasi-isomorphic  to  a perfect  $(A^0)^{\oL,e}$-module and a perfect $(B^0)^{\oL,e}$-module, respectively,  with $\omega$ bi-symplectic, then we say that $\lambda$ is bi-Lagrangian if contraction induces a quasi-isomorphism
\begin{align*}
 (f \circ \omega_2^{\sharp}, \lambda_2^{\sharp}) \co &\cone(\hatHHom_{B^{\oL,e}}(\Omega^1_B,(B^0)^{\oL,e}) 
 \to   \hatHHom_{A^{\oL,e}}(\Omega^1_A,(B^0)^{\oL,e}))\\ &\to   \Tot^{\Pi}(\Omega^1_{B/R}\ten_{B^{\oL,e}}(B^0)^{\oL,e})_{[-n]}.
\end{align*}

\end{definition}

 \begin{definition}\label{Isodef2}
  Given a morphism $A \to B$ of  cofibrant stacky  $R$-DGAAs, define 
  the space $\BiIso(A,B;n)=\Lim_{p\ge 2}\BiIso(A,B;n)/F^p$ of  $n$-shifted bi-isotropic structures on the pair $(A,B)$ over $R$ to be the simplicial set given in degree $k$ by setting
\[
 (\Iso(A,B;n)/F^p)_k:= \z^{n+1}(\cone(F^2\DR_{\cyc}(A/R)/F^p  \to F^2\DR_{\cyc}(B/R)/F^p)\ten_{\Q} \Omega^{\bt}(\Delta^k)).
\]

Set $\BiLag(A,B;n) \subset \BiIso(A,B;n)$ to consist of the   bi-Lagrangians on bi-symplectic structures --- this is a union of path-components.
\end{definition}

\subsubsection{Structures on derived NC prestacks}\label{hgsBiSpsn}

The following is \cite[Definition \ref{NCstacks-Dlowerdef}]{NCstacks}:
\begin{definition}
 Given a functor $F:dg_+\Alg(R) \to s\Set$, we define a functor $D_*F$ on  $DG^+dg_+\Alg(R)$ as the homotopy limit
\[
 D_*F(B):= \ho\Lim_{n \in \Delta} F(D^nB),
\]
for the cosimplicial denormalisation functor $D\co DG^+dg_+\Alg(R) \to dg_+\Alg(R)^{\Delta} $  (cf. \cite[Definition \ref{ddt1-nabla}]{ddt1}).
\end{definition}

For instance, \cite[Example \ref{NCstacks-DstarPerf}]{NCstacks} shows that  $D_*\Perf(B)$ is equivalent to the space of homotopy-Cartesian right $B$-modules $P$ in double complexes  for which $P^0$ is perfect over $B^0$, with equivalences defined levelwise in the chain direction. 

%
%
%


The following is \cite[Definition \ref{NCstacks-Frigdef}]{NCstacks}:
\begin{definition}\label{Frigdefhere}
 Given a stacky DGAA $B \in DG^+dg_+\Alg(R)$  for which the chain complexes
$
( \oL\Omega^1_B\ten^{\oL}_{B^{\oL,e}}(B^0)^{\oL,e})^i
$
are acyclic for all $i > q$, and  a  functor $F\co dg_+\Alg(R) \to s\Set$   which is homogeneous with a cotangent complex $L_F(B^0,x)$ (in the sense of \cite[Definition \ref{NCstacks-Fcotdef}]{NCstacks}) at a point $x \in F(B^0)$, we say that a point $y \in D_*F(B)$ lifting $x \in F(B^0)$ is \emph{rigid} if the induced morphism
 \[
  L_F(B^0,x)\to \Tot \sigma^{\le q} \oL\Omega^1_B\ten^{\oL}_{B^{\oL,e}}(B^0)^{\oL,e}
 \]
is a quasi-isomorphism of $B^0$-bimodules. We denote by $(D_*F)_{\rig}(B) \subset D_*F(B)$ the space of rigid points (a union of path components).
\end{definition}

The following is \cite[Definition \ref{NCstacks-hetdef}]{NCstacks}:
\begin{definition}
 A morphism  $A \to B$ in $DG^+dg\Alg(R)$ is said to be  homotopy \'etale when the maps 
\[
  (\oL\Omega_{A}^1\ten_{A^{\oL,e}}^{\oL}(B^{\oL,e})^0)^i \to (\oL\Omega_{B}^1\ten_{B^{\oL,e}}^{\oL}{B^{\oL,e}}^0)^i
\]
are quasi-isomorphisms for all $i \gg 0$, and 
\[
\Tot \sigma^{\le q} (\oL\Omega_{A}^1\ten_{A^{\oL,e}}^{\oL}(B^{\oL,e})^0) \to \{\Tot \sigma^{\le q}(\oL\Omega_{B}^1\ten_{B^{\oL,e}}^{\oL}{B^{\oL,e}}^0)
\]
is a quasi-isomorphism  for all $q \gg 0$, where $\sigma^{\le q}$ denotes the brutal cotruncation
\[
 (\sigma^{\le q}M)^i := \begin{cases} 
                         M^i & i \le q, \\ 0 & i>q.
                        \end{cases}
\]
\end{definition}

\begin{definition}
Define the $\infty$-category $\oL dg^+DG^+\Aff(R)$ by localising $DG^+dg_+\Alg(R)^{\op}$ at levelwise weak equivalences, and let $\oL dg^+DG^+\Aff(R)^{\et}$ be the full $2$-sub-$\infty$-category of $ \oL dg^+DG^+\Aff(R) $ with the same objects but only spaces of homotopy  \'etale morphisms (so $\map_{ \oL dg^+DG^+\Aff(R)^{\et}}(A,B) \subset \map_{ \oL dg^+DG^+\Aff(R)}(A,B)$ is a union of path components).
  \end{definition}


\begin{definition}\label{PreBiSphgsdef}
Given an derived NC prestack   $F \in DG^+\Aff^{nc}(R)^{\wedge}$ which is homogeneous and has bounded below cotangent complexes $L_F(B,x)$ (in the sense of \cite[Definition \ref{NCstacks-Fcotdef}]{NCstacks}) at all points $x \in F(B)$, we define 
 the spaces $\PreBiSp(F,n)$ and $\BiSp(F,n)$ of $n$-shifted pre-bisymplectic and bisymplectic structures on $F$ by
\begin{align*}
 \PreBiSp(F,n)&:= \map_{ \oL (dg^+DG^+\Aff(R))^{\wedge}}(D_*F, \PreBiSp(-,n)),\\
 \BiSp(F,n)&:= \map_{\oL (dg^+DG^+\Aff(R)^{\et})^{\wedge}}(D_*F_{\rig}, \BiSp(-,n)),
 \end{align*}
 noting that the construction of pre-bisymplectic structures from Definition \ref{PreSpdef2}  is clearly functorial, while non-degeneracy is preserved by homotopy \'etale morphisms.

Given a morphism $\eta \co G \to F$ of  homogeneous derived NC prestacks with bounded below cotangent complexes at all points, we then define
the spaces $\BiIso(F,G;n)$ and $\BiLag(F,G;n)$ of $n$-shifted bi-isotropic and bi-Lagrangian  structures on $G$ over $F$ by
\begin{align*}
 \BiIso(F,G;n)&:=  \PreBiSp(F,n)\by^h_{\PreBiSp(G,n)}\{0\},\\ 
 \BiLag(F,G;n)&:= \map_{\oL ((dg^+DG^+\Aff(R)^{[1]})^{\et})^{\wedge}}(D_*(G \to F)_{\rig}, \BiLag(-,-;n)),
 \end{align*}
where $dg^+DG^+\Aff(R)^{[1]}$ denotes the arrow category of $dg^+DG^+\Aff(R)^{[1]}$, with  $(dg^+DG^+\Aff(R)^{[1]})^{\et} $ its wide subcategory in which the only permitted morphisms from $(A \to B)$ to $(A' \to B')$ being those commutative diagrams in which the maps $A \to A'$ and $B \to B'$ are \'etale, $\BiLag(-,-;n)$ denotes the functor $(A \to B) \mapsto \BiLag(A,B;n)$, and $D_*(G \to F)_{\rig}$ denotes the functor $(A \to B) \mapsto D_*F_{\rig}(A) \by^h_{D_*F(B)}D_*G_{\rig}(B)$.
 \end{definition}

In other words, an $n$-shifted bisymplectic structure on $F$ consists of compatible   $n$-shifted bisymplectic structures on $A$ for all rigid points $x \in D_*F(A)$ and all $A$. Similarly, an $n$-shifted bi-Lagrangian structure of $G$ over $F$  consists of compatible   $n$-shifted bi-Lagrangian structures on arrows $\phi \co A \to B$  for all rigid points $x \in D_*F(A)$ and $y \in D_*G(B)$ with $\phi(x) \simeq \eta(y) \in D_*F(B)$.

\begin{remark}
At first sight, there might seem to be a discrepancy between the definitions of   $\PreBiSp(F,n)$ and  $\BiSp(F,n)$, since the latter is defined as a mapping space on a smaller category. However, \cite[Corollary \ref{NCstacks-etsitecor}]{NCstacks} ensures  that 
 \begin{align*}
  &\map_{ \oL (dg^+DG^+\Aff(R)^{\et})^{\wedge}}(D_*F_{\rig}, \PreBiSp(-,n))\\
&  \simeq \map_{\oL dg^+DG^+\Aff(R)^{\wedge}}(D_*F,\PreBiSp(-,n)), 
 \end{align*}
so the two spaces can be defined consistently, but we will wish to exploit the greater functoriality enjoyed by pre-bisymplectic structures.

 Since NC derived Artin prestacks have  bounded below cotangent complexes, we also  have to check that there is no conflict between Definitions \ref{PreBiSpArtindef} and \ref{PreBiSphgsdef}. 
That the characterisation above  coincides with $\PreBiSp(F,n)$ in the sense of  Definition \ref{PreBiSpArtindef} just follows by adapting \cite[Lemma 3.25]{poisson} to the non-commutative setting. The equivalence of the definitions of $\BiSp(F,n)$ then follows because non-degeneracy is preserved by homotopy \'etale morphisms.

The same reasoning also ensures consistency of the respective definitions for bi-isotropic and bi-Lagrangian structures.  
 \end{remark}

%


Working in this level of generality gives us the following broad generalisation of Proposition \ref{weilArtinprop}, without having to worry about whether the resulting derived NC prestacks are Artin:

\begin{proposition}\label{weilhgsprop}
 Take  a cofibrant $R$-DGAA $S \in dg_+\Alg(R)$ is perfect as an  $R$-module, and equipped with an $R$-linear chain  map $\tr \co S/[S,S] \to R_{[-d]}$ which is non-degenerate in the sense that $a \mapsto \tr(a\cdot-)$ is a quasi-isomorphism $S \to \HHom_R(S,R)[d]$.
 Then for any  homogeneous derived NC prestack $F$ with perfect cotangent complexes at all points, there is a natural map
\[
 \BiSp(F,n) \to \BiSp(\Pi_{S/R}F,n+d)
\]
from the space of $n$-shifted bisymplectic structures on $F$ to the space of $(n+d)$-shifted bisymplectic structures on the Weil restriction $\Pi_{S/R}F$, where $\Pi_{S/R}F(B):=F(B\ten_R^{\oL}S)$.

Moreover, for any morphism $G \to F$ of homogeneous derived NC prestacks with perfect cotangent complexes,
 there is similarly a natural map
\[
 \BiLag(F,G;n) \to \BiLag(\Pi_{S/R}F,\Pi_{S/R}G;n+d)
\]
from the space of $n$-shifted bi-Lagrangian structures on $G$ over $F$ to the space of $(n+d)$-shifted bi-Lagrangian structures on $\Pi_{S/R}G$ over $\Pi_{S/R}F$.
 \end{proposition}
\begin{proof}
It is straightforward to check that $\Pi_{S/R} $ is homotopy-preserving and homogeneous in the sense of \cite[Definition \ref{NCstacks-Fcotdef}]{NCstacks}, and arguing as in the proof of Proposition \ref{weilArtinprop} then shows that it has cotangent complexes
\[
 \bL_{\Pi_{S/R}F,x} \simeq \bL_{F,x}\ten^{\oL}_{S^{\oL,e}}S^*, 
\]
where $S^*:=\HHom_R(S,R)$.

The rest of the proof then follows exactly as in Proposition \ref{weilArtinprop}, replacing DGAAs with stacky DGAAs.
\end{proof}

\begin{remark}\label{cyclicrmk1}
Since $S$ is assumed cofibrant in Proposition \ref{weilhgsprop}, it follows from \cite[Theorem 3.1.6]{lodayCyclic} that the complex $(S/[S,S])/R$ is quasi-isomorphic to the cone of  $ \mathbf{HC}_{R,\bt}(R) \to \mathbf{HC}_{R,\bt}(S)$ for the complex $\mathbf{HC}_{R,\bt}(S)$ calculating $R$-linear cyclic homology of $S$. 

Since $ \mathbf{HC}_{R,\bt}(R)$ is quasi-isomorphic to $R[u^{-1}]$ for $u$ of chain degree $2$, which has $R$-linear dual $R\llb u \rrb$, it follows that  the $R$-linear dual $(S/[S,S])^*$ is quasi-isomorphic to  cocone of $ \mathbf{HC}_{R,\bt}(S)^*  \to u R\llb u \rrb$, giving a long exact sequence
\[
\to  \HC_R^{-d-1}(S) \to \H_{d+1}(uR\llb u \rrb ) \to   \H_d((S/[S,S])^*)  \to  \HC_R^{-d}(S) \to \H_d( u R\llb u \rrb) \to , 
\]
thus describing 
homotopy classes  $\H_d((S/[S,S])^*)$  of chain maps $\tr \co S/[S,S] \to R_{[d]}$ in terms of cyclic cohomology $\HC_R^{i}(S)=\Ext^i_R(\mathbf{HC}_{R,\bt}(R),R)$.
\end{remark}

\begin{remark}\label{weilhgsprop2}
Proposition \ref{weilhgsprop} admits a generalisation to mapping stacks with far more general sources, going significantly beyond the generality of  \cite[Theorem 2.5]{PTVV} in the commutative setting. 

If $F$ is a homogeneous derived NC prestack, and  $Y \in dg^+DG^+\Aff(R))^{\wedge}$ a functor on stacky $R$-DGAAs, then we have a derived NC prestack 
 $\oR\Map(Y,D_*F)$ given by $\oR\Map(Y,D_*F)(B):=  \map_{dg^+DG^+\Aff(R)^{\wedge}}(Y \by \oR\Spec^{nc} B,D_*F)$; when $Y=D_*X$ for $X$ Artin, \cite[Proposition \ref{NCstacks-replaceprop}]{NCstacks} implies that this is just the mapping prestack $\oR\Map(X,F)$. Note that $\oR\Map(Y,D_*F)$ inherits homogeneity from $Y$ because it is defined as a homotopy limit.

 For any point $x \in \Map(Y,D_*F)(B)$, pulling back gives
  a  filtered morphism
\[
 \oR\Gamma(F, \DR_{\cyc}(\sO_F/R)) \to \oR\Gamma(Y,  \DR_{\cyc}(B\ten^{\oL}_R\sO_Y /R)),
\]
 which as in the proof of Proposition \ref{weilArtinprop} we can combine with the natural map.
 \[
\oR\Gamma(Y, \oL\Omega^p_{(B\ten^{\oL}_R\sO_Y)/R,\cyc})\to \oR\Gamma(\oL\Omega^p_{B/R,\cyc}\ten^{\oL}_R\oL(\sO_Y/[\sO_Y,\sO_Y])).
\]

The trace map we now need is  a $d$-cycle in  the  derived cosections complex $\hatTot \oL\Gamma(Y, \oL(\sO_Y/[\sO_Y,\sO_Y])^*)$ of the  copresheaf $\oL(\sO_Y/[\sO_Y,\sO_Y])^*$ over $Y$. For any $B^{\oL,e}\ten_R^{\oL}\sO_Y^{\oL,e}$-module $\sF$, this induces a pairing 
\begin{align*}
&\oR\hatHHom_{B^{\oL,e}\ten_R^{\oL} \sO_Y^{\oL,e}}(\sF, B^{\oL,e}\ten_R^{\oL}\sO_Y)\ten_R^{\oL} \hatTot\oR\Gamma(Y,\sF\ten_{\sO_Y^{\oL,e}}^{\oL}\sO_Y)\\
&\to \hatTot\oR\Gamma(Y, B^{\oL,e} \ten_R^{\oL} \sO_Y\ten_{\sO_Y^{\oL,e}}^{\oL}\sO_Y) \to B^{\oL,e}_{[d]}, 
\end{align*}
where the final map combines multiplication with the trace. Provided this is a perfect pairing whenever $\sF$ is perfect, the proof of Proposition \ref{weilArtinprop} adapts to give us a natural map 
\[
\BiSp(F,n) \to \BiSp(\oR\Map(Y,D_*F),n-d);
\]
in particular, duality ensures that $\Map(Y,D_*F)$ has perfect cotangent complexes
\[
\bL_{\Map(Y,D_*F),x} \simeq  \hatTot\oR\Gamma(Y, x^*\bL_F\ten_{\sO_Y^{\oL,e}}^{\oL}\sO_Y)[d].
 \]
\end{remark}


\subsection{Structures on derived Artin NC prestacks, revisited}
 
\subsubsection{Integration of formal bisymplectic structures}

\begin{lemma}\label{stacky0QIMlemma}
 For $C \in DG^+dg\Alg(R)$, let $P$ be a left $C$-module in double complexes with $P^{\#}$ cofibrant in the projective model structure on graded $C^{\#}$-modules, and for which $P^{<0}=0$, with  $\bar{P}^q:=(C^0\ten_CP)^q$ acyclic for $q >N$. Let $M,Q$ be right $C$-modules satisfying the analogous conditions Then the complex
 $\Tot^{\Pi}(P\ten_CM)$, 
  resp. $\hatHHom_C(Q,M)$, 
 is  acyclic whenever  $\Tot (\bar{M}^{\le N})$  or $\Tot (\bar{P}^{\le N})$, resp. $\Tot (\bar{Q}^{\le N})$,  is acyclic.
\end{lemma}
\begin{proof}
 Since all complexes are concentrated in non-negative cochain degrees, the filtration of $M$ by submodules $M C^{\ge i}\cong M\ten_CC^{\ge i}$ is complete, as are the induced filtrations on $\ten_C$ and $\cHom_C$. We thus have completely convergent spectral sequences 
 \begin{align*}
\H_j  \Tot^{\Pi}(\bar{P}\ten_{C^0}C^i\ten_{C^0}\bar{M})&\abuts \H_{j-i}\Tot^{\Pi}(P\ten_CM),\text{ and} \\
  \H_j\cHom_{C^0}( \bar{Q},\bar{M}\ten_{C^0}C^i )^r &\abuts \H_{j}\cHom_C(Q,M)^{r+i}.
 \end{align*}
 
 The hypotheses also give us a levelwise quasi-isomorphism $ \bar{P}\ten_{C^0}C^i\ten_{C^0}\bar{M} \to  \bar{P}^{\le N}\ten_{C^0}C^i\ten_{C^0}\bar{M}^{\le N} $, so acyclicity of  $ \Tot^{\Pi}(P\ten_CM)$ follows because $\Tot \bar{M}^{\le N}$ or  $\Tot \bar{P}^{\le N}$ is an acyclic cofibrant $C^0$-module.
 
Similarly, we have a zigzag of levelwise quasi-isomorphisms between $\cHom_{C^0}( \bar{Q},\bar{M}\ten_{C^0}C^i )$ and $\cHom_{C^0}( \bar{Q}^{\le N},\bar{M}^{\le N}\ten_{C^0}C^i )$. From the spectral sequence above, we thus have that $\cHom_C(Q,M)^r$ is acyclic for all $r<-N$ and
\[
 \H_j \Tot^{\Pi} \cHom_{C^0}( \bar{Q}^{\le N},\bar{M}^{\le N}\ten_{C^0}C^i ) \abuts \H_{j-i}\Tot^{\Pi}\cHom_C(Q,M)^{\ge -N},
\]
with $\Tot^{\Pi}\cHom_C(Q,M)^{\ge -N} \simeq \hatHHom_C(Q,M)$.
Acyclicity then follows because
\[
  \Tot^{\Pi} \cHom_{C^0}( \bar{Q}^{\le N},\bar{M}^{\le N}\ten_{C^0}C^i )\cong \HHom_{C^0}(\Tot \bar{Q}^{\le N},\Tot(\bar{M}^{\le N})\ten_{C^0}C^i). \qedhere
\]
\end{proof}

\begin{lemma}\label{calcTOmegaetlemma}
 If $f \co  A \to B$ is a homotopy \'etale cofibration between cofibrant objects in $DG^+dg\Alg(R)$, then the maps
 \begin{align*}
 \Tot^{\Pi}((\Omega^1_A)^{\ten p}\ten_{(A^e)^{\ten p}} \{B^{\ten p}\}) &\to  \Tot^{\Pi}((\Omega^1_B)^{\ten p}\ten_{(B^e)^{\ten p}} \{B^{\ten p}\}), \text{ and}\\
 \hatHHom_{(B^e)^{\ten p}}((\Omega^1_B)^{\ten p},  \{B^{\ten p}\})&\to  \hatHHom_{(A^e)^{\ten p}}((\Omega^1_A)^{\ten p},  \{B^{\ten p}\})
 \end{align*}
are quasi-isomorphisms. If moreover $f$ admits a retraction $\rho$ as an $A$-bimodule in double complexes, then the maps
 \begin{align*}
  \Tot^{\Pi}((\Omega^1_A)^{\ten p}\ten_{(A^e)^{\ten p}} \{A^{\ten p}\}) &\to  \Tot^{\Pi}((\Omega^1_A)^{\ten p}\ten_{(A^e)^{\ten p}} \{B^{\ten p}\}), \text{ and}\\
  \hatHHom_{(A^e)^{\ten p}}((\Omega^1_A)^{\ten p},  \{A^{\ten p}\})&\to\hatHHom_{(A^e)^{\ten p}}((\Omega^1_A)^{\ten p},  \{B^{\ten p}\})
 \end{align*}
are also quasi-isomorphisms.
\end{lemma}
 \begin{proof}
 The first pair of equivalences follow from the left-hand acyclicity cases of  Lemma \ref{stacky0QIMlemma}, taking $C= (B^e)^{\ten p}$, $P=Q= \coker( (\Omega^1_A)^{\ten p}\ten_{(A^e)^{\ten p}} (B^e)^{\ten p} \to (\Omega^1_B)^{\ten p})$,  and $M= \{ \tilde{B}^{\ten p}\}$ for the graded-cofibrant replacement $\tilde{B}:= \cone(\Omega^1_B \to B^e)$ of $B$ as a $B^e$-module.
 
 For the second pair of equivalences, note that we can replace $\{A^{\ten p}\}$ with $\{ \cocone(B \to B/A)^{\ten p}\}$. The retraction $\rho$ allows us to identify $B/A$ as an $A$-bilinear direct summand of $\Omega^1_{B/A}:=\ker(B\ten_AB \to B)$, where the inclusion map is $[b] \mapsto db$ and the retraction is the composite $\Omega^1_{B/A} \to B\ten_AB \xra{\rho\ten \id} B \to B/A$. The morphism $\{ \cocone(B \to B/A)^{\ten p}\} \to \{B^{\ten p}\}$ is thus a retract of $\{ \cocone(B \to \Omega^1_{B/A})^{\ten p}\} \to \{B^{\ten p}\}$, so it suffices to show that the latter map  gives rise to quasi-isomorphisms of the $\Hom$ and $\ten$ constructions above.
 
 Filtering by powers of $\Omega^1_{B/A}$, this reduces to showing that the  $(A^e)^{\ten p}$-modules $M_{q,r}:=\{\tilde{B}^{\ten q_1}\ten (\Omega^1_{B/A})^{\ten r_i} \ten \dots \ten\tilde{B}^{\ten q_m}\ten (\Omega^1_{B/A})^{\ten r_m}\}$ produce acyclic output when entered in place of  $\{A^{\ten p}\}$ in the expressions above whenever the $r_i$ are not all $0$. This follows from the right-hand acyclicity cases of  Lemma \ref{stacky0QIMlemma}, taking $C= (B^e)^{\ten p}$, $P=Q=(\Omega^1_A\ten_{A^e}B^e)^{\ten p}$ and  $M= M_{q,r}$. 
  \end{proof}


 \begin{corollary}\label{integratecorBiSp}
  For a sqc derived NC Artin prestack $F$ coming from a simplicial diagram $X_{\bt}$ of derived NC affines, the natural maps $\BiSp(F,n) \to \BiSp(D^*O(X),n)$ and $\PreBiSp(F,n) \to \PreBiSp(D^*O(X),n)$ are weak equivalences.
  \end{corollary}
\begin{proof}
The proof of \cite[Proposition \ref{NCstacks-replaceprop}]{NCstacks} resolves  the functor $D_*F$ by the simplicial stacky  derived NC affine $j \mapsto \oR \Spec^{nc}D^*O(X^{\Delta^j})$, so $\BiSp(F,n) \simeq \ho\Lim_{j \in \Delta} \BiSp(D^*O(X^{\Delta^j},n)$ and similarly for $\PreBiSp$. The simplicial operations are all homotopy \'etale, and the face maps  have sections given by  degeneracy maps, so all maps in the diagram are equivalences by the tensor cases of Lemma \ref{calcTOmegaetlemma}.
\end{proof}

In other words, every shifted bisymplectic structure on $D^*O(X)$ (an NC analogue of a derived Lie algebroid) integrates uniquely to one on $F$ (an algebraic NC analogue of a derived  Lie groupoid), in stark contrast with the commutative setting.

\subsubsection{Calabi--Yau structures via \tps{$\bG_m$}{Gm}-actions}\label{CYsn} 


We now explicitly describe shifted bisymplectic structures on the stacky DGAA $O([\Spec^{nc} A/\g_m])$ from Example \ref{BGex1}, which we can think of as the formal completion of the adjoint action of $\bG_m$ on  $\Spec^{nc} A$. Explicitly, for a stacky DGAA $B$, we have
\[
 \Hom(O([Y/\g_m]),B) \cong \{(y, b) \in Y(B^0)\by B^1 ~:~ \pd b = b^2 \in B^2, ~ \pd y = [b, y]\co A \to B^1\}.
\]

We begin by fixing some notation to simplify the calculations.

\begin{definition}\label{saddef}
 For $A \in DG^+dg\Alg(R)$, write $A\<s_{\ad}\> := O([\Spec^{nc} A/\g_m])$, the stacky DGAA $A\<s\>$ for $s \in A\<s\>^1_0$, with $\pd s = s^2$, $\delta s =0$,  and $\pd a = \pd_A a + [s,a]$ for all $a \in A$.
 
 Similarly,  for $s$ as above, write $A\<s_0\>$ for  the stacky DGAA $A\<s\>$ with $\pd_A a = 0$ for all $a \in A$. Then write $A\<s_l\>$ (resp. $A\<s_r\>$) for the $A\<s_{\ad}\>\!-\!A\<s_0\>$-bimodule (resp.  $A\<s_0\>\!-\!A\<s_{\ad}\>$-bimodule) with $\pd a =  \pd_Aa +sa$ (resp. $\pd a = \pd_Aa \pm as$) for all $a \in A$.
\end{definition}
Note that $A\<s_0\> \cong A\<s_r\>\ten_{A\<s_{\ad}}A\<s_l\>$.

Writing $\cS'(A)$ for the cocone $(\Omega^1_A \oplus (A^e)^{[-1]},\delta,\pd+ \alpha)$ of the inclusion map $\alpha \co \Omega^1_A \to A^e$ (corresponding to $\cS(A)^{[-1]}$ from \cite{yeungPreCYModReps})  and identifying it with the cocone of $[ds,-] \co \Omega^1_A \to A(ds)A$, the following then follows easily from the definitions.
\begin{lemma}\label{CYlemmaOmega}
There are canonical isomorphisms 
 \begin{align*}
  \Omega^1_{A\<s_{\ad}\>} &\cong A\<s_l\>\ten_A\cS'(A)\ten_AA\<s_r\>, \text{ and hence}\\
  \Omega^p_{A\<s_{\ad}\>} &\cong A\<s_l\>\ten_A(\cS'(A)\ten_AA\<s_0\>)^{\ten_A p} \ten_{A\<s_0\>}A\<s_r\>.
 \end{align*}
 \end{lemma}

Following \cite[\S 2.2]{yeungPreCYModReps}, write $F^p\sX^{\tot}(A)$ for the extended de Rham  complex given by  replacing $\Omega^1_A$ with $\cS'(A)$ in the definition of $\F^p\DR(A)_{\cyc}$ (\ref{biDRdef}). 
 \begin{lemma}\label{CYlemmaDR}
  For $p>0$, the complex $F^p\DR(A\<s_{\ad}\>)_{\cyc}$ is canonically quasi-isomorphic to  $F^p\sX^{\tot}(A)$. 
 \end{lemma}
\begin{proof}
 It follows from Lemma \ref{CYlemmaOmega} that $\Omega^j_{A\<s_{\ad}\>}\ten_{A\<s_{\ad}\>^e}A\<s_{\ad}\>$ is  isomorphic to the cyclic tensor product $\ldots \ten_A \cS'(A)\ten_AA\<s_0\>\ten_A\cS'(A)\ten_AA\<s_0\>\ten_A \ldots$ for $j>0$. Moreover, the $A$-bimodule  $A\<s_0\>$ admits a contracting cochain homotopy to $A$ (i.e. $h$ with $[\pd,h]=\id$ and $[\delta,h]=0$) because $\pd s^{2n+1}=s^{2n+2}$. 
 
 Thus $\Omega^j_{A\<s_{\ad}\>}\ten_{A\<s_{\ad}\>^e}A\<s_{\ad}\>$ admits a contracting cochain homotopy to  the cyclic tensor product $\ldots \ten_A \cS'(A)\ten_A\cS'(A)\ten_A \ldots$, and hence becomes quasi-isomorphic to it on applying $\Tot^{\Pi}$. Combining these inclusions for $j \ge p$ and taking $C_p$-coinvariants gives the quasi-isomorphism $F^p\sX(A)\into F^p\DR(A\<s_{\ad}\>)_{\cyc}$.  
 \end{proof}
Beware that the lemma does not give a subcomplex of $F^p\DR(A\<s_{\ad}\>)$ itself, because $A$ is a subcomplex of $A\<s_0\>$ but not of $A\<s_l\>$ and $A\<s_r\>$. In other words, on taking the cyclic quotient, the non-$\pd$-closed subalgebra of $\Omega^*_{A\<s_{\ad}\>}$ generated by $\Omega^*_A$ and $ds$ becomes a double subcomplex which is moreover $\Tot^{\Pi}$-quasi-isomorphic on $F^p$. Equivalently, $\sX^{\tot}(A)$ is the cyclic quotient of an analogue of $\DR(A)\<ds\>$ with curvature $[ds,-]$. 


 \medskip
 Extending \cite[Definition 2.12]{yeungPreCYModReps} or 
 \cite[Definition 2.1]{KontsevichTakedaVlassopoulosLegendre} to stacky algebras along the lines of Definition \ref{bisymplecticdefstacky}, a smooth (or left) $d$-Calabi--Yau (d-CY) structure on $A$ is an  element of $\z_d\Tot^{\Pi}\mathbf{HN}_{\bt}(A)$ (negative cyclic homology) whose image in Hochschild homology induces a quasi-isomorphism
 \[
  \hatHHom_{A^{\oL,e}}(A, (A^0)^{\oL,e})\to  \Tot^{\Pi} (A\ten_{A^{\oL,e}}(A^0)^{\oL,e})_{[d]}.
 \]
 As in earlier versions of \cite{yeungPreCYModReps}, we can then form a space of smooth $d$-CY structures as the union of the non-degenerate path components of $\tau_{\le 0}(\Tot^{\Pi}\mathbf{HN}_{\bt}(A)_{[d]})$  (similarly to Definition \ref{PreSpdef}). 

 For the prestack $[ \Spec^{nc} A/\bG_m] $ of Remark \ref{EHCrmk} given by the adjoint action of $\bG_m$, we now have:
\begin{proposition}\label{CYprop}
 For $A \in DG^+dg\Alg(R)$ cofibrant, the following are equivalent:
 \begin{enumerate}
  \item The space of smooth $d$-Calabi--Yau structures on $A$.
  \item The space $\BiSp([\Spec^{nc} A/\g_m],2-d)$ of $(2-d)$-shifted bisymplectic structures on the formal quotient of $\Spec^{nc} A$ by $\g_m$.
  \item The space $\BiSp(\oR[\Spec^{nc} A/\bG_m],2-d)$ of $(2-d)$-shifted bisymplectic structures on the derived quotient of $\Spec^{nc} A$ by $\bG_m$.
  \end{enumerate}
\end{proposition}
\begin{proof}
 Replacing a cone in the cochain direction with one in the chain direction does not affect $\Tot^{\Pi}$, so the proof of \cite[Theorem 4.11]{yeungExtnFeiginTsygan} gives a quasi-isomorphism $\Tot^{\Pi}\mathbf{HN}(A)_{[4]}\to F^2\sX^{\tot}(A)$, and hence from $d$-chains in negative cyclic homology to $(2-d)$-shifted bisymplectic structures. Since $\Omega^1_{A\<s_{\ad}\>}\ten_{A\<s_{\ad}\>^e}(A^0)^e \cong  \cS'(A)\ten_{A^e}A^0$, the non-degeneracy conditions also correspond, giving an equivalence between the first two spaces.
 
The second equivalence then follows because $ \BiSp(\oR[\Spec^{nc} A/\bG_m],2-d) \to \BiSp(\tilde{A}\<s_{\ad}\>,2-d)$ is a weak equivalence by Proposition \ref{integratecorBiSp} and \cite[Example \ref{NCstacks-DstarBGex}]{NCstacks}, for any cofibrant replacement  $\tilde{A}\<s_{\ad}\> $ of $A\<s_{\ad}\>$, but this in turn is equivalent to $\BiSp(A\<s_{\ad}\>,2-d)$ (cf. Remark \ref{weakerassumptionrmk}).
 \end{proof}

\begin{remark}[dg categories]\label{dgcatrmk}
 Although we have only considered the case where $A$ is a dg algebra, there are obvious generalisations of these calculations to dg categories $\cA$. These come from a natural prestack $[\Spec^{nc} A/\mathrm{Stab}_{\bG_m}X]$ when $\cA$ has a finite set $X$ of objects; formally adding a unit would give something closely related for dg categories with infinitely many objects (thus losing the implicit requirement $\sum e_x=1$ below). 
 
 We first form the DGAA $A:= \bigoplus_{x,y}\cA(x,y)$, then define $[\Spec^{nc} A/\mathrm{Stab}_{\bG_m}X](B) \subset [\Spec^{nc} A/\bG_{m}](B)$ to be the nerve of the groupoid  with the same objects $\Hom(A,B)$, but with a morphism $g \in (\z_0B)^{\by}$  from $\phi$ to $g\phi g^{-1}$ only if $g$ commutes with the orthogonal idempotents $e_x:= \phi(1_x)$. We can interpret $\phi$ as determining a   dg category $\cB$  with objects $X$ and morphisms $\{e_xB e_y\}_{x,y \in X}$, together with a dg functor $\cA \to \cB$, and then $g$ is a natural isomorphism between the respective dg functors.
 
The stacky DGAA $D^*O([\Spec^{nc} A/\mathrm{Stab}_{\bG_m}X])$ is $A\<s_{\ad}\>/([s,1_x])_{x \in X}$. 
Replacing $\ten_R$ with $\ten_{R^X}$ throughout the arguments above,
the space of smooth $d$-Calabi--Yau structures on $\cA$ is then equivalent to the spaces $\BiSp(A\<s_{\ad}\>/([s,1_x])_{x \in X} ,2-d)$ and $\BiSp(\oR[\Spec^{nc} A/\mathrm{Stab}_{\bG_m}X],2-d)$, with the same proof as in the one object case.
\end{remark}

 \begin{remark}
  Although the characterisation of CY structures as shifted bisymplectic structures on $[\Spec^{nc} A/\g_m]$ gives a fairly simple description, the characterisation as structures on  $\oR[\Spec^{nc} A/\bG_m]$ is much more powerful. In particular, combined with Proposition \ref{weilhgsprop} and Lemma \ref{commSplemma}, it immediately induces shifted symplectic structures on (commutative) derived moduli stacks of representations, recovering  \cite[Theorem 4.53]{yeungPreCYModReps}.
  
 \end{remark}

 \subsection{The \tps{$(2-d)$}{(2-d)}-shifted bisymplectic structures on \tps{$\Perf_{\cA}$}{Perf(A)} and \tps{$\Mor_{\cA}$}{Mor(A)}}

The reasoning of \cite[\S2]{dmsch} ensures that for any dg category $\cA$, the moduli functor sending $B$ to the nerve of the core of the dg category of $\cA\ten^{\oL}_RB$-modules is homogeneous, as is any open subfunctor. 
 
Combining Proposition \ref{weilhgsprop} with Theorem \ref{Perfthm1} gives a criterion for the derived NC prestack $\Perf_A \co B \mapsto \Perf(A\ten^{\oL}_RB)$ to carry a $(2-d)$-shifted bisymplectic structure, and Remark \ref{weilhgsprop2} could extend this to some more dg categories, but we now give general conditions on $\cA$ for $\Perf_{\cA}$ to carry a shifted bisymplectic structure, and then consider the analogous question for the derived NC prestack of Morita morphisms.
 
\subsubsection{Perfect complexes}
  
For an $R$-linear dg category, consider the derived NC prestack $\Perf_{\cA}\co   B \mapsto \Perf(\cA\ten^{\oL}_RB)$, the nerve of the $\infty$-category of perfect right $\cA\ten^{\oL}_RB$-modules in complexes --- see for instance \cite{kellerModelDGCat} for more detail on modules over dg categories.

As in \cite[Remark \ref{NCstacks-Perfnicermk}]{NCstacks}, the derived NC prestack $\Perf_{\cA}$ is homogeneous, with  a perfect cotangent complex whenever $\cA$ is a locally proper dg category over $R$.
As one would expect, existence of a $d$-Calabi--Yau structure on $\cA$ will suffice to put a $(2-d)$-shifted symplectic structure on $\Perf_{\cA}$; following \cite[\S 10.2]{KontsevichSoibelmanAinftyalgcats}, we can weaken this from the classical notion by taking a non-degenerate element in cyclic cohomology (i.e. cohomology of the $R$-linear dual of the complex computing cyclic homology of a dg category).

\begin{theorem}\label{Perfthm2}
If $\cA$ is a locally proper dg category over $R$, then there is a natural map from the cyclic cohomology group $\HC_R^{-d}(\cA)$ to equivalence classes $\pi_0\PreBiSp(\Perf_{\cA},2-d)$ of $(2-d)$-shifted pre-bisymplectic structures on the homogeneous  derived  NC prestack $\Perf_{\cA}$.

The shifted pre-bisymplectic structure coming from a class $\alpha \in \HC_R^{-d}(\cA)$ is bisymplectic if and only if its image $\bar{\alpha}$ under the natural map $\HC_R^{-d}(\cA) \to \HH_R^{-d}(\cA,\cA^*)=\Ext^{-d}_{\cA^{\oL,e}}(\cA(-,-),\cA^*(-,-))$ is non-degenerate, i.e.  a quasi-isomorphism
\[
\cA(-,-) \to \cA^*(-,-)[-d] 
\]
of $\cA^{\oL,e}$-modules, where $\cA^*(X,Y):=\oR\HHom_R(\cA(Y,X),R)$.
\end{theorem}
\begin{proof}
By Proposition \ref{HNsympprop1}, for any $B \in dg_+\Alg(R)$, there is a natural 
map from the space of $d$-cycles in  $ \mathbf{HN}_{R,\bt}(B)$ to $\PreBiSp(B,2-d)$, the space of $(2-d)$-shifted pre-bisymplectic structures on $B$. 

Adapting Lemma \ref{goodwillielemma} to dg categories by using the Chern character of \cite{tabuadaUniversal}, we also have a Chern character
\[
\ch^-\co K(\cA\ten^{\oL}_RB)_{\Q} \to \mathbf{HN}_{R,\bt}(\cA\ten^{\oL}_RB)
\]
from $K$-theory to the complex computing negative cyclic homology of $\cA\ten^{\oL}_RB$.

Now, if $\mathbf{HC}_{R,\bt}(\cA)$ is the complex computing cyclic homology of $\cA$, and $\mathbf{HC}_{R}^{\bt}(\cA)$ its $R$-linear dual,
observe that there is a natural pairing of complexes
\[
 \mathbf{HC}_{R}^{\bt}(\cA)\ten^{\oL}_R \mathbf{HN}_{R,\bt}(\cA\ten^{\oL}_RB) \to \mathbf{HN}_{R,\bt}(B)
\]
--- this  follows because $\mathbf{HC}_{R}^{\bt}(\cA)$ is essentially the topological $R\llb u \rrb$-linear dual of $\mathbf{HN}_{R,\bt}(\cA)$, while $\mathbf{HN}_{R,\bt}(\cA\ten^{\oL}_RB)$ can be written as a topological $R\llb u \rrb$-linear tensor product of $\mathbf{HN}_{R,\bt}(\cA)$ and $\mathbf{HN}_{R,\bt}(B)$.

Combined with  the factor $\xi_2$ from Lemma \ref{KaroubiChernLemma}, this gives us a map 
\[
K(\cA\ten^{\oL}_RB)_{\Q}\ten_{\Q} \mathbf{HC}_{R}^{\bt}(\cA)\to F^2\DR_{\cyc}(B/R)[2],
\]
functorial in $B$. Combining this with the 
natural map from $\Perf$ to $K$
gives us our desired map from the space of $(-d)$-cocycles in $\mathbf{HC}_{R}^{\bt}(\cA)$ to  the space $\PreBiSp(\Perf_{\cA},2-d)$ from Definition \ref{PreBiSphgsdef}, and our first required statement follows by taking path components.

In order to establish non-degeneracy, we begin by describing the cotangent complex of  $\Perf_{\cA}$.
If we take  a point $[E] \in \Perf_{\cA}(B)$ corresponding to a perfect complex $E$ of right $\cA\ten^{\oL}_RB$-modules, then as in \cite[Remark \ref{NCstacks-Perfnicermk}]{NCstacks}, the  tangent complex $\oR\HHom_{B^{\oL,e}}(\bL_{\Perf_{\cA},[E]},B^{\oL,e})$ is given by 
\[
\oR\HHom_{\cA \ten_R^{\oL} B}(E, E\ten_R^{\oL}B)[1] \simeq E\ten_{\cA}^{\oL}E^*[1],
\]
where $B$ acts on both terms in the $\HHom$ on the right and $E^*:=\oR\HHom_{\cA \ten_R^{\oL}B}(E,\cA(-,-) \ten_R^{\oL} B)$. 

By writing the right-hand copy of $E$ as $E^{**}$, 
we can also rearrange this expression as
\begin{align*}
\oR\HHom_{\cA \ten_R^{\oL} B}(E, E\ten_R^{\oL}B)[1] &\simeq \oR\HHom_{\cA^{\oL,e} \ten_R^{\oL} B^{\oL,e}}(E\ten_R^{\oL}E^*, \cA(-,-) \ten_R^{\oL} B^{\oL,e})[1]\\
&\simeq \oR\HHom_{\cA^{\oL,e} \ten_R^{\oL} B^{\oL,e}}(E\ten_R^{\oL}\cA^*(-,-) \ten_R^{\oL}E^*,  B^{\oL,e})[1],\\
&\simeq  \oR\HHom_{B^{\oL,e}}(E\ten_{\cA}^{\oL}\cA^*(-,-) \ten_{\cA}^{\oL}E^*,B^{\oL,e})[1],
\end{align*}
(each $\cA(X,Y)$ being perfect over $R$), so
\[
 \bL_{\Perf_{\cA},[E]} \simeq E\ten_{\cA}^{\oL}\cA^*(-,-) \ten_{\cA}^{\oL}E^*[-1].
\]

The proof of Theorem \ref{Perfthm1} now adapts to show that the map
\[
 \omega_2^{\sharp}\co\oR\HHom_{B^{\oL,e}}(\bL_{\Perf_{\cA},[E]},B^{\oL,e}) \to \bL_{\Perf_{\cA},[E]}[2-d]
\]
induced by our pre-bisymplectic structure $\omega$ is just the map
\[
 E\ten_{\cA}^{\oL}E^*[1] \to E\ten_{\cA}^{\oL}\cA^*(-,-) \ten_{\cA}^{\oL}E^*[1-d]
\]
induced by the quasi-isomorphism
\[
\bar{\alpha} \co \cA(-,-) \to \cA^*(-,-)[-d] 
\]
of $\cA^{\oL,e}$-modules, so is a quasi-isomorphism.
\end{proof}

\begin{remark}
 To understand how the conditions of Theorem \ref{Perfthm2} can be satisfied, first observe that for a $d$-CY dg category $\cA$, the Serre functor gives us bifunctorial isomorphisms $\cA(X,Y) \to \cA(Y,X)^*[-d]$, and in particular $\id_X \in \cA(X,X)$ gives rise to an $R$-linear  trace $\cA(X,X) \to R[-d]$; bifunctoriality then ensures that this defines a $(-d)$-cocycle in cyclic cohomology (explicitly calculated using the quotient form of cyclic homology as in \cite[Lemma 9.6.10]{W}). Non-degeneracy follows from these maps being isomorphisms. The datum of a non-degenerate element of cyclic cohomology $\HC_R^{-d}(\cA)$ precisely corresponds to a homotopy class   of a $d$-dimensional right Calabi--Yau structures on $\cA$ in the terminology of \cite[Definition 3.2]{BravDyckerhoff}, giving many examples of derived NC moduli functors with shifted bisymplectic structures.
 
 
 In particular, the conditions are satisfied by cyclic $A_{\infty}$-algebras and categories, along the lines of \cite[\S 10.2 and in particular Theorem 10.2.2]{KontsevichSoibelmanAinftyalgcats}. Explicitly, we can apply the quotient form of cyclic homology to the $A_{\infty}$ Hochschild complex of \cite[Theorem 3.7]{GetzlerJonesAinftyCyclicBar}, and then obtain the data of our required trace by applying the  given non-degenerate pairing, with the data of a weakly unital system as one of the inputs. 
 \end{remark}

 \begin{remark}
Several precursors of Theorem \ref{Perfthm2} can be found, such as \cite{ChenEshmatovCYAlgShiftedNCSympStr}, which constructs a form of $(2-d)$-shifted bisymplectic structure on the Koszul dual of a Koszul  $d$-CY algebra concentrated in degree $0$. However, beware that their definitions for shifted bisymplectic and symplectic structures are  not homotopy-invariant, being  much more restrictive than ours or those of \cite{KhudaverdianVoronov,PTVV}.
  \end{remark}

  \begin{remark}
Via  \cite[Remark \ref{NCstacks-DstarPerf}]{NCstacks}, Theorem  \ref{Perfthm2} also generalises to problems such as moduli of projective modules with connections. Instead of a locally proper dg category, take a category $\cA$ enriched in double complexes (in the form of cochain chain complexes), with the chain complex $\bigoplus_i \cA(X,Y)^i$ perfect over $R$ for all objects $X,Y \in \cA$, and $\cA(X,Y)^i=0$ for $i<0$. Our associated derived NC prestack is $\Perf_{\cA}:=\ho\Lim_{i \in \Delta}\Perf_{D^i\cA}$ for the cosimplicial denormalisation functor $D$, and an element of $\Perf_{\cA}(B)$ is then an $\cA\ten^{\oL}_RB$-module $E$ in double complexes for which $E^0$ is perfect over $\cA^0\ten^{\oL}_RB$ and which is homotopy Cartesian in the sense that the natural maps $E^0\ten_{\cA^0}\cA^i(-,-)\to E^i$ are quasi-isomorphisms. 

We can then form a bigraded cyclic cohomology complex, and take the (direct sum) total complex to give a notion of cyclic cohomology of $\cA$, elements of which will give rise to shifted pre-bisymplectic structures on $\Perf_{\cA}$. The non-degeneracy condition giving rise to  shifted bisymplectic structures will  then be that the natural map
\[
\hatTot(\cA^0\ten^{\oL}_{\cA}\cA^0)\simeq  \hatTot  (\cA^0\ten^{\oL}_{\cA}\cA\ten^{\oL}_{\cA}\cA^0) \to \hatTot  (\cA^0\ten^{\oL}_{\cA}\cA^*\ten^{\oL}_{\cA}\cA^0)[-d]
\]
induced by the cyclic cohomology element is a quasi-isomorphism of $\cA^0$-bimodules.
 \end{remark}


We now turn our attention to bi-Lagrangian structures. Given a dg functor $\theta \co \cB \to  \cA$ between locally proper dg categories over $R$, we  consider the cone $\mathbf{HC}_{R,\bt}(\cB,\cA)$ of the natural map $\theta \co \mathbf{HC}_{R,\bt}(\cB) \to\mathbf{HC}_{R,\bt}(\cA) $ between the complexes calculating cyclic homology, and then look at the relative cyclic cohomology groups $\HC_R^i(\cB,\cA)$ defined as cohomology of the $R$-linear dual $\mathbf{HC}_{R}^{\bt}(\cB,\cA)$ of 
 $\mathbf{HC}_{R,\bt}(\cB,\cA)$.
 
Via the natural  morphism from Hochschild homology  to cyclic homology, there is a natural map
\[
  \mathbf{HC}_{R}^{\bt}(\cB,\cA) \to \cocone(\CC^{\bt}_R(\cA,\cA^*) \to \CC^{\bt}_R(\cB,\cB^*)).
\]
Rewriting the right-hand side as 
$\cocone(\oR\HHom_{\cA^{\oL,e}}(\cA,\cA^*)\to \oR\HHom_{\cB^{\oL,e}}(\cB,\cB^*))$,
we see that this in turn maps to
\begin{align*}
 &\cocone(\oR\HHom_{\cB^{\oL,e}}(\theta_*\cA,\cB^*)\to \oR\HHom_{\cB^{\oL,e}}(\cB,\cB^*))\\
 &\simeq \oR\HHom_{\cB^{\oL,e}}(\cone(\cB \to \theta_*\cA),\cB^*).
\end{align*}
 
The proof of Theorem \ref{Perfthm2} now readily adapts to give the following: 
\begin{theorem}\label{Perfthm2Lag}
If $\theta \co \cB \to  \cA$  is a dg functor between locally  proper dg categories over $R$, then there is a natural map from the cyclic cohomology group $\HC_R^{-d}(\cB,\cA)$ to equivalence classes $\pi_0\BiIso(\Perf_{\cA},\Perf_{\cB};2-d)$ of $(2-d)$-shifted bi-isotropic structures on the homogeneous  derived  NC prestack $\Perf_{\cB}$ over $\Perf_{\cA}$, lifting the $(2-d)$-shifted pre-bisymplectic structure on $\Perf_{\cA}$ given by applying Theorem \ref{Perfthm2} to the image of the natural map $ \HC_R^{-d}(\cB,\cA) \to\HC_R^{-d}(\cA)$.  

When the image of a class  $\alpha \in \HC_R^{-d}(\cB,\cA)$ in $\HH_R^{-d}(\cA,\cA^*)$ is non-degenerate (making $\Perf_{\cA}$ bisymplectic as in Theorem \ref{Perfthm2}), the resulting  shifted bi-isotropic structure on $\Perf_{\cB}$ over $\Perf_{\cA}$ is 
is bi-Lagrangian if and only if its image  under the natural map 
\[
 \HC_R^{-d}(\cB,\cA) \to \Ext^{-d}_{\cB^{\oL,e}}(\cone(\cB \to \theta_*\cA),\cB^*)
\]
is non-degenerate, i.e. a quasi-isomorphism
\[
\cone(\cB \to \theta_*\cA) \to \cB^*[-d] 
\]
of $\cB^{\oL,e}$-modules.
\end{theorem}

\begin{remark}
 Note that a class in $\HC_R^{-d}(\cB,\cA)$ with the non-degeneracy properties above is precisely the datum of
 a homotopy class of $(d+1)$-dimensional right Calabi-Yau structures on the morphism $\theta \co \cB \to  \cA$  in the terminology of \cite[Definition 4.7]{BravDyckerhoff}, giving many examples of derived NC moduli functors with shifted bisymplectic structures.
 
 Using the data of a weak unit applied to a cyclic pairing, such classes arise from the pre-Calabi--Yau algebras $B$ of  \cite{IyuduKontsevichPreCYNCPoisson}, taking $\theta$ to be the inclusion $A_{\infty}$-morphism $B \to B\oplus B^*[-d]$  associated to such structures. 
\end{remark}

\subsubsection{Derived Morita morphisms}

We now consider the derived NC prestack $\Mor_{\cA}$, which sends $B \in dg_+\Alg(R)$ to 
 the space of derived Morita morphisms from $\cA$ to $B$, i.e. the space of dg functors from $\cA$ to $\per_{dg}(B)$; by \cite[Theorem 1.1]{toenMorita}, 
  this is equivalent to the nerve $\Mor(\cA,B)$ of the core of the simplicial category  associated to the dg category $\mathrm{mor}_{dg}(\cA,B)$ of those $\cA-B$-bimodules which are perfect over $B$. 
  When $\cA$ is a smooth and proper dg category over $R$, observe that we have a natural equivalence $\Mor_{\cA} \simeq \Perf_{\cA^{\op}}$. 

We have already seen part of this prestack: when $A \in dg_+\Alg(R)$, the derived NC prestack $\Mor_A$ contains an open (i.e. homotopy \'etale monomorphic) sub-prestack $[\oR\Spec^{nc}A/\bG_m]$ as in Remark \ref{EHCrmk}, which parametrises only those $\cA-B$-bimodules which are $B$-linearly isomorphic to $B$.
  
The reasoning of \cite[\S2]{dmsch} shows that  the derived NC prestack $\Mor_{\cA}$ is homogeneous, with  a perfect cotangent complex whenever $\cA$ is a smooth dg category over $R$.

Explicitly, if we take  a point $[E] \in \Mor_{\cA}(B)$ corresponding to an $\cA-B$-bimodule  $E$ which is perfect over $B$, then  the  tangent complex $\oR\HHom_{B^{\oL,e}}(\bL_{\Mor_{\cA},[E]},B^{\oL,e})$ is given by
 \begin{align*}
  \oR\HHom_{\cA^{\op} \ten_R^{\oL} B}(E, E\ten_R^{\oL}B)[1] &\simeq \oR\HHom_{\cA^{\oL,e}}(\cA(-,-), \oR\HHom_{B}(E, E \ten_R^{\oL}B))[1]\\
  &\simeq \oR\HHom_{\cA^{\oL,e}}(\cA(-,-), E'\ten_R^{\oL}E)[1]\\
   &\simeq\oR\HHom_{\cA^{\oL,e}\ten_R^{\oL} B^{\oL,e}}(\cA(-,-)\ten_R^{\oL}(E'\ten_R^{\oL}E),B^{\oL,e} )[1]\\
   &\simeq\oR\HHom_{B^{\oL,e}}(E'\ten_{\cA}^{\oL} \cA(-,-) \ten_{\cA}^{\oL}E,B^{\oL,e} )[1]\\
&\simeq\oR\HHom_{B^{\oL,e}}(E'\ten_{\cA}^{\oL} E,B^{\oL,e} )[1]
 \end{align*}
where $B$ acts on both terms in the $\HHom$ on the right and we set $E':= \oR\HHom_{B}(E,  B)$. 
  
Thus $\bL_{\Mor_{\cA},[E]}$ is quasi-isomorphic to the $B^{\oL,e}$-linear predual $E'\ten_{\cA}^{\oL} E[-1]$. Writing this as $ \cA(-,-)\ten_{\cA^{\oL,e}}^{\oL}(E'\ten_R^{\oL}E)[-1]$ shows that the $B^{\oL,e}$-module $ \bL_{\Mor_{\cA},[E]}$ is perfect whenever $\cA$ is smooth.     

Meanwhile, since $\cA(-,-)$ is assumed to be a perfect $\cA^{\oL,e}$-module, we can write $\cA(-,-)^! := \oR\HHom_{\cA^{\oL,e}}(\cA(-,-),\cA(-,-)\ten^{\oL}_R\cA(-,-))$ and then
\begin{align*}
 \oR\HHom_{\cA^{\op} \ten_R^{\oL} B}(E, E\ten_R^{\oL}B)[1] &\simeq\oR\HHom_{\cA^{\oL,e}}(\cA(-,-), E'\ten_R^{\oL}E)[1]\\
 &\simeq E'\ten_{\cA}^{\oL} \cA(-,-)^! \ten_{\cA}^{\oL}E[1].
\end{align*}

\begin{theorem}\label{Morthm}
If $\cA$ is a smooth dg category over $R$, then there is a natural map from the negative cyclic homology group $\HN_{R,d}(\cA)$ to equivalence classes $\pi_0\PreBiSp(\Mor_{\cA},2-d)$ of $(2-d)$-shifted pre-bisymplectic structures on the homogeneous  derived  NC prestack $\Mor_{\cA}$.

The shifted pre-bisymplectic structure coming from a class $\alpha \in \HN_{R,d}(\cA)$ is bisymplectic if and only if its image $\bar{\alpha}$ under the natural map $\HN_{R,d}(\cA) \to \HH_{R,d}(\cA)= \Tor_{d}^{\cA^{\oL,e}}(\cA(-,-),\cA(-,-))$ is non-degenerate, in the sense that it induces a  quasi-isomorphism
\[
\ev_{\bar{\alpha}}\co \cA(-,-)^!\to \cA(-,-)[-d] 
\]
given by allowing $ \cA^!= \oR\HHom_{cA^{\oL,e}}(\cA,\cA\ten^{\oL}_R\cA)$ to act on the first factor of  $\bar{\alpha} \in \H_d(\cA\ten^{\oL}_{\cA^{\oL,e}}\cA)$.
\end{theorem}
\begin{proof}
Since negative cyclic homology is functorial with respect to derived Morita morphisms, an element $[E] \in \Mor_{\cA}(B)$ defines a map $\mathbf{HN}_{R,\bt}(\cA)\to \mathbf{HN}_{R,\bt}(B)$. Combined with the map of Proposition \ref{HNsympprop1}, this gives us a map from $d$-cycles in $\mathbf{HN}_{R,\bt}(\cA)$ to $\PreBiSp(B,2-d)$, functorially in $(B,E)$, and hence a map to the space $\PreBiSp(\cM_{\cA},2-d)$ from Definition \ref{PreBiSphgsdef}, and our first required statement follows by taking path components.

It remains to show that whenever $\alpha \in \HN_{R,d}(\cA)$ is non-degenerate, the resulting pre-bisymplectic structure $\omega \in\PreBiSp(\cM_{\cA},2-d)$ is bisymplectic.

Given the quasi-isomorphisms $\oR\HHom_{B^{\oL,e}}(\bL_{\Mor_{\cA},[E]},B^{\oL,e})  \simeq E'\ten_{\cA}^{\oL} \cA(-,-)^! \ten_{\cA}^{\oL}E[1] $ and $\bL_{\Mor_{\cA},[E]}  \simeq E'\ten_{\cA}^{\oL} \cA(-,-) \ten_{\cA}^{\oL}E[-1] $ above, the map $ \ev_{\bar{\alpha}}$ induces a natural quasi-isomorphism
$\sigma_{\bar{\alpha}} \co \oR\HHom_{B^{\oL,e}}(\bL_{\Mor_{\cA},[E]},B^{\oL,e})\to \bL_{\Mor_{\cA},[E]}[2-d]$ by tensoring with $E'$ and $E$.

The argument now proceeds exactly as in the proof of Theorem \ref{Perfthm1}; it suffices to show that 
\[
 \omega_2^{\sharp}\co  \oR\HHom_{B^{\oL,e}}(\oL\Omega^1_B,B^{\oL,e}) \to \oL\Omega^1_B[2-d]
\]
is naturally homotopic to the composition $\At^*_E \circ \sigma_{\alpha} \circ \At_E$, for the non-commutative Atiyah class $\At_E \co \oR\HHom_{B^{\oL,e}}(\oL\Omega^1_B,B^{\oL,e}) \to \oR\HHom_{B^{\oL,e}}(\bL_{\Mor_{\cA},[E]},B^{\oL,e})$, and its adjoint  $\At^*_E \co\bL_{\Mor_{\cA},[E]} \to  \oL\Omega^1_B$. Again, $\At_E$ factorises as the composite
\[
 \oR\HHom_{B^{\oL,e}}(\oL\Omega^1_B,B^{\oL,e}) \xra{u} \oR\HHom_{B^{\oL,e}}(B,B^{\oL,e})[1] \xra{i_E} \oR\HHom_{\cA^{\op} \ten_R^{\oL}B}(E, E\ten_R^{\oL}B)[1] 
\]
for $u$ as in Proposition \ref{HNsympprop1} and $i_E$ given by $E\ten^{\oL}_B-$. Similarly, we have $\At_E^* \simeq u^*\circ i_E^*$, where $i_E^* \co E'\ten_{\cA}^{\oL}E[-1] \to B[-1]$ is given by evaluating $E'$ on $E$. Thus
\[
 \At^*_E \circ \sigma_{\alpha} \circ \At_E \simeq u^*\circ (i_E^* \circ \sigma_{\alpha} \circ i_E) \circ u.
\]

Now, the image  $\bar{\omega}_2$ of $\omega_2$ in $\HH_{R,d}B$ is comes from the element $\bar{\alpha} \in \HH_{R,d}(\cA)$
via the transformation $   \HH_{R,d}(\cA) \to \HH_{R,d}B $ associated to the Morita morphism $E$. 
This transformation is given by the composite
\begin{align*}
 \cA(-,-)\ten^{\oL}_{\cA^{\oL,e}}\cA(-,-) &\to \cA(-,-)\ten^{\oL}_{\cA^{\oL,e}}\oR\End_B(E(-),E(-))\\
 &\simeq   \cA(-,-)\ten^{\oL}_{\cA^{\oL,e}}(B\ten^{\oL}_{B^{\oL,e}}(E\ten^{\oL}_RE'))\\
&\simeq (E'\ten^{\oL}_{\cA}E) \ten^{\oL}_{B^{\oL,e}}B\\
&\xra{i_E^* \ten \id} B\ten^{\oL}_{B^{\oL,e}}B.
\end{align*}

Thus the evaluation map $\ev_{\bar{\omega}_2} \co \oR\HHom_{B^{\oL,e}}(B,B^{\oL,e}) \to B$ 
 is given by $i_E^* \circ \sigma_{\alpha}\circ i_E$, so 
$ 
\omega_2^{\sharp} \simeq \At^*_E \co \sigma_{\alpha} \circ \At_E
$, 
as required.
\end{proof}

\begin{remark}\label{locsysrmk}
 Note that the datum of a non-degenerate element of negative cyclic homology $\HN_{R,d}(\cA)$ precisely corresponds to   a 
 a homotopy class of $(d+1)$-dimensional left Calabi--Yau structures in the terminology of \cite[Definition 3.5]{BravDyckerhoff}. Examples for $\Mor_{\cA}$ satisfying the conditions thus include  dg categories of $\infty$-local systems in perfect complexes on a $d$-dimensional oriented manifold.
 
 On restricting to commutative input, Theorem \ref{Morthm} recovers the $(2-d)$-shifted symplectic structure on $\Mor_{\cA}^{\comm}$ constructed in \cite[Theorem 5.5(1)]{BravDyckerhoffII}, via the map $\BiSp(\Mor_{\cA},2-d) \to \Sp(\Mor_{\cA}^{\comm},2-d)$ coming from Lemma \ref{commSplemma}.
%
\end{remark}

\begin{remark}
Where Theorem \ref{Perfthm2} obtained shifted bisymplectic structures from classes in cyclic cohomology, Theorem \ref{Morthm} produces them from classes in \emph{negative} cyclic homology. One interpretation for this is that  cyclic cohomology of $\cA$ parametrises $S^1$-equivariant maps from Hochschild homology of $\cA$ to that of our base ring $R$, while negative cyclic homology parametrises the $S^1$-equivariant maps going the other way.

These are the two extremes of the bifunctorial theory \cite{quillenBivariantHC} which associates to a pair $\cA,\cC$ of dg categories over $R$ the complex $(\HHom_R(\CC(\cA), \CC(\cC))\llb u \rrb, b+ud)$ of $S^1$-equivariant maps from Hochschild homology of $\cA$ to Hochschild homology of $\cC$. 
For $\cA$ locally proper and $\C$ smooth, we can regard a cohomology class in this complex as non-degenerate if it induces a quasi-isomorphism $\C(-,-)^!\ten^{\oL}_R\cA(-,-) \to \C(-,-)\ten^{\oL}_R\cA(-,-)^*[-d]$, and then there is a common generalisation of Theorem \ref{Perfthm2} and  Theorem \ref{Morthm}, with essentially the same proof, putting a $(2-d)$-shifted bisymplectic structure on the derived NC prestack which sends a DGAA $B$ to the space of derived Morita morphisms from $\C$ to $\cA\ten^{\oL}_RB$, or equivalently the space of $\C^{\op}\ten^{\oL}_R\cA\ten^{\oL}_RB$-modules which are perfect over $\cA\ten^{\oL}_RB $.
\end{remark}

Given a dg functor $\phi \co \cA \to  \cB$ between smooth dg categories over $R$, we now consider the cocone $\mathbf{HN}_{R,\bt}(\cA,\cB)$ of the natural map $\phi \co \mathbf{HN}_{R,\bt}(\cA) \to\mathbf{HN}_{R,\bt}(\cB) $ between the complexes calculating negative cyclic homology, and then look at the relative negative cyclic homology groups $\HN_{R,i}(\cA,\cB):=\H_i\mathbf{HN}_{R,\bt}(\cA,\cB)$.
 
Via the natural  morphism from negative cyclic homology to  Hochschild homology, there is a natural map
\[
\mathbf{HN}_{R,\bt}(\cA,\cB) \to \cocone(\CC_{R,\bt}(\cA) \to \CC_{R,\bt}(\cB)).
\]
Rewriting the right-hand side as $\cocone(\cA\ten^{\oL}_{\cA^{\oL,e}}\cA\to \cB\ten^{\oL}_{\cB^{\oL,e}}\cB)$, this in turn maps to $ \cocone(\cA\ten^{\oL}_{\cA^{\oL,e},\phi}\cB\to \cB\ten^{\oL}_{\cB^{\oL,e}}\cB)$, so 
an element of this complex defines a map from $\oR\HHom_{\cB^{\oL,e}}(\cB, \cB^{\oL,e})$ to $\cocone( \cA\ten_{\cA^{\oL,e},\phi}\cB^{\oL,e} \to \cB)$ by evaluation on the second factor.

The proof of Theorem \ref{Morthm} now readily adapts to give the following:
\begin{theorem}\label{MorthmLag}
 If $\phi \co \cA \to  \cB$ is a dg functor between smooth dg categories over $R$, then there is a natural map from the negative cyclic homology group $\HN_{R,d}(\cA,\cB)$ to equivalence classes $\pi_0\BiIso(\Mor_{\cA},\Mor_{\cB};2-d)$ of $(2-d)$-shifted bi-isotropic structure on the homogeneous  derived  NC prestack $\Mor_{\cB}$ over $\Mor_{\cA}$, lifting the $(2-d)$-shifted pre-bisymplectic structure on $\Mor_{\cA}$ given by applying Theorem \ref{Perfthm2} to the image of the natural map $ \HN_{R,d}(\cA,\cB) \to\HN_{R,d}(\cA)$.  

When the image of a class  $\alpha \in \HN_{R,d}(\cA,\cB)$ in $\HH_{R,d}(\cA,\cA^*)$ is non-degenerate (making $\Mor_{\cA}$ bisymplectic as in Theorem \ref{Morthm}), the resulting  shifted bi-isotropic structure on $\Mor_{\cB}$ over $\Mor_{\cA}$ is 
is bi-Lagrangian if and only if its image  under the natural map 
\[
 \HN_{R,d}(\cA,\cB) \to\Ext^{-d}_{\cB^{\oL,e}}(\oR\HHom_{\cB^{\oL,e}}(\cB, \cB^{\oL,e}), \cocone( \cA\ten_{\cA^{\oL,e},\phi}\cB^{\oL,e} \to \cB))
 \]
is non-degenerate, i.e. a quasi-isomorphism
\[
\cB^! \to  \cocone( \cA\ten_{\cA^{\oL,e},\phi}\cB^{\oL,e} \to \cB)[-d] 
\]
of $\cB^{\oL,e}$-modules.
\end{theorem}

\begin{remark}
An element of the negative cyclic homology group $\HN_{R,d}(\cA,\cB)$
 with the non-degeneracy properties above is precisely the datum of
 a homotopy class of $n$-dimensional left Calabi-Yau structures on the morphism  $\phi \co \cA \to  \cB$ in the terminology of \cite[Definition 4.11]{BravDyckerhoff}. 
 
 On restricting to commutative input, Theorem \ref{MorthmLag} recovers the $(2-d)$-shifted Lagrangian structure on $\Mor_{\cB}^{\comm}\to \Mor_{\cA}^{\comm}$ constructed in \cite[Theorem 5.5(2)]{BravDyckerhoffII}, via the  map $\BiLag(\Mor_{\cA},\Mor_{\cB}; 2-d) \to \Lag(\Mor_{\cA}^{\comm},\Mor_{\cB}^{\comm};2-d)$ from Lemma \ref{commSplemma}.
\end{remark}


\section{Shifted double Poisson structures}\label{poisssn}

\subsection{Shifted double Poisson structures on derived NC affines}\label{affpoisssn}

\subsubsection{Double \tps{$P_n$}{Pn}-algebras}

We now generalise the definitions of  \cite[\S 2.3]{vdBerghDoublePoisson} from algebras to DGAAs, incorporating shifts in the brackets.

\begin{definition}\label{nbracketdef}
Adapting \cite[Definition 2.2.1]{vdBerghDoublePoisson}, we say  that an $n$-shifted 
 $k$-bracket on an  $R$-DGAA $A$ is an  $R$-linear map
\[
\{\{-,\ldots, -\}\} \co  (A_{[-n-1]})^{\ten_R k} \to (A^{\ten_R k})_{[-n-1]} 
\]
which is a derivation $A \to  A^{\ten k}$ in its last argument for the outer bimodule structure
on $A^{\ten k}$ 
and which is cyclically anti-symmetric in the sense that
\[
\tau_{(1\ldots k)} \circ \{\{-,\ldots, -\}\}=  \mp   \{\{-,\ldots, -\}\} \circ \tau_{(1\ldots k)},
\]
where $\mp$ is the relevant Koszul sign, which is $(-1)^{k+1}$ when $n=0$ and  all elements of $A$ are of even degree, and $\tau_{(1\ldots k)}$ is the cyclic permutation.


\end{definition}

\begin{definition}
 Define a double $P_n$-algebra over $R$ to be an $R$-DGAA $A$ equipped with a $(n-1)$-shifted $2$-bracket $\{\{-,-\}\}$ satisfying the double Jacobi identity of \cite[\S 2.3]{vdBerghDoublePoisson}   
\end{definition}
Thus a double $P_1$-algebra concentrated in degree $0$ is just a double Poisson algebra in the sense of \cite[Definition 2.3.2]{vdBerghDoublePoisson}.

\begin{definition}\label{lodaydef}
 Given a double $P_n$-algebra $A$, define the $(n-1)$-shifted Loday bracket $\{-,-\}\co A \ten A \to A_{[n-1]}$ by composing the $2$-bracket $\{\{-,-\}\}$ with the multiplication map $A \ten A \to A$.
\end{definition}

The proof of \cite[Corollary 2.4.6]{vdBerghDoublePoisson} generalises to give the following:
\begin{lemma}\label{quotientDGLAlemma}
If $A$ is a double $P_n$-algebra, then the bracket $\{-,-\}$ makes the shifted quotient $(A/[A,A])_{[1-n]} $ by the commutator into a DG Lie algebra (DGLA). 
\end{lemma}

\subsubsection{Non-commutative polyvector fields}\label{polsn}

\begin{definition}\label{poldef}
Define the filtered cochain complex of $n$-shifted non-commutative multiderivations (or polyvectors) on an $R$-DGAA  $A$ by
\[
 F^i\widehat{\Pol}^{nc}(A,n):= \prod_{p \ge i}\HHom_{(A^{e})^{\ten p}}(((\Omega^1_{A/R})_{[-n-1]})^{\ten p},[A^{\ten (p+1)}]), 
\]
for $i \ge 0$,
where we follow \cite[Proposition 4.2.1]{vdBerghDoublePoisson} in writing $[A^{\ten (p+1)}]$ for the $A^{\ten p}$-bimodule given by $A^{\ten (p+1)}$ with the $i$th copy of $A$ acting on the $i$th copy of $A$ on the right and on the $(i+1)$th copy of $A$ on the left.
We then write  $\widehat{\Pol}^{nc}(A,n):= F^0\widehat{\Pol}^{nc}(A,n)$. 

There is an associative product on $\widehat{\Pol}^{nc}(A,n)$ given by setting  the product  of a $p$-derivation $\phi$ and a $q$-derivation $\psi$ to be the $(p+q)$-derivation: 
\[
(\phi \cdot \psi)(\alpha_1 \ten \ldots \ten \alpha_p \ten \beta_1 \ten \ldots \ten \beta_q) = \phi(\alpha_1 \ten \ldots \ten \alpha_p) \ten \psi(\beta_1 \ten \ldots \ten \beta_q) \in [A^{\ten (p+1)}]\ten_A[A^{\ten q+1}],
\]
where we regard $[A^{\ten (p+1)}]$ and $[A^{\ten q+1}]$ as $A$-bimodules via the outer action.
This respects the filtration in the sense that $(F^i\widehat{\Pol}^{nc}(A,n))\cdot (F^j\widehat{\Pol}^{nc}(A,n))\subset F^{i+j}\widehat{\Pol}^{nc}(A,n)$.
\end{definition}

\begin{definition}\label{polcycdef}
We  define the filtered cochain complex of $n$-shifted cyclic multiderivations  on an $R$-DGAA $A$ by
\[
 F^i\widehat{\Pol}^{nc}_{\cyc}(A,n):= \prod_{p \ge i} (\HHom_{(A^{e})^{\ten p}}(((\Omega_A^1)_{[-n-1]})^{\ten p}, \{A^{\ten p}\}))^{C_p},
\]
for $i \ge 1$, where (following the proof of \cite[Proposition 4.1.2]{vdBerghDoublePoisson}) $\{A^{\ten p}\}:= [A^{\ten (p+1)}]\ten_{A^{e}}A$ is $A^{\ten p}$ given the $A^{\ten p}$-bimodule structure for which the $i$th copy of $A$ acts on the right on the $i$th copy of $A$, and acts on the left on the $(i+1 \mod p)$th copy of $A$.
We then set $F^0\widehat{\Pol}^{nc}_{\cyc}(A,n):= A/[A,A] \oplus F^1\widehat{\Pol}^{nc}_{\cyc}(A,n)$.
\end{definition}
In particular note that  a cyclic $p$-derivation corresponds to a  $p$-bracket in the sense of Definition \ref{nbracketdef}. 

\begin{definition}\label{cyclictracedef} 
Adapting \cite[Proposition 4.1.1]{vdBerghDoublePoisson},   the trace isomorphism between cyclic invariants and coinvariants gives a natural filtered map
\[
\tr \co \widehat{\Pol}^{nc}(A,n)/[\widehat{\Pol}^{nc}(A,n),\widehat{\Pol}^{nc}(A,n)]\to  \widehat{\Pol}^{nc}_{\cyc}(A,n) 
\]
from the quotient by the commutator of the associative multiplication. This is a filtered quasi-isomorphism whenever $\Omega^1_A$ is perfect and cofibrant as an $A^{e}$-module, and in particular whenever $A$ is finitely presented and cofibrant.
\end{definition}


\begin{proposition}\label{bracketprop}
There is a natural  bracket $\{-,-\}$ making  $\widehat{\Pol}^{nc}_{\cyc}(A,n)^{[n+1]}$ into a differential graded Lie algebra (DGLA) over $R$, satisfying $\{F^i,F^j\} \subset F^{i+j-1}$. 

There  is also an
$R$-bilinear map 
\[
 \{-,-\}^{\smile} \co \widehat{\Pol}^{nc}_{\cyc}(A,n) \by \widehat{\Pol}^{nc}(A,n) \to \widehat{\Pol}^{nc}(A,n)^{[-n-1]} 
\]
which lifts $\{-,-\}$ in the sense that $\tr(\{\pi, \alpha\}^{\smile})=\{\pi, \tr\alpha\}$. This also satisfies   $\{F^i,F^j\}^{\smile} \subset F^{i+j-1}$, and is a derivation in  its second argument. Given   a  $p$-bracket $\pi$ and an element $a \in A$, the $(p-1)$-derivation $\{\pi, a\}^{\smile}$ is given by   $\pm\pi(-,-, \ldots, -,a)$.
\end{proposition}
\begin{proof}
Since the complex $\DDer_R(A,A^{e})$ of double derivations is isomorphic to $\HHom_{A^e}(\Omega^1_A,A^e)$, 
the associative multiplication defines a natural map from the completed tensor algebra of   $\DDer_R(A,A^{e})^{[n+1]}$ to the DGAA $\widehat{\Pol}^{nc}(A,n)$, and this is a quasi-isomorphism whenever $\DDer_R(A,A^{e})$ is perfect and cofibrant as an $A^e$-module. Adapting \cite[Proposition 1.6, \S 3.2]{vdBerghDoublePoisson}, this gives rise to a form of  completed double $P_{n+2}$-algebra structure on $\widehat{\Pol}^{nc}(A,n)$, with the double bracket $\{\{-,-\}\}$ taking values in the completion of   $\widehat{\Pol}^{nc}(A,n)\ten \widehat{\Pol}^{nc}(A,n)$ with respect to the filtration $F$. 
This double bracket combines with the associative multiplication to give a Loday bracket on $\widehat{\Pol}^{nc}(A,n)_{[-n-1]}$ which induces the maps $\{-,-\}^{\smile}
$ and $\{-,-\}$ by the same reasoning as  \cite[Lemma 2.4.1 and Corollary 2.4.6]{vdBerghDoublePoisson}. The filtration properties and expression for $\{\pi, a\}^{\smile}$ follow immediately from the explicit description in \cite[\S 3.2]{vdBerghDoublePoisson}.
%
%
\end{proof}

\subsubsection{Double Poisson structures}

We now extend the definition of double Poisson structures to incorporate higher homotopical information.

\begin{definition}\label{mcPLdef}
 Given a   DGLA $(L, \{-,-\})$, define the the Maurer--Cartan set by 
\[
\mc(L):= \{\omega \in  L^{1}\ \,|\, \delta\omega + \half\{\omega,\omega\}=0 \in  \bigoplus_n L^{2}\}.
\]

Following \cite{hinstack}, define the Maurer--Cartan space $\mmc(L)$ (a simplicial set) of a nilpotent  DGLA $L$ by
\[
 \mmc(L)_k:= \mc(L\ten_{\Q} \Omega^{\bt}(\Delta^k)),
\]
for the differential graded commutative algebras $ \Omega^{\bt}(\Delta^n)$ of de Rham polynomial forms on the $k$-simplex, as in Definition \ref{PreSpdef}.
\end{definition}

\begin{definition}
Given an inverse system $L=\{L_{\alpha}\}_{\alpha}$ of nilpotent DGLAs, define
\[
 \mc(L):= \Lim_{\alpha} \mc(L_{\alpha}) \quad  \mmc(L):= \Lim_{\alpha} \mmc(L_{\alpha}).
\]
Note that  $\mc(L)= \mc(\Lim_{\alpha}L_{\alpha})$, but $\mmc(L)\ne \mmc(\Lim_{\alpha}L_{\alpha}) $. 
\end{definition}

\begin{definition}\label{poissdef}
Define an $R$-linear  $n$-shifted double Poisson structure on $A$ to be an element of
\[
 \mc(F^2 \widehat{\Pol}^{nc}_{\cyc}(A,n)^{[n+1]}), 
\]
 and the space $\cD\cP(A,n)$ of $R$-linear  $n$-shifted double Poisson structures on $A$ to be given by the simplicial 
set
\[
 \cD\cP(A,n):= \mmc( \{F^2 \widehat{\Pol}^{nc}_{\cyc}(A,n)^{[-n-1]}/F^{i+2}\}_i).
\]

Also write $\cD\cP(A,n)/F^{i+2}:= \mmc(F^2 \widehat{\Pol}^{nc}_{\cyc}(A,n)^{[-n-1]}/F^{i+2})$, so $\cD\cP(A,n)= \Lim_i \cD\cP(A,n)/F^{i+2}$.
\end{definition}

\begin{remark}\label{DPinterpretrmk}
Observe that elements of $\cD\cP_0(A,n)= \mc(F^2 \widehat{\Pol}^{nc}_{\cyc}(A,n)^{[n+1]})$ correspond 
to infinite sums $\pi = \sum_{i \ge 2}\pi_i$ with $\pi_i$ a cyclic shifted $i$-bracket, 
 satisfying $\delta(\pi_i) + \half \sum_{j+k=i+1} \{\pi_j,\pi_k\}=0$. 
This is a higher generalisation of the double Jacobi identity, so $\pi$ defines a form of double $P_{n+1}$-infinity algebra structure on $A$, generalising  the double Poisson infinity-algebras of \cite[Definition 4.1]{schedlerPoissonYangBaxter} 
when $n=0$, the difference being that our operations are cyclic rather than symmetric. 

In the notation of 
\cite[Proposition 5.7 and Corollary 5.10]{lerayProtoperads2Koszul}, an $n$-shifted double Poisson structure in our sense is an algebra for the dg protoperad $\cA s\boxtimes_c^{\mathrm{Val}} s^n\mathrm{Ind}(\cD\cL ie)_{\infty}$, where $s$ denotes suspension. 

In particular, if $\pi_i=0$ for all $i \ge 3$, then $\pi_2$ gives $A$ the structure of a double $P_{n+1}$-algebra. In general, the $n$-shifted double Poisson structure $\{\pi_i\}_i$ always gives rise to an   $L_{\infty}$-algebra structure on the quotient $(A/[A,A])_{[-n]}$ by the commutator.
\end{remark}

\begin{remark}\label{poissoncommrmk}
 Observe that for the abelianisation functor $A \mapsto A^{\comm}$ from DGAAs to CDGAs of Definition \ref{commdefhere}, we have a natural map  $F^i\widehat{\Pol}^{nc}_{\cyc}(A,n) \to F^i\widehat{\Pol}(A^{\comm},n)$, for the complex $\widehat{\Pol}$ of polyvectors from \cite{poisson}. This map is compatible with the Lie bracket, so gives a natural map $\cD\cP(A,n) \to  \cP(A^{\comm},n)$ from the space of $n$-shifted double Poisson structures on $A$ to the space of $n$-shifted  Poisson structures on $A^{\comm}$.

Moreover, we will see in Remark \ref{CostelloRozenblyumRmk} that for any Frobenius algebra $S$, there is a natural map $\cD\cP(A,n) \to  \cD\cP(\Pi_{S/R}A,n)$, giving $n$-shifted double Poisson structures on the Weil restriction of scalars $\Pi_{S/R}A$ of \cite[\S \ref{NCstacks-weilrestrn}]{NCstacks}; in particular, this applies when $S$ is the matrix algebra $\Mat_k$. As above, this in turn maps to the space $ \cP((\Pi_{S/R}A)^{\comm},n)$, which when $S=\Mat_k$ gives us shifted Poisson structures on the affine scheme representing framed rank $k$ representations of $A$.
\end{remark}

\begin{remark}[Quantisation]
 It is natural to ask whether there is any notion of quantisation for shifted double Poisson structures, such that a quantisation on $A$ gives rise to a quantisation of the induced shifted Poisson structure on the commutative quotient $A^{\comm}$. However, it is not at all clear what form such quantisations could take. For instance, in the $n=0$ case we would have to induce a non-commutative associative deformation of $A^{\comm}$, and in the $n=-1$ case,  a BV algebra deformation.

An answer of sorts might be provided by  wheelgebras as in \cite{GinzburgSchedlerDiffOps}, where  to any DGAA there is associated a commutative algebra $\cF(A)$ in the monoidal category of wheelspaces, giving rise to a notion of non-commutative differential operators, and hence of BV algebras. In \cite[Remark 3.5.19]{GinzburgSchedlerDiffOps}, it is observed that every double Poisson algebra gives rise to the structure of a wheeled Poisson algebra on $\cF(A)$, i.e. a Poisson algebra in the category of wheelspaces, though not all wheeled Poisson structures arise in this way. The same will automatically be true for double $P_k$-algebras, and it seems reasonable to expect that the construction will extend to shifted Poisson structures, with strong homotopy  double $P_k$-algebras $A$  giving rise to strong homotopy $P_k$-algebras $\cF(A)$ in wheelspaces.  

For $n>0$, this would automatically give quantisations of $n$-shifted double Poisson structures, in the form of $E_{n+1}$-algebra (or, strictly speaking, $BD_{n+1}$-algebra) deformations of $\cF(A)$, by the operadic formality equivalence $E_{n+1}\simeq P_{n+1}$, or more precisely by the filtered version $BD_{n+1}\simeq P_{n+1}\brh$. For the other cases with established quantisations ($n=0,-1,-2$), there would be significantly more work involved in adapting commutative arguments to the wheelgebra setting. In all cases, there is also the drawback that the space of wheeled Poisson structures is much larger than the space of double Poisson structures.
\end{remark}

Now, given a $k$-shifted cyclic $2$-bracket 
 $\pi_2 \in  \z^0(\HHom_{(A^{\oL,e})^{\ten 2}}((\Omega_A^1)^{\ten 2}, \{A^{\ten 2}\}_{[k]}))^{C_2}$ as in Definition \ref{nbracketdef}, there is an associated morphism
\[
 \pi_2^{\flat} \co \Omega^1_{A/R} \to \DDer_R(A,A^{\oL,e})_{[k]}
\]
given by $\pi_2^{\flat}(\omega)(f) := \pi_2(\omega \ten df)$. The cyclic symmetry of $\pi_2$ then ensures that this descends to a map
\[
 \pi_2^{\flat} \co \Omega^1_{A/R,\cyc} \to \DDer_R(A,A)_{[k]}
\]
by inner multiplication.

\begin{definition} \label{nondegdef}
For an $R$-DGAA $A$ with $A^e \simeq A^{\oL,e}$ (e.g. $A$ cofibrant) for which   $\Omega^1_{A/R} $ is perfect as an $A^{e}$-module,  we say that an  $n$-shifted double Poisson structure $\pi = \sum_{i \ge 2}\pi_i $ is non-degenerate if $\pi_{2} \co ((\Omega^1_{A/R})_{[-n-1]})^{\ten 2} \to (A^{\ten 2})_{[-n-2]}$ 
induces a quasi-isomorphism 
\[
\pi_{2}^{\flat}\co  (\Omega^1_{A/R})_{[-n]} \to \DDer_R(A,A^{e}).
\]

We then define $\cD\cP(A,n)^{\nondeg}\subset \cD\cP(A,n)$ to consist of the non-degenerate elements --- this is a union of path-components.
\end{definition}

\subsubsection{The canonical tangent vector \tps{$\sigma$}{sigma}}

The space $\cD\cP(A,n) $ admits an action of $\bG_m(R_0)$, which is inherited from the scalar multiplication on $\Pol(A,n)$ in which $\HHom_{(A^{e})^{\ten p}}((\Omega_A^1)^{\ten p}, \{A^{\ten p}\})$
 is given weight $p-1$. Differentiating this action gives us a global tangent vector on $\cD\cP(A,n) $, as follows. 

Take $\eps$ to be a variable of degree $0$, with $\eps^2=0$. 
Observe that the DGLA structure on $\widehat{\Pol}^{nc}_{\cyc}(A,n)^{[-n-1]}$ automatically makes  $\widehat{\Pol}^{nc}_{\cyc}(A,n)^{[-n-1]}\ten_{\Q} \Q[\eps]$ a DGLA, and also that 
\[
\Pol_{\cyc}(A,n)\ten_{\Q} \Q[\eps] \cong (\Pol(A,n)\ten_{\Q} \Q[\eps])/[ \Pol(A,n)\ten_{\Q} \Q[\eps], \Pol(A,n)\ten_{\Q} \Q[\eps]]. 
\]

\begin{definition}\label{Tpoissdef} 
Define the tangent spaces
\begin{eqnarray*}
 T\cD\cP(A,n)&:=& \mmc( \{F^2 \widehat{\Pol}^{nc}_{\cyc}(A,n)^{[-n-1]}\ten_{\Q} \Q[\eps]/F^{i+2}\}_i)\\
T\cD\cP(A,n)/F^{i+2}&:=& \mmc( F^2 \widehat{\Pol}^{nc}_{\cyc}(A,n)^{[-n-1]}\ten_{\Q} \Q[\eps]/F^{i+2}).
\end{eqnarray*}
 \end{definition}
These are simplicial sets over $\cD\cP(A,n)$ (resp. $\cD\cP(A,n)/F^{i+2}$), fibred in simplicial abelian groups. 

\begin{definition}\label{cPdef} 
Given $\pi \in  \cD\cP_0(A,n)$, observe that the Maurer--Cartan condition implies that $\delta+\{\pi,-\}^{\smile}$ defines a square-zero derivation on $\widehat{\Pol}^{nc}(A,n) $, 
and denote the resulting complex by 
\[
\widehat{\Pol}^{nc}_{\pi}(A,n).
\]
The product and double bracket on non-commutative polyvectors make this a form of completed double  $P_{n+2}$-algebra, and it inherits the filtration $F$. Write $\widehat{\Pol}^{nc}_{\pi,\cyc}(A,n) := (\widehat{\Pol}^{nc}_{\cyc}(A,n),\delta+\{\pi,-\})$ for the quotient of $\Pol_{\pi}(A,n)$ by the commutator.

Given $\pi \in  \cD\cP_0(A,n)/F^p$, we define $\widehat{\Pol}^{nc}_{\pi}(A,n)/F^p$ similarly. This is a DGAA, and $ F^1\Pol_{\pi}(A,n)/F^p$ is a form of completed double $P_{n+2}$-algebra, because $F^i\cdot F^j \subset F^{i+j}$ and $\{\{F^i,F^j\}\} \subset F^{i+j-1}$.
\end{definition}

The following is a standard consequence of the map
\[
 \widehat{\Pol}^{nc}_{\cyc}(A,n)^{[-n-1]}\ten_{\Q} \Q[\eps]\to \widehat{\Pol}^{nc}_{\cyc}(A,n)^{[-n-1]}
\]
being a square-zero extension of DGLAs:
\begin{lemma}
 The  fibre $T_{\pi}\cD\cP(A,n)$ of $T\cD\cP(A,n)$  over $\pi \in \cD\cP(A,n)$ is canonically homotopy equivalent to the Dold--Kan denormalisation of the good truncation   $\tau^{\le 0} (F^2\widehat{\Pol}^{nc}_{\pi,\cyc}(A,n)^{[n+2]})$. 
In particular, its
 homotopy groups  are given by
\[
 \pi_iT_{\pi}\cD\cP(A,n)= \H^{n+2-i}(F^2 \widehat{\Pol}^{nc}_{\cyc}(A,n), \delta +\{\pi,-\}).
\]
The corresponding statements for $T_{\pi}\cD\cP(A,n)/F^{i+2}$ also hold.
\end{lemma}

Now observe that the map
\begin{align*}
 \id +\sigma\eps \co F^2 \widehat{\Pol}^{nc}_{\cyc}(A,n)^{[-n-1]} &\to F^2 \widehat{\Pol}^{nc}_{\cyc}(A,n)^{[-n-1]}\ten_{\Q} \Q[\eps]\\
 \sum_{i \ge 2} \alpha_i  &\mapsto \sum_{i \ge 2} (\alpha_i+ (i-1)\alpha_i\eps)
\end{align*}
is a morphism of filtered DGLAs, for $\alpha_i \co (\Omega_A^1)^{\ten p} \to \{A^{\ten p}\}$. This can be seen either by direct calculation or by observing that $\sigma$ arises by differentiating the $\bG_m$-action on $\Pol^{nc}$.

\begin{definition}\label{sigmadef}
 Define the canonical tangent vector $(\id,\sigma) \co \cD\cP(A,n) \to  T\cD\cP(A,n)$ on the space of $n$-shifted Poisson structures by applying $\mmc$ to the morphism $\id +\sigma\eps$ of DGLAs.
\end{definition}

Explicitly, this sends $\pi= \sum \pi_i $ to $\sigma(\pi)=\sum_{i \ge 2} (i-1)\pi_i \in T_{\pi}\cD\cP(A,n)$, which thus has the property that $\delta \sigma(\pi) +\{\pi, \sigma(\pi)\}=0$.

The map $(\id,\sigma)$ preserves the cofiltration in the sense that it comes from a system of  maps $(\id,\sigma) \co  \cD\cP(A,n)/F^{i+2} \to  T\cD\cP(A,n)/F^{i+2} $.

\subsection{Comparison with bisymplectic structures on derived  NC affines}\label{equivaffinesn}

\subsubsection{Compatible pairs}\label{compsn}

We will now develop the notion of compatibility between a pre-bisymplectic structure and a double Poisson structure, analogous to \cite[\S \ref{poisson-compsn}]{poisson}.  

\begin{definition}\label{mudef}
Given  $\pi \in (F^2\widehat{\Pol}^{nc}_{\cyc}(A,n)/F^{p})^{n+2}$, define 
\[
 \mu(-,\pi) \co \DR(A/R)/F^{p} \to \widehat{\Pol}^{nc}(A,n)/F^{p}
\]
to be the  morphism of graded algebras given on generators $a,df$ for $a,f \in A$  by 
\begin{align*}
\mu(a,\pi)&:= a,\\
 \mu(df, \pi)&:=\{\pi,f\}^{\smile}, 
\end{align*}
for the bilinear map $\{-,-\}^{\smile} \co \widehat{\Pol}^{nc}_{\cyc}(A,n) \by \widehat{\Pol}^{nc}(A,n) \to \widehat{\Pol}^{nc}(A,n)_{[n+1]}$ of Proposition \ref{bracketprop} lifting the Lie bracket on $\widehat{\Pol}^{nc}_{\cyc}(A,n)$.

Given $b \in \widehat{\Pol}^{nc}(A,n)/F^{p}$, we then  define 
\[
 \nu(-, \pi, b) \co \DR(A/R)/F^{p} \to \widehat{\Pol}^{nc}(A,n)/F^{p}
\]
by setting $\mu(\omega,\pi +b\eps)= \mu(\omega,\pi)+ \nu(-, \pi, b)\eps$ for $\eps^2=0$. More explicitly, $\nu(-, \pi, b)$ is  the $A$-linear derivation with respect to the ring homomorphism $ \mu(-,\pi) $ given on generators 
 $a,df$ for $a,f \in A$  by 
\begin{align*}
\nu(a,\pi,b) &:=0\\
 \nu(df, \pi, b)&:= \{b,f\}^{\smile}. 
\end{align*}
\end{definition}
To see that these are well-defined in the sense that they descend to the quotients by $F^{p}$, observe that because $\pi \in F^2$, contraction has the property that
\[
 \mu(\Omega^1, \pi) \subset  F^1, \quad \nu(\Omega^1, \pi, b) \subset F^1,
\]
it follows that $\mu(F^p, \pi)\subset F^p$ and $\nu(F^p,\pi,b)   \subset F^p$ by multiplicativity. 

Moreover, note that since $\mu$ is a ring homomorphism and $\nu$ a derivation, we may quotient by commutators and apply the trace map of Definition \ref{cyclictracedef}
  to obtain filtered maps
\begin{align*}
 \tr \mu(-,\pi) \co \DR_{\cyc}(A/R)/F^{p} &\to \widehat{\Pol}^{nc}_{\cyc}(A,n)/F^{p}\\ 
\tr \nu(-,\pi,b) \co \DR_{\cyc}(A/R)/F^{p} &\to \widehat{\Pol}^{nc}_{\cyc}(A,n)/F^{p}.
\end{align*}


\begin{lemma}\label{keylemma} 
 For $\omega \in \DR(A)/F^p$ and $\pi \in F^2\widehat{\Pol}^{nc}_{\cyc}(A,n)^{n+2}/F^p$, we have
\begin{align*}
\{\pi,\mu(\omega, \pi)\}^{\smile} = \mu(d\omega, \pi) + \half \nu(\omega, \pi, \{\pi,\pi\}),\\
\delta_{\pi}\mu(\omega, \pi) = \mu(D\omega, \pi) + \nu(\omega, \pi,\kappa(\pi)),
\end{align*}
where $\delta_{\pi}= \{\delta + \pi,-\}^{\smile}$ lifts the differential on  $T_{\pi}\cD\cP(A,n)/F^{p}$, with  $D= d \pm \delta$  the total differential on $(F^2\DR(A)/F^p) $, and $ \kappa(\pi)=\{\delta, \pi\} + \half \{\pi,\pi\}$. 
\end{lemma}
\begin{proof}
This follows with exactly the same proof as \cite[Lemma \ref{poisson-keylemma}]{poisson}, once we observe that $\{-,-\}^{\smile}$ being a Loday bracket implies that $\{\pi, \{\pi,f\}^{\smile}\}^{\smile}=\half \{\{\pi,\pi\},f\}^{\smile}$ (cf. \cite[Corollary 2.4.4]{vdBerghDoublePoisson}).
\end{proof}

In particular, this implies that when $\pi$ is Poisson, $\mu(-,\pi)$ defines a map from non-commutative de Rham cohomology to non-commutative Poisson cohomology, and on the associated cyclic cohomology theories.

The proof of \cite[Lemma \ref{poisson-mulemma}]{poisson} now adapts to give the following. 
\begin{lemma}\label{mulemma}
There  are  maps
\[
 (\pr_2 + \tr \mu\eps) \co  \PreBiSp(A,n)/F^{p} \by \cD\cP(A,n)/F^{p} \to T\cD\cP(A,n)/F^{p}
\]
 over $\cD\cP(A,n)/F^{p}$ for all $p$, compatible with each other.
In particular, we have
\[
 (\pr_2 + \mu\eps)\co  \PreBiSp(A,n) \by \cD\cP(A,n) \to T\cD\cP(A,n).
\]
\end{lemma}

\begin{example}\label{compatex1} 
When $\omega =a (dx)b(dy)c$, note that  $\mu(\omega, \pi)= \pm a\{\pi,x\}^{\smile}\,b\{\pi,y\}^{\smile}\,c$. If $\pi=\pi_2$, then this is an element of $\HHom_{(A^{e})^{\ten 2}}(((\Omega^1_{A/R})_{[-n-1]})^{\ten 2},[A^{\ten 3}])^{[n+2]}$.

The morphism
\[
(\tr  \mu(\omega, \pi))^{\flat} \co (\Omega^1_{A/R})_{[-n]}\to \DDer_R(A)
\] 
induced by the associated cyclic double bracket
$
\tr  \mu(\omega, \pi)
$
is then
given by 
\[
 (\tr \mu(\omega, \pi))^{\flat} = \pi^{\flat} \circ \omega^{\sharp} \circ \pi^{\flat}.
\]
\end{example}
\begin{proof}
The double bracket associated to $\mu(\omega,\pi)$ is given in Sweedler notation by 
\begin{align*}
 (\tr \mu(\omega,\pi))(f,g)=  \pm &\{\pi,y\}^{\smile}(g)''ca\{\pi,x\}^{\smile}(f)'\ten \{\pi,x\}^{\smile}(f)''b\{\pi,y\}^{\smile}(g)' \\&\pm\{\pi,x\}^{\smile}(g)''b\{\pi,y\}^{\smile}(f)'\ten \{\pi,y\}^{\smile}(f)''ca\{\pi,x\}^{\smile}(g)',
\end{align*}

so the associated map 
\[
(\tr \mu(\omega, \pi))^{\flat} \co \Omega^1_{A/R}\to \DDer_R(A,A^{e})^{[-n]} 
\]
is determined on generators $df$ by
\begin{align*}
 (\tr \mu(\omega, \pi))^{\flat}(df)=\pm &(\{\pi,y\}^{\smile})''ca\{\pi,x\}^{\smile}(f)'\ten \{\pi,x\}^{\smile}(f)''b(\{\pi,y\}^{\smile})'\\ &\pm(\{\pi,x\}^{\smile})''b\{\pi,y\}^{\smile}(f)'\ten \{\pi,y\}^{\smile}(f)''ca(\{\pi,x\}^{\smile})'.
\end{align*}

Now, since $\{\pi,a\}^{\smile}= \pi(-,a)$ for $a \in A$, the expression above becomes
\[
  (\tr \mu(\omega, \pi))^{\flat}(df)=
 \pm \pi(-,y)''ca\pi(f,x) b\pi(-,y)' \pm\pi(-,x)''b\pi(f,y)ca\pi(-,x)'. 
\]

Meanwhile, $\pi^{\flat}(df ) = \pi(f,-)$, so $(\omega^{\sharp}\circ \pi^{\flat})(df )= \iota_{ \pi(f,-) }(\omega)$. Now, 
\[
i_{ \pi(f,-)}( a (dx) b (dy) c)=\pm a\pi(f,x)b (dy) c \pm a(dx)b\pi(f,y)c,
\]
so 
\[
 \iota_{\pi(f,-)}(\omega)= \pm \pi(f,x)''b(dy)ca\pi(f,x)' \pm \pi(f,y)''ca(dx)b\pi(f,y)',
\]
 and then   
 \begin{align*}
 &(\pi^{\flat} \circ \omega^{\sharp} \circ \pi^{\flat})(df)\\
 &= \pi( (\pm \pi(f,x)''b(dy)ca\pi(f,x)' \pm \pi(f,y)''ca(dx)b\pi(f,y)')\ten-),
 \\
 &=\pm \pi(y,-)' ca\pi(f,x)'\ten  \pi(f,x)''b \pi(y,-)'' \pm  \pi(x,-)'b\pi(f,y)'\ten \pi(f,y)''ca\pi(x,-)'',
 \\
 &= \pm \pi(-,y)'' ca\pi(f,x)b \pi(-,y)' \pm  \pi(-,x)''b\pi(f,y)ca\pi(-,x)'\\
 &= (\tr \mu(\omega, \pi))^{\flat}(df).\qedhere
  \end{align*}
\end{proof}

\begin{definition}
We say that an $n$-shifted  pre-bisymplectic structure $\omega$ and a double $n$-Poisson structure $\pi$ are  compatible (or a compatible pair) if 
\[
 [\mu(\omega, \pi)] = [\sigma(\pi)] \in  \H^{n+2}(F^2\widehat{\Pol}^{nc}_{\cyc,\pi}(A,n)) =\pi_0T_{\pi}\cD\cP(A,n),
\]
where $\sigma$ is the canonical tangent vector of Definition \ref{sigmadef}. 
\end{definition}

\begin{example}\label{compatex2}
Following Example \ref{compatex1}, if $\omega_i=0$ for $i >2$ and $\pi_i=0$ for $i>2$, then $(\omega, \pi)$ are a compatible pair
if and only if the map
\[
 \pi^{\flat} \circ \omega^{\sharp} \circ \pi^{\flat} \co (\Omega^1_{A/R})_{[-n]}\to \DDer_R(A)
\]
is homotopic to $\pi^{\flat}$, because $\sigma(\pi)=\pi$ in this case. 

In particular, if $\pi$ is non-degenerate, this means that $\omega$ and $\pi$ determine each other up to homotopy. 
\end{example}

\begin{lemma}\label{compatnondeg}
If  $(\omega, \pi)$ is a compatible pair and $\pi$ is non-degenerate, then $\omega$ is bisymplectic.
\end{lemma}
\begin{proof}
Even when the vanishing conditions of Example \ref{compatex1} are not satisfied, we still have
\[
 \pi^{\flat}_2 \circ \omega^{\sharp}_2 \circ \pi^{\flat}_2 \simeq \pi^{\flat}_2,
\]
so if $\pi^{\flat}_2$ is a quasi-isomorphism, then  $\omega^{\sharp}_2$ must be its homotopy inverse.
\end{proof}

\begin{definition}\label{vanishingdef}
Given a simplicial set $Z$, an abelian group object $A$ in simplicial sets over $Z$, and a section $s \co Z \to A$, define the homotopy vanishing locus of $s$ to be the homotopy limit of the diagram
\[
\xymatrix@1{ Z \ar@<0.5ex>[r]^-{s}  \ar@<-0.5ex>[r]_-{0} & A \ar[r] & Z}.
\]
\end{definition}

We can write this as a homotopy fibre product $Z \by_{(s,0), A \by^h_Z A}^hA$, for the diagonal map $A \to A \by^h_Z A$. When $A$ is a trivial bundle $A = Z \by V$, for $V$ a simplicial abelian group, note that the homotopy vanishing locus is just the homotopy fibre of $s \co Z \to V$ over $0$. 

\begin{definition}\label{compdef}
Define the space $\cD\Comp(A,n)$ of compatible $n$-shifted pairs to be the homotopy vanishing locus of
\[
(\tr \mu) - (\id,\sigma) \co \PreBiSp(A,n) \by \cD\cP(A,n) \to \pr_2^*T\cD\cP(A,n) =\PreBiSp(A,n) \by T\cD\cP(A,n).
\]

We define a cofiltration on this space by setting $ \cD\Comp(A,n)/F^{p}$ to be the homotopy
vanishing locus of
\[
 (\tr \mu) - (\id,\sigma) \co \PreBiSp(A,n)/F^{p} \by \cD\cP(A,n)/F^{p} \to \pr_2^*T\cD\cP(A,n)/F^{p}.
\]
\end{definition}

We can rewrite $\cD\Comp(A,n)$ as the homotopy limit of the diagram
\[
\xymatrix@C=8ex{  \PreBiSp(A,n) \by \cD\cP(A,n) \ar@<0.5ex>[r]^-{(\pr_2 + \mu\eps)}  \ar@<-0.5ex>[r]_-{ \pr_2 +\sigma\pr_2\eps} & T\cD\cP(A,n) \ar[r] & \cD\cP(A,n) }
\]
of simplicial sets. 

In particular, an element of this space is given by a pre-bisymplectic structure $\omega$, a double Poisson structure $\pi$, and a homotopy $h$ between $\mu(\omega,\pi)$ and $\sigma(\pi)$ in $T_{\pi}\cD\cP(A,n)$.
%

\begin{definition}
 Define $\cD\Comp(A,n)^{\nondeg} \subset \cD\Comp(A,n)$ to consist of compatible pairs with $\pi$ non-degenerate. This is a union of path-components, and by Lemma \ref{compatnondeg} has a natural projection 
\[
 \cD\Comp(A,n)^{\nondeg}\to \BiSp(A,n)
\]
as well as the canonical map
\[
 \cD\Comp(A,n)^{\nondeg} \to\cD\cP(A,n)^{\nondeg}.
\]
\end{definition}

\begin{proposition}\label{compatP1}
The canonical map
\begin{eqnarray*}
    \cD\Comp(A,n)^{\nondeg} \to  \cD\cP(A,n)^{\nondeg}           
\end{eqnarray*}
 is a weak equivalence.
\end{proposition}
\begin{proof}
For any $\pi \in \cD\cP(A,n)$, the homotopy fibre of $\cD\Comp(A,n)^{\nondeg} $ over $\pi$ is just the homotopy fibre of
\[
\tr \mu(-,\pi)  \co \PreBiSp(A,n)  \to T_{\pi}\cD\cP(A,n) 
\]
over $\sigma(\pi)$.

The map $\tr \mu(-,\pi) \co \DR_{\cyc}(A/R) \to (\widehat{\Pol}^{nc}_{\cyc}(A,n), \delta_{\pi})$ is a morphism of complete filtered complexes by Lemma \ref{keylemma}, and non-degeneracy of $\pi_2$ implies that we have a quasi-isomorphism on the associated gradeds 
\begin{align*}
\gr_F^p \tr \mu(-,\pi)\co &(\Omega_A^p\ten_{A^{e}}A)/C_p \simeq  (\HHom_{(A^{e})^{\ten p}}((\Omega_A^1)^{\ten p}, \{A^{\ten p}\})^{[-pn]})^{C_p}\\
\gr_F^0 \tr \mu(-,\pi)\co & A/[A,A]= A/[A,A].
\end{align*}

We therefore have a quasi-isomorphism of filtered complexes, so we have isomorphisms on homotopy groups:
\begin{eqnarray*}
 \pi_j\PreBiSp(A,n)  &\to& \pi_jT_{\pi}\cD\cP(A,n)\\
 \H^{n+2-j}(F^2 \DR_{\cyc}(A/R)) &\to&  \H^{n+2-j}(F^2\widehat{\Pol}^{nc}_{\cyc}(A,n), \delta_{\pi}).\qedhere
\end{eqnarray*}
\end{proof}

\subsubsection{Towers of obstructions}\label{towersn}

We will now set about showing that the tower $\cD\Comp(A,n) \to \ldots \to \cD\Comp(A,n)/F^{i+2} \to \ldots \to \cD\Comp(A,n)/F^2 $ does not contain nearly as much information as first appears.

The long exact sequence of cohomology yields the following:
\begin{proposition}\label{DRobs}
For each $p$, there is a canonical long exact sequence
\begin{align*}
\ldots\to \H_{p+i-n-2}(\Omega^{p}_{A/R,\cyc}) \to &\pi_i(\PreBiSp/F^{p+1})\\
&\to  \pi_i(\PreBiSp/F^{p}) \to \H_{p+i-n-3}(\Omega^{p}_{A/R,\cyc}) 
\to\ldots 
\end{align*}
 of homotopy groups, where $\PreBiSp=\PreBiSp(A,n)$.
\end{proposition}

Standard DGLA obstruction theory as in \cite[Corollary \ref{poisson-obsDGLAcor}]{poisson} also gives:
\begin{proposition}\label{poissonobs}
 For each $p \ge 2$, there is a canonical long exact sequence
 \begin{align*}
\ldots\to \H^{-i-n}\gr_F^p\widehat{\Pol}^{nc}_{\cyc}(A,n)\to &\pi_i(\cD\cP(A,n)/F^{p+1})\\
&\to  \pi_i(\cD\cP(A,n)/F^{p}) \to \H^{1-i-n}\gr_F^p\widehat{\Pol}^{nc}_{\cyc}(A,n)
\to\ldots 
\end{align*}
 of homotopy groups and sets, where $\pi_i$ indicates the homotopy group at any basepoint. 
\end{proposition}



We now establish the corresponding statements for  the space $\cD\Comp(A,n)$ of compatible pairs.

\begin{definition}\label{Mdef}
Given a compatible pair  $(\omega, \pi) \in \cD\Comp(A,n)/F^3$ and $p \ge 0$, define the cochain complex 
$
 M(\omega,\pi,p) 
$
to be the cocone of the map 
\begin{align*}
 (\Omega^{p}_{A,\cyc})^{[n-p+1]} \oplus &(\HHom_{(A^{e})^{\ten p}}((\Omega_A^1)^{\ten p}, \{A^{\ten p}\})^{[-(p-1)(n+1)]})^{C_p}
\\
 &\to(\HHom_{(A^{e})^{\ten p}}((\Omega_A^1)^{\ten p}, \{A^{\ten p}\})^{[-(p-1)(n+1)]})^{C_p}
\end{align*}
given by combining                                
\[
 (\pi^{\flat})^{\ten p}\co (\Omega^{p}_{A,\cyc})\to (\HHom_{(A^{e})^{\ten p}}((\Omega_A^1)^{\ten p}, \{A^{\ten p}\})^{[-np]})^{C_p}
\]
with the map 
\[
\tr  \nu(\omega, \pi) - (p-1) \co(\HHom_{(A^{e})^{\ten p}}((\Omega_A^1)^{\ten p}, \{A^{\ten p}\})^{[-np]})^{C_p} \to (\HHom_{(A^{e})^{\ten p}}((\Omega_A^1)^{\ten p}, \{A^{\ten p}\})^{[-np]})^{C_p}
\]
where                              
\[
 \nu(\omega, \pi)(\phi):= \nu(\omega, \pi, \phi).
\]
\end{definition}
%

\begin{lemma}\label{nondegtangent}
 If $\pi$ is non-degenerate, then the projection
\[
 M(\omega,\pi,p) \to (\Omega^{p}_{A,\cyc})^{[n-p+1]}
\]
is a quasi-isomorphism.
\end{lemma}
\begin{proof} By linearity, it suffices to consider the case $\omega =a(dx)b(dy)c$. 
Then for any $p$-bracket $\phi$ and any $x \in A$ 
we 
have
\[
 \nu(\omega,\pi)(\phi)= \pm a \{\phi,x\}^{\smile}b\{\pi,y\}^{\smile}c \pm a \{\pi,x\}^{\smile}b\{\phi,y\}^{\smile}c. 
\]
For the double bracket $\{\{-,-\}\}$ of Proposition \ref{bracketprop}, the trace $\tr$ of Definition \ref{cyclictracedef}, and any $\theta \in \widehat{\Pol}^{nc}(A,n)$, the element $ \nu(\omega,\pi)(\tr \theta) $ of  $\widehat{\Pol}^{nc}(A,n)$ is then cyclically equivalent to
\[
\tilde{\nu}(\omega, \pi)(\theta):= \pm \{\{x,\theta\}\}'b\{\pi,y\}^{\smile}ca\{\{x,\theta\}\}''  \pm \{\{y,\theta\}\}'ca\{\pi,x\}^{\smile}b\{\{y,\theta\}\}'', 
\]
and this expression is obviously a derivation in $\theta$, in the sense that the map $\tilde{\nu}(\omega, \pi)$ is a derivation on $ \widehat{\Pol}^{nc}(A,n)$ with respect to the associative product.

In particular,  the derivation $\tilde{\nu}(\omega, \pi)$ on $\widehat{\Pol}^{nc}(A,n)$ is given on generators $\theta\in 
\HHom_{A^{e}}((\Omega^1_{A})_{[-n-1]},\{A^{\ten 2}\})$ by 
\[
 \theta \mapsto \pm \theta(x)' b\pi(-,y)ca\theta(x)'' \pm  \theta(y)' ca \pi(-,x)b\theta(x)''.
\]

Meanwhile, if $\tau$ denotes the isomorphism $A^{e} \cong \{A^{\ten 2}\}$ given by transposing factors, we have
\[
 \omega^{\sharp}(\tau\theta)= \iota_{\tau\theta}(\omega)= \pm \theta(x)'b(dy)ca\theta(x)'' \pm \theta(y)'ca(dx)b\theta(y)'',
\]
so 
\begin{align*}
 \tau\pi^{\flat}(\omega^{\sharp}(\tau\theta))&= 
 \pm \theta(x)'b\pi^{\flat}(dy)ca\theta(x)'' \pm \theta(y)'ca\pi^{\flat}(dx)b\theta(y)''\\
&= \pm \theta(x)'b \pi(-,y)ca\theta(x)'' \pm \theta(y)'ca\pi(-,x)b\theta(y)''. 
\end{align*}

Thus  $\nu(\omega, \pi)\circ \tr$ is  cyclically equivalent to the derivation
\[
\tilde{\nu}(\omega, \pi) \co  \widehat{\Pol}^{nc}(A,n) \to \widehat{\Pol}^{nc}(A,n)
\]
given on generators by $\tau \circ \pi^{\flat} \circ \omega^{\sharp}\circ \tau$.
If $\pi$ is non-degenerate, then (by compatibility) $\pi^{\flat} \circ \omega^{\sharp}$ is homotopic to the identity map, and thus $ \gr_F^p\nu(\omega, \pi)_{\cyc}$ is homotopic to $p$. In particular, $\gr_F^p\nu(\omega, \pi)_{\cyc} - (p-1) $ is a quasi-isomorphism in this case, so the projection
$
 M(\omega,\pi,p) \to (\Omega^{p}_{A/R,\cyc})^{[n-p+1]}
$
is a quasi-isomorphism.
\end{proof}


The proof of \cite[Proposition \ref{poisson-compatobs}]{poisson} adapts verbatim to our non-commutative context, giving:
\begin{proposition}\label{compatobs}
For each $p \ge 3$, there is a canonical long exact sequence
\[
 \xymatrix@C=1ex{ 
\vdots  \ar[d]^-{e_*} && \vdots  \ar[d]^-{e_*} &&\\ \pi_i(\cD\Comp(A,n)/F^{p})  \ar[d]^-{o_e}&&\pi_1 (\cD\Comp(A,n)/F^{p}) \ar[d]^-{o_e}&& \pi_0(\cD\Comp(A,n)/F^{p}) \ar[d]^-{o_e} \\ \H^{2-i}M(\omega_2,\pi_2,p) \ar[d]&& \H^1 M(\omega_2,\pi_2,p) \ar[d] &&  \H^2 M(\omega_2,\pi_2,p) \\ \pi_{i-1}(\cD\Comp(A,n)/F^{p+1}) \ar[d]^-{e_*}&&\pi_0(\cD\Comp(A,n)/F^{p+1})  
\ar `d[dr]   `r[r]^{e_*}   `[uuur] `r[uuurr] [uurr] &&
\\ \vdots \ar@{-->} `r[ur]   `u[uuuurr]   [uuuurr] &&&&
}
\]
 of homotopy groups and sets, where $\pi_i$ indicates the homotopy group at basepoint $(\omega, \pi)$, and the target of the final map is understood to mean
\[                                                                                                                             
   {o_e}(\omega, \pi) \in \H^2 M(\omega_2,\pi_2,p).                                                                                                                                             \]
\end{proposition}

\subsubsection{The equivalence}\label{equivsn}

Substituting Lemma \ref{nondegtangent} and Proposition \ref{compatP1} for \cite[Lemma \ref{poisson-nondegtangent} and Proposition \ref{poisson-compatP1}]{poisson}, the proofs of \cite[Corollary \ref{poisson-compatcor1}, Proposition \ref{poisson-level0prop} and Corollary \ref{compatcor2}]{poisson} adapt to give:

\begin{proposition}\label{compatcor1}
The canonical map
\begin{eqnarray*}
  \cD\Comp(A,n)^{\nondeg} \to (\cD\Comp(A,n)/F^3)^{\nondeg} \by^h_{(\PreBiSp(A,n)/F^3)} \PreBiSp(A,n)  
\end{eqnarray*}
 is a weak equivalence.
\end{proposition}

\begin{proposition}\label{level0prop}
The canonical map
 \begin{eqnarray*}
  \cD\Comp(A,n)^{\nondeg}/F^3 &\to& \BiSp(A,n)/F^3    
\end{eqnarray*}
 is a weak equivalence.
\end{proposition}

\begin{corollary}\label{compatcor2}
The canonical maps
\begin{eqnarray*}
  \cD\Comp(A,n)^{\nondeg} &\to& \BiSp(A,n)  \\   
    \cD\Comp(A,n)^{\nondeg} &\to&  \cD\cP(A,n)^{\nondeg}           
\end{eqnarray*}
 are weak equivalences.
\end{corollary}

\begin{remark}
 Note that when $n=0$ and $A$ is concentrated in degree $0$, this recovers the results of \cite[Appendix A]{vdBerghDoublePoisson}, which associate double Poisson structures to bisymplectic structures.
\end{remark}

\subsection{Shifted double Poisson structures on derived NC prestacks}\label{ArtinPoisssn}

The functoriality of shifted double Poisson structures is much more limited than that of shifted pre-bisymplectic structures. In order to extend the definitions to derived Artin NC prestacks (and more generally to  homogeneous derived NC prestacks with cotangent complexes), we have to define structures on the stacky DGAAs introduced in \S \ref{spcotsn}.

\subsubsection{Shifted double Poisson structures on stacky DGAAs} \label{stackyPoisssn}

For the purposes of this section, we now fix a stacky DGAA $A \in DG^+dg\Alg(R)$ (see \S \ref{spcotsn}) which is  cofibrant in the model structure of \cite[Lemma \ref{NCstacks-bidgaamodel}]{NCstacks}, though this can be relaxed  along similar lines to  Remark \ref{weakerassumptionrmk}. 

All the definitions and properties of \S \ref{affpoisssn} extend naturally to stacky DGAAs.  In particular:

\begin{definition}\label{bipoldef}
Define the complex of $n$-shifted non-commutative multiderivations (or polyvectors) on $A$ by
\[
F^i \widehat{\Pol}^{nc}(A,n):= \prod_{p \ge i} \hatHHom_{(A^{e})^{\ten p}}(((\Omega^1_{A/R})_{[-n-1]})^{\ten p},[A^{\ten (p+1)}]), 
\]
for $i \ge 0$,
where we follow \cite[Proposition 4.2.1]{vdBerghDoublePoisson} in writing $[A^{\ten (p+1)}]$ for the $A^{\ten p}$-bimodule given by $A^{\ten (p+1)}$ with the $i$th copy of $A$ acting on the $i$th copy of $A$ on the right and on the $(i+1)$th copy of $A$ on the left.
We then write  $\widehat{\Pol}^{nc}(A,n):= F^0\widehat{\Pol}^{nc}(A,n)$. 

This has an associative multiplication given by the same formulae as in Definition \ref{poldef}.
\end{definition}

\begin{definition}\label{bipolcycdef}
We  define the filtered cochain complex of $n$-shifted cyclic multiderivations  on $A$ by
\[
 F^i\widehat{\Pol}^{nc}_{\cyc}(A,n):= \prod_{p \ge i} (\hatHHom_{(A^{e})^{\ten p}}(((\Omega_A^1)_{[-n-1]})^{\ten p}, \{A^{\ten p}\})^{C_p},
\]
for $i \ge 1$, where (following the proof of \cite[Proposition 4.1.2]{vdBerghDoublePoisson}) $\{A^{\ten p}\}:= [A^{\ten p}]\ten_{A^{e}}A$ is $A^{\ten p}$ given the $A^{\ten p}$-bimodule structure for which the $i$th copy of $A$ acts on the right on the $i$th copy of $A$, and acts on the left on the $(i+1 \mod p)$th copy of $A$.
We then set $F^0\widehat{\Pol}^{nc}_{\cyc}(A,n):= A/[A,A] \oplus F^1\widehat{\Pol}^{nc}_{\cyc}(A,n)$.
\end{definition}

We then define the space $\cD\cP(A,n)$ of $n$-shifted double Poisson structures and its tangent space $T\cD\cP(A,n)$ by the formulae of Definitions \ref{poissdef}, \ref{Tpoissdef}. As in Definition \ref{sigmadef}, there is a  canonical tangent vector $(\id,\sigma) \co \cD\cP(A,n) \to  T\cD\cP(A,n)$. 

Unwinding the definitions, observe that a double Poisson structure on a stacky DGAA $A \in DG^+dg_+\Alg(R)$ gives rise to a sequence $\pi_2, \pi_3, \ldots$ of $k$-brackets on the product total complex $\hatTot A =\Tot^{\Pi}A$, satisfying a form of higher double Jacobi identity. However, since the $k$-brackets lie in the spaces defined using $\hatHHom$, they come with boundedness restrictions from the original cochain direction: if we filter $\hatTot A$ by setting  $\Fil^p \hatTot A:= \Tot^{\Pi} A^{\ge p}$, then each component $\pi_k$  must be $\Fil$-bounded in the sense that for some integer $r$,  each $\Fil^p$ is mapped to $\Fil^{p+r}$.

The  internal $\cHom$ construction of Definition \ref{hatHomdef} interacts as expected with tensors, so we may adapt Definition \ref{mudef} to see that  an element
\[
 \pi \in F^2\widehat{\Pol}^{nc}_{\cyc}(A,n)^{n+2}/F^r   
\]
defines a contraction morphism
\[
\mu(-, \pi) \co \Tot^{\Pi} \Omega_A^p \to F^p\widehat{\Pol}^{nc}(A,n)^{[p]}/F^{p(r-1)}
\]
of bigraded vector spaces, noting that $\Tot^{\Pi} \Omega_A^p = \hatTot\Omega_A^p $.

\begin{definition}\label{binondegdef}
If 
$\Tot^{\Pi} (\Omega^1_{A/R}\ten_{A^{e}}(A^0)^{e}) $ is perfect as an $(A^0)^{e}$-module, and
there exists $N$ for which the chain complexes $(\Omega^1_{A/R}\ten_{A^{e}}(A^0)^{e})^i $ are acyclic for all $i >N$,
we say that an  $n$-shifted double Poisson structure $\pi = \sum_{i \ge 2}\pi_i $ is non-degenerate if contraction with  $\pi_{2}$ 
induces a quasi-isomorphism 
\[
\pi_{2}^{\flat}\co \Tot^{\Pi} (\Omega^1_{A/R}\ten_{A^{e}}(A^0)^{e})_{[-n]} \to \hatHHom_{A^{e}}(\Omega^1_A, (A^0)^{e})[-n].
\]
\end{definition}

\begin{example}\label{2PBGex1}
 As in Example \ref{BGex2}, given 
 a finite flat  non-unital associative $R$-algebra $\g$, we have a cochain algebra  $O(B\g)$ given by the bar construction of $\g$. In that example, we described the $2$-shifted pre-bisymplectic structures, and we now investigate the $2$-shifted double Poisson structures. 
 
 By considering degrees of generators, the only possible non-zero term of an element $\pi \in \cD\cP(O(B\g),2)$ is $\pi_2$, determined by its restriction to a map $\g^*\ten \g^* \to R= O(B\g)^0$; the same considerations show that the higher simplices are degenerate, so the space $ \cD\cP(O(B\g),2)$ is discrete. Cyclic invariance then implies that $\pi_2 \in \Symm^2\g$, with compatibility between $\pd$ and $\pi$ reducing to the condition that
 $[v, \pi_2]=0 \in \g^{\ten 2}$ for all $v \in \g$, where
\[
 [v, a\ten b]:= (v\cdot a)\ten b - a\ten (b\cdot v) +(v\cdot b)\ten a - b\ten(a\cdot v),
\]
for the natural left and right actions of $\g$ on $\g$ given  by multiplication. The double Poisson structure $\pi=\pi_2$ is then non-degenerate if and only if it induces an isomorphism $\g^* \to \g$.

This gives a non-commutative analogue of \cite[Example \ref{poisson-2PBG}]{poisson}, where $\cP(O(B\g)^{\comm},2) \simeq (\Symm^2\g)^{\g}$, part of a general phenomenon that shifted double Poisson structures induce shifted Poisson structures on their commutative restrictions.
\end{example}

The proof of \cite[Lemma \ref{poisson-binondeglemma}]{poisson} now adapts to give:

\begin{lemma}\label{binondeglemma}
 If $\pi_2 \in  \cP(A,n)/F^3$ is non-degenerate, then the maps
\[
 \mu(-, \pi_2) \co \Tot^{\Pi}\Omega_{A,\cyc}^p \to  \gr_F^p\widehat{\Pol}^{nc}(A,n)^{[p]}
\]
are all  quasi-isomorphisms.
\end{lemma}

\subsubsection{Compatible pairs}\label{Artincompat}

The proof of Lemma \ref{mulemma} adapts to give  maps
\[
 (\pr_2 + \mu\eps) \co  \PreBiSp(A,n)/F^{p} \by \cD\cP(A,n)/F^{p} \to T\cD\cP(A,n)/F^{p}
\]
 over $\cD\cP(A,n)/F^{p}$ for all $p$, compatible with each other.

We may now define the space $\cD\Comp(A, n)$ of compatible pairs as in Definition \ref{compdef}, with the results of \S \ref{equivaffinesn} all adapting to show that the maps     $\BiSp(A,n) \la  \cD\Comp(A,n)^{\nondeg} \to  \cD\cP(A,n)^{\nondeg}$ are weak equivalences. The only modification is in the definitions of the obstruction spaces, which feature $\hatHHom$ instead of $\HHom$ and $\Tot^{\Pi}\Omega^p_{\cyc}$ instead of $\Omega^p_{\cyc}$.

\subsubsection{\'Etale functoriality}\label{Artindiagramsn}

We now adapt the constructions of \cite[\S\S  \ref{poisson-DMdiagramsn}, \ref{poisson-Artindiagramsn}]{poisson} to our context.

Given a small category $I$, an $I$-diagram $A$ of stacky DGAAs over $R$, and   right $A$-modules $M,N$ in $I$-diagrams of chain cochain complexes, we can define the  cochain complex $\hatHHom_A(M,N)$ to be the equaliser of the obvious diagram
\[
\prod_{i\in I} \hatHHom_{A(i)}(M(i),N(i)) \implies \prod_{f\co i \to j \text{ in } I}   \hatHHom_{A(i)}(M(i),f_*N(j)).
\]
All the constructions of \S \ref{stackyPoisssn} then adapt immediately; in particular, we can define 
\begin{align*}
 F^i\widehat{\Pol}^{nc}(A,n)&:= \prod_{p \ge i}\hatHHom_{(A^{e})^{\ten p}}(((\Omega^1_{A/R})_{[-n-1]})^{\ten p},[A^{\ten (p+1)}]), \\
 F^i\widehat{\Pol}^{nc}_{\cyc}(A,n)&:= \prod_{p \ge i} \hatHHom_{(A^{e})^{\ten p}}(((\Omega_A^1)_{[-n-1]})^{\ten p}, \{A^{\ten p}\})^{C_p},
\end{align*}
leading to a space $\cD\cP(A,n)$ of shifted double  Poisson structures on the diagram $A$. For most diagrams $A$, this definition is too crude to have the correct homological properties, so we now restrict attention to  categories  of  the form $[m]= (0 \to 1 \to \ldots \to m)$.

\begin{lemma}\label{calcTOmegalemma}
 If $A$ is an $[m]$-diagram in   $DG^+dg_+\Alg(R)$  which is cofibrant and  fibrant for the injective model structure (i.e. each $A(i)$ is cofibrant in the model structure of \cite[Lemma \ref{NCstacks-bidgaamodel}]{NCstacks} and the maps $A(i) \to A(i+1)$ are surjective), 
 then 
$ \hatHHom_{(A^{e})^{\ten p}}(((\Omega^1_{A/R})_{[-n-1]})^{\ten p},[A^{\ten (p+1)}])$ (resp. $\hatHHom_{(A^{e})^{\ten p}}(((\Omega_A^1)_{[-n-1]})^{\ten p}, \{A^{\ten p}\})$)
 is a model for the derived $\Hom$-complex $ \oR\hatHHom_{(A^{e})^{\ten p}}(((\Omega^1_{A/R})_{[-n-1]})^{\ten p},[A^{\ten (p+1)}])$ (resp. $\oR\hatHHom_{(A^{e})^{\ten p}}(((\Omega_A^1)_{[-n-1]})^{\ten p}, \{A^{\ten p}\})$.
 \end{lemma}
\begin{proof}
We will prove the first statement; for the second, replace $[A^{\ten (p+1)}]$ with $\{A^{\ten p}\}$ throughout.
Because $A(i)$ is cofibrant, the double complex $C(i):= (\Omega^1_{A(i)/R})_{[-n-1]})^{\ten p}$ is cofibrant as an $(A(i)^{\ten p})^{e}$-module, so   the complex   $ \oR\hatHHom_{(A^{e})^{\ten p}}(((\Omega^1_{A/R})_{[-n-1]})^{\ten p},[A^{\ten (p+1)}])$
 is the homotopy limit of the diagram
\[
\xymatrix@R=0ex{ \hatHHom_{(A(0)^{\ten p})^{e}}(C(0),[A(0)^{\ten (p+1)}])\ar[r] &  \hatHHom_{(A(0)^{\ten p})^{e}}(C(0),[A(1)^{\ten (p+1)}])\\
\ar[ur] \hatHHom_{(A(1)^{\ten p})^{e}}(C(1),[A(1)^{\ten (p+1)}]) \ar[r] & \ldots.\phantom{  \la \HHom_{A(m)}(C(m),A(m))}
}
\]

Since the maps $A(i) \to A(i+1)$ are surjective, the maps 
\[
\hatHHom_{(A(i)^{\ten p})^{e}}(C(i),[A(i)^{\ten (p+1)}])\to  \hatHHom_{(A(i)^{\ten p})^{e}}(C(i),[A(i+1)^{\ten (p+1)}])
\]
are also, and thus the homotopy limit is calculated by the ordinary limit, which is precisely $\hatHHom_{(A^{\ten p})^{e}}(C,[A^{\ten (p+1)}])$.
\end{proof}

\begin{definition}
 Write  $DG^+dg_+\Alg(R)_{c, \onto}\subset DG^^+dg_+\Alg(R) $ for the subcategory with all cofibrant stacky  $R$-DGAAs as objects, and only surjective morphisms.
\end{definition}

We already have functors $\PreBiSp(-,n)$ and $\BiSp(-,n)$ from  $DG^+dg_+\Alg(R)$ to $s\Set$, the category of  simplicial sets, mapping quasi-isomorphisms in $DG^+dg_+\Alg(R)_{c}$ to weak equivalences. Double Poisson structures are only functorial with respect to homotopy  \'etale morphisms, in an $\infty$-functorial sense which we now make precise. 

Observe that, writing $F$ for any  of the constructions  $\cD\cP(-,n)$, $\cD\Comp(-,n)$, $\PreBiSp(-,n)$,  and the associated filtered and graded functors, applied to $[m]$-diagrams in $DG^+dg_+\Alg(R)_{c, \onto}$,  we have:

\begin{properties}\phantomsection\label{Fproperties}

 \begin{enumerate}
 \item the natural maps from $F(A(0)\to \ldots \to A(m))$ to 
\[
 F(A(0)\to A(1))\by_{F(A(1))}^h F(A(1)\to A(2))\by^h_{F(A(2))}\ldots \by_{F(A(m-1))}^hF(A(m-1) \to A(m))
\]
 are weak equivalences;

\item if the $[1]$-diagram $A \to B$ is a quasi-isomorphism, then the natural maps from $F(A \to B)$ to $F(A)$ and to $F(B)$ are weak equivalences. 

\item if the $[1]$-diagram $A \to B$ is homotopy \'etale, then the natural map from $F(A \to B)$ to $F(A)$ is a  weak equivalence. 
\item if the $[1]$-diagram $A \to B$ is homotopy \'etale and admits an $A$-bilinear retraction, then the natural map from $F(A \to B)$ to $F(B)$ is also a  weak equivalence.

\end{enumerate}
These properties follow from Lemmas \ref{calcTOmegalemma} and \ref{calcTOmegaetlemma}, together with  the obstruction calculus of \S \ref{towersn} extended to diagrams. Note that the final property has no commutative analogue.
\end{properties}

Property \ref{Fproperties}.(1) ensures that the simplicial classes $\coprod_{ A \in B_m DG^+dg_+\Alg(R)_{c, \onto}} F(A)$ fit together to give a complete Segal space $\int F$ over the nerve $BDG^+dg_+\Alg(R)_{c, \onto} $. Taking complete Segal spaces \cite[\S 6]{rezk}  as our preferred model of $\infty$-categories:

\begin{definition}\label{LintFdef}
Define $\oL DG^+dg_+\Alg(R)_{c, \onto}$, $\oL DG^+dg_+\Alg(R)$, $\oL\int F$, and $\oL s\Set$   to be the $\infty$-categories obtained by localising the respective 
categories at quasi-isomorphisms or weak equivalences.
\end{definition}

Property \ref{Fproperties}.(2)  ensures that the homotopy fibre of $\oL\int F \to  \oL DG^+dg_+\Alg(R)_{c, \onto}$ over $A$ is still just the simplicial set $F(A)$, regarded as an $\infty$-groupoid.

Since the surjections in $DG^+dg_+\Alg(R)$ are the fibrations, the inclusion   $\oL DG^+dg_+\Alg(R)_{c, \onto} \to \oL DG^+dg_+\Alg(R)$ is a weak equivalence, and we may regard $\oL\int F $ as an $\infty$-category over $\oL dgCAlg(R) $.

Furthermore,  Property \ref{Fproperties}.(3) implies that the $\infty$-functor $\oL\int F \to \oL DG^+dg_+\Alg(R) $ is a co-Cartesian fibration when we restrict to the $2$-sub-$\infty$-category $\oL DG^+dg_+\Alg(R)^{\et} $ of homotopy formally \'etale morphisms, giving:

\begin{definition}\label{inftyFdef}
When $F$ is any of   the constructions above, we define 
\[
 \oR F \co \oL DG^+dg_+\Alg(R)^{\et} \to  \oL s\Set
\]
to be the $\infty$-functor whose Grothendieck construction is $\oL\int F  $.
\end{definition}

In particular, the observations above ensure that 
\[
 (\oR F)(A) \simeq F(A)
\]
for all cofibrant stacky DGAAs $A$ over $R$.

An immediate consequence of the generalisation in \S \ref{Artincompat}  of Corollary \ref{compatcor2} to stacky DGAAs  is that the canonical maps
\begin{eqnarray*}
  \oR\cD\Comp(-,n)^{\nondeg} &\to& \oR\BiSp(-,n)  \\   
    \oR\cD\Comp(-,n)^{\nondeg} &\to&  \oR\cD\cP(-,n)^{\nondeg}           
\end{eqnarray*}
 are natural weak equivalences of $\infty$-functors on $\oL DG^+dg_+\Alg(R)^{\et}$. 
In other words, on the $\infty$-category of homotopy \'etale morphisms of stacky DGAAs, we have a functorial equivalence between shifted bi-symplectic structures and shifted double Poisson structures.

\subsubsection{Structures on derived NC prestacks}\label{derNCprestacksn}

We are now in a position to adapt Definition \ref{PreBiSphgsdef} to double Poisson structures:
\begin{definition}\label{cPhgsdef}
Given an derived NC prestack   $F \in DG^+\Aff^{nc}(R)^{\wedge}$ which is homogeneous and has bounded below cotangent complexes $L_F(B,x)$ (in the sense of \cite[Definition \ref{NCstacks-Fcotdef}]{NCstacks}) at all points $x \in F(B)$, we define 
 the space $\cD\cP(F,n)$  of $n$-shifted double Poisson structures on $F$ by
\begin{align*}
\cD\cP(F,n):= \map_{\oL (dg^+DG^+\Aff(R)^{\et})^{\wedge}}((D_*F)_{\rig}, \oR\cD\cP(-,n))
 \end{align*} 
 for the functor $(D_*F)_{\rig}$ of rigid points from Definition \ref{Frigdefhere},
 and similarly for the space  $\cD\cP(F,n)^{\nondeg}$  of non-degenerate $n$-shifted double Poisson structures on $F$, noting that non-degeneracy is preserved by homotopy \'etale morphisms. We define the space $\cD\Comp(F,n)$ of non-degenerate compatible pairs similarly.
\end{definition}
 
In other words, an $n$-shifted double Poisson structure on $F$ consists of compatible   $n$-shifted double Poisson structures on $A$ for all rigid points $x \in D_*F(A)$ and all stacky DGAAs $A$. 

Reasoning as in Corollary \ref{integratecorBiSp}, but with the $\Hom$ form of Lemma \ref{calcTOmegaetlemma}, we  have:
 \begin{corollary}\label{integratecorDP}
  For a sqc derived NC Artin prestack $F$ coming from a simplicial diagram $X_{\bt}$ of derived NC affines, the natural maps $\cD\cP(F,n) \to \cD\cP(D^*O(X),n)$ and $\cD\cP(F,n) \to \cD\cP(D^*O(X),n)$ are weak equivalences.
  \end{corollary}

In other words, every shifted double Poisson structure on $D^*O(X)$ (an NC analogue of a derived Lie algebroid) integrates uniquely to one on $F$ (an algebraic NC analogue of a derived  Lie groupoid), in stark contrast with the commutative setting.

\begin{example}\label{2PBGex2b}
For a finite flat associative $R$-algebra $\g$, we can consider the NC affine group $G:= \Pi_{\g/R}\bG_m= \bG_m(A\ten_R-)$  and its nerve   $BG=\Pi_{\g/R}B\bG_m$, for the Weil restriction functor $\Pi_{\g/R}$ of \cite[\S \ref{NCstacks-weilrestrn}]{NCstacks}; when $\g$ is the matrix ring $\Mat_n$, note that $\Pi_{\Mat_n}\bG_m=\GL_n$.  

Corollary \ref{integratecorDP} then gives $\cD\cP(BG,n) \simeq \cD\cP(B\g,n)$, which is roughly the space of double $P_{n+1}$-algebra structures on the DGAA $O(B\g)$ from Example \ref{BGex2}. When $R=\H_0R$ and $n=2$, this is just the discrete space of  Example \ref{2PBGex1}.
\end{example}

An immediate consequence of \S\S \ref{Artincompat}--\ref{Artindiagramsn} is now: 
  \begin{theorem}\label{Artinthm}
For any derived NC prestack $F$ over $R$ which is homogeneous with perfect cotangent complexes (in the sense of \cite[Definition \ref{NCstacks-Fcotdef}]{NCstacks}) at all points, there are natural weak equivalences
\[
 \BiSp(F,n) \la \cD\Comp(F,n)^{\nondeg}\to \cD\cP(F,n)^{\nondeg}
\]
between the spaces of $n$-shifted bi-symplectic structures and of non-degenerate $n$-shifted double Poisson structures on $F$.
\end{theorem}

Combining Theorem \ref{Artinthm} with Theorems \ref{Perfthm2} and \ref{Morthm} gives the following immediate corollaries:

\begin{corollary}\label{Perfpoissoncor}
Take  a locally proper dg category  $\cA$ over $R$  which is right $d$-Calabi--Yau in the sense that it is  equipped with a cyclic cohomology class $\alpha \in \HC_R^{-d}(\cA)$ whose  image $\bar{\alpha}$ under the natural map $\HC_R^{-d}(\cA) \to \HH_R^{-d}(\cA,\cA^*)=\Ext^{-d}_{\cA^{\oL,e}}(\cA(-,-),\cA^*(-,-))$ is non-degenerate. Then there is an associated class in $\pi_0\cD\cP(\Perf_{\cA},2-d)^{\nondeg}$ of  non-degenerate $(2-d)$-shifted double Poisson structures on the derived NC prestack of perfect $\cA$-modules.
\end{corollary}

 \cite[Example \ref{NCstacks-DstarPerf}]{NCstacks} to interpret $D_*\Perf$, this says that given a right $d$-CY locally proper dg category $\cA$ as above, there is a functorial $(2-d)$-shifted double Poisson structure  (as in \S \ref{stackyPoisssn}, Definition  \ref{poissdef}) on every  stacky DGAA $B$  equipped with a  homotopy-Cartesian right $\cA\ten^{\oL}_RB$-module $E^{\bt}_{\bt}$ in double complexes  for which $E^0_{\bt}$ is perfect over $\cA\ten^{\oL}_RB^0$, and which is rigid in the sense that the non-commutative Atiyah class
\[
\At_E \co  \oR\hatHHom_{B^{\oL,e}}(\Omega^1_B,M)  \to \oR\HHom_{ \cA\ten^{\oL}_RB^0}(E^0_{\bt}, E^0_{\bt}\ten^{\oL}_{B^0}M )
\]
is a quasi-isomorphism for all $B^0$-bimodules $M$.

For a precursor of this result, see \cite[Theorem 15]{BerestChenEshmatovRamadoss}, which puts a double $(n + 2)$-Poisson structure on the Koszul dual of an  $n$-cyclic coassociative DG coalgebra, where the latter is required to satisfy a strict non-degeneracy condition which is not quasi-isomorphism invariant.

\begin{corollary}\label{Morpoissoncor}
Take a  smooth dg category $\cA$ over $R$ which is left $d$-Calabi--Yau in the sense that it is  equipped with a negative cyclic homology class $\alpha \in \HN_{R,d}(\cA)$ whose image $\bar{\alpha}$ under the natural map $\HN_{R,d}(\cA) \to \HH_{R,d}(\cA)= \Tor_{d}^{\cA^{\oL,e}}(\cA(-,-),\cA(-,-))$ is non-degenerate. Then there is an associated class in $\pi_0\cD\cP(\Mor_{\cA},2-d)^{\nondeg}$ of  non-degenerate $(2-d)$-shifted double Poisson structures on the derived NC prestack of derived Morita morphisms from $\cA$.
\end{corollary}

Adapting \cite[Example \ref{NCstacks-DstarPerf}]{NCstacks} to interpret $D_*\Mor$,  this says that given a left $d$-CY smooth dg category $\cA$ as above, there is a functorial $(2-d)$-shifted double Poisson structure (as in \S \ref{stackyPoisssn}, Definition  \ref{poissdef}) on every  stacky DGAA $B$  equipped with a  homotopy-Cartesian right $\cA^{\op}\ten^{\oL}_RB$-module $E^{\bt}_{\bt}$ in double complexes  for which $E^0_{\bt}(X)$ is perfect over $B^0$ for all $X \in \cA$, and which is rigid in the sense that the non-commutative Atiyah class 
\[
\At_E \co  \oR\hatHHom_{B^{\oL,e}}(\Omega^1_B,M)  \to \oR\HHom_{ \cA^{\op}\ten^{\oL}_RB^0}(E^0_{\bt}, E^0_{\bt}\ten^{\oL}_{B^0}M )
\]
is a quasi-isomorphism for all $B^0$-bimodules $M$.

\subsubsection{Pre-Calabi--Yau structures via \tps{$\bG_m$}{Gm}-actions}\label{preCYsn}


Following \cite[\S 2.5]{yeungPreCYModReps}, write $F^p\sP^{\tot}(A,n)$ for the analogue of $F^p\widehat{\Pol}^{nc}_{\cyc}(A,n)$ given by replacing $\Omega^1_A$ with the complex $\cS'(A)$ of \S \ref{CYsn}. For $p=2$, this carries a shifted Lie bracket defined similarly to that of Proposition \ref{bracketprop}.

Resuming the notation of \S \ref{CYsn}, we have:
 \begin{lemma}\label{CYlemmaPol}
  For $p>0$, the DGLA $F^2\Pol^{nc}_{\cyc}(A\<s_{\ad}\>,n)^{[n+1]}$ admits a canonical filtered quasi-isomorphism from  $F^2\sP(A,n)^{[n+1]}$. 
 \end{lemma}
\begin{proof}
The isomorphisms  $\Omega^1_{A\<s_{\ad}\>} \cong A\<s_l\>\ten_A\cS'(A)\ten_AA\<s_r\>$, $A\<s_{\ad}\>\cong A\<s_l\>\ten_{A\<s_0\>}A\<s_r\>$ and $A\<s_0\>\cong A\<s_r\>\ten_{A\<s_{\ad}>}A\<s_l\>$ from \S \ref{CYsn} give
\begin{align*}
&\hatHHom_{(A\<s_{\ad}\>^{e})^{\ten p}}(((\Omega^1_{A\<s_{\ad}\>})_{[-n-1]})^{\ten p},\{A\<s_{\ad}\>^{\ten p}\})\\
&\cong  
\hatHHom_{(A\<s_0\>^{e})^{\ten p}}(((A\<s_0\>\ten_A\cS'(A)\ten_AA\<s_0\>)_{[-n-1]})^{\ten p}, \{A\<s_0\>^{\ten p}\}) \\
&\cong
\hatHHom_{(A^{e})^{\ten p}}((\cS'(A)_{[-n-1]})^{\ten p},\{A\<s_0\>^{\ten p}\}).
\end{align*}
Arguing as in Lemma \ref{CYlemmaDR}, this admits a contracting  homotopy to 
\[
 \hatHHom_{(A^{e})^{\ten p}}((\cS'(A)_{[-n-1]})^{\ten p}, \{A^{\ten p}\}),
\]
so these combine to give  a quasi-isomorphic subcomplex  of $F^j\Pol^{nc}_{\cyc}(A\<s_{\ad}\>,n)$ isomorphic to $F^j\sP^{\tot}(A,n)$. A routine check shows that for $j=2$ this subcomplex is closed under the Lie bracket, which agrees with that on  $F^2\sP^{\tot}(A,n)$.
 \end{proof}
Similarly to the situation with Lemma \ref{CYlemmaDR}, the proof does not give a subcomplex of $F^2\Pol^{nc}(A\<s_{\ad}\>,n)^{[n+1]}$, only of its cyclic quotient.

\medskip
Following \cite[Definition 2.17 and Theorem 2.46]{yeungPreCYModReps}, 
We can define a $d$-pre-Calabi--Yau (or $d$-pre-CY) structure on stacky DGAA $A$ to be an element of $\mc(F^2\sP^{\tot}(A,2-d)$ whose leading term induces a quasi-isomorphism
 \[
  \Tot^{\Pi} (A\ten_{A^{\oL,e}}(A^0)^{\oL,e})_{[d]} \hatHHom_{A^{\oL,e}}(A, (A^0)^{\oL,e})
 \]
--- this is consistent with \cite[Definition 2.17]{yeungPreCYModReps} by \cite[Theorem 2.46]{yeungPreCYModReps}. As in earlier versions of \cite{yeungPreCYModReps}, we can form a space of such structures as  $\Lim_p \mmc(F^2\sP^{\tot}(A,2-d)/F^p)$.

\begin{proposition}\label{preCYprop}
 For $A \in DG^+dg\Alg(R)$ cofibrant, the following are equivalent:
 \begin{enumerate}
  \item The space of smooth $d$-pre-Calabi--Yau structures on $A$.
  \item The space $\cD\cP([\Spec^{nc} A/\g_m],2-d)$ of $(2-d)$-shifted double Poisson structures on the formal quotient of $\Spec^{nc} A$ by $\g_m$.
  \item The space $\cD\cP(\oR[\Spec^{nc} A/\bG_m],2-d)$ of $(2-d)$-shifted double Poisson structures on the derived  quotient of $\Spec^{nc} A$ by $\bG_m$.
  \end{enumerate}
\end{proposition}
 \begin{proof}
The equivalence of the first two follows immediately from the filtered quasi-isomorphism of Lemma \ref{CYlemmaPol}. The second equivalence then follows from Proposition \ref{integratecorDP}, reasoning as in Proposition \ref{CYprop}.
 \end{proof}

 \begin{remarks}
  Although the characterisation of pre-CY structures as shifted double Poisson structures on $[\Spec^{nc} A/\g_m]$ gives a fairly simple description, the characterisation as structures on  $\oR[\Spec^{nc} A/\bG_m]$ is much more powerful. In particular, combined with Remarks \ref{CostelloRozenblyumRmk} below, it  induces shifted Poisson structures on (commutative) derived moduli stacks of representations, recovering  \cite[Theorem 4.39]{yeungPreCYModReps} for $A$.
  
The methods of Remark \ref{dgcatrmk} also provide an extension of  Proposition \ref{preCYprop} to dg categories. In particular, for a dg category with finite object set $X$, the space of  pre-$d$-Calabi--Yau structures on $\cA$ is equivalent to the spaces $\cD\cP(A\<s_{\ad}\>/([s,1_x])_{x \in X} ,2-d)$ and $\cD\cP(\oR[\Spec^{nc} A/\mathrm{Stab}_{\bG_m}X],2-d)$.

Applied to these prestacks, Theorem \ref{Artinthm} combines with Propositions \ref{CYprop} and \ref{preCYprop} to recover the equivalence between CY and non-degenerate pre-CY structures from earlier versions of \cite{yeungPreCYModReps} and from  \cite{KontsevichTakedaVlassopoulosLegendre} (which, though later than the main body of this article, precedes the comparisons of \S\S \ref{CYsn}, \ref{preCYsn}). All three approaches use a Legendre transformation, with apparent simplifications in our setting because the  Lie bracket on cyclic NC polyvectors lifts to a double Poisson bracket.
 \end{remarks}

\subsection{Shifted double co-isotropic structures}\label{coisosn}

We will now see how the comparisons we have seen so far extend to shifted bi-Lagrangians.

In the commutative setting, as in \cite{MelaniSafronovI}, 
an $n$-shifted co-isotropic structure on a morphism $f \co A \to B$ consists of an $n$-shifted Poisson structure on $A$, an $(n-1)$-shifted Poisson  structure $\pi$ on $B$ and a lift of $f$ to a strong homotopy $P_{n+1}$-algebra morphism $A \to (\widehat{\Pol}(B,n-1),\delta +[\pi,-])$ to the complex of twisted polyvectors. There is an interpretation in terms of additivity showing that the latter data are equivalent to a strongly homotopy associative action of $A$ on the $P_n$-algebra $B$.
 
In our non-commutative setting, we would like to adapt this definition to say that an $n$-shifted double  co-isotropic structure on a morphism $f \co A \to B$ consists of an $n$-shifted double Poisson structure on $A$, an $(n-1)$-shifted double Poisson  structure $\pi$ on $B$ and a lift of $f$ to a strong homotopy double $P_{n+1}$-algebra morphism $A \to \widehat{\Pol}_{\pi}^{nc}(B,n-1)$ to the complex of twisted polyvectors. The required notion of strong homotopy morphisms is provided by Koszul duality for protoperads, 
 via \cite[Theorems 2.24 and 5.11]{lerayProtoperads2Koszul} and \cite{HoffbeckLerayVallette}. Beware that (similarly to  the commutative case) strong homotopy double Poisson algebras form a larger category than our formulation of shifted double Poisson structures as DGAAs equipped with $k$-brackets satisfying a higher double Jacobi identity as described in Remark \ref{DPinterpretrmk}. However, the description of \cite[Proposition 5.7 and Corollary 5.10]{lerayProtoperads2Koszul} ensures that  every  strong homotopy double $P_{n+1}$-algebra is quasi-isomorphic to a cofibrant DGAA $A$ equipped with an $n$-shifted double Poisson structure in our sense. 

 
We will not attempt to establish non-commutative analogues of  all the results of \cite{MelaniSafronovI,MelaniSafronovII} here, but instead focus on non-degenerate double co-isotropic structures, i.e. those for which the double Poisson structure on $A$ is non-degenerate and the morphism $A \to \widehat{\Pol}^{nc}_{\pi}(B,n-1) $ is homotopy formally \'etale. In these cases, we can bypass Koszul duality altogether, since double Poisson structures are functorial with respect to homotopy formally \'etale morphisms. 
This allows us to characterise a  non-degenerate $n$-shifted double co-isotropic structure on $f \co A \to B$  as a  non-degenerate $n$-shifted double  Poisson structure $\varpi$ on $A$, an $(n-1)$-shifted double  Poisson structure $\pi$ on $B$, a homotopy formally \'etale morphism $\tilde{f} \co  A \to \widehat{\Pol}_{\pi}^{nc}(B,n-1)$ of associative algebras (regarding the target as a pro-object) lifting $f$, and a homotopy $k \co \tilde{f}^*\varpi \sim \phi$ between the pullback of $\varpi$ and the canonical completed $n$-shifted double  Poisson structure $\phi$  on $\widehat{\Pol}_{\pi}^{nc}(B,n-1)$ (see the proof of Proposition \ref{bracketprop}).

\begin{proposition}\label{coisoprop}
 On a morphism $f \co A \to B$ of cofibrant $R$-DGAAs in $dg_+\Alg(R)$, there is a natural equivalence between the space $\BiLag(A,B;n)$ of $n$-shifted  bi-Lagrangian structures  and the space $\cD\Coiso(A,B;n)$ of non-degenerate $n$-shifted double co-isotropic structures.
\end{proposition}
\begin{proof}
 Given a non-degenerate $n$-shifted double co-isotropic structure $(\varpi, \pi, \tilde{f},k)$ as above, we may use a completed version of the equivalence of Corollary \ref{compatcor2} between bisymplectic and non-degenerate double Poisson structures to replace $(\varpi,k)$ with an $n$-shifted bisymplectic structure $\omega$ on $A$ and a homotopy $l \co  \tilde{f}^*\omega \sim \psi$ between the pullback of $\omega$ and the canonical $n$-shifted bisymplectic structure $\psi$  on $\widehat{\Pol}_{\pi}^{nc}(B,n-1)$ corresponding to the $n$-shifted double  Poisson structure $\phi$. 
 
For the natural map $r \co \widehat{\Pol}_{\pi}^{nc}(B,n-1) \to B$, we have $r^*\psi= 0$, so $\lambda:= r^*l$ is a homotopy from $r^*\tilde{f}^*\omega=f^*\omega$ to $0$, and hence is a shifted bi-isotropic structure on $B$ over $A$. Since  $\tilde{f}$ is homotopy formally \'etale,  the shifted bi-isotropic structure $(\omega, \lambda)$ induces a the required quasi-isomorphism between tangent and relative cotangent spaces  to make it  bi-Lagrangian, thus giving us our desired map from non-degenerate shifted double co-isotropic structures to shifted bi-Lagrangians.

In order to show that this map $\cD\Coiso(A,B;n)^{\nondeg} \to \BiLag(A,B;n) $  is an equivalence, we cofilter the respective spaces, and work up a tower of obstructions, similarly to our comparison of double Poisson and bi-symplectic structures in \S \ref{towersn}. For bi-Lagrangians, we use the cofiltration on bi-isotropic structures from Definition \ref{Isodef}, so 
we have a long exact sequence
\begin{align*}
 \ldots \to &\pi_i(\BiLag(A,B;n)/F^{p+1}) \\
 &\to \pi_i(\BiLag(A,B;n)/F^{p})\to \H_{P+i-n-2}\cone(\Omega^{p+1}_{\cyc}(A/R)  \to \Omega^{p+1}_{\cyc}(B/R))\to \ldots 
\end{align*}
for $p \ge 2$.

The cofiltration on $\cD\Coiso(A,B;n)^{\nondeg}$ is more subtle, involving both the Hodge filtration and the canonical filtration $F$ on $\widehat{\Pol}_{\pi}^{nc}(B,n-1)$. We let $\Fil$ be the smallest multiplicative filtration on 
\[
\DR(\widehat{\Pol}_{\pi}^{nc}(B,n-1)):= \Lim_p  \DR(\widehat{\Pol}_{\pi}^{nc}(B,n-1)/F^p)
\]
for which $\Fil^p$ contains both the Hodge filtration $F^p_{\DR}$ and the subspace $F^p_{\Pol}\widehat{\Pol}_{\pi}^{nc}(B,n-1) \subset \widehat{\Pol}_{\pi}^{nc}(B,n-1)$. For instance, if $B$ is freely generated by elements $x_i$, with corresponding tangent vectors $\xi_i$, then the index of $\Fil^p$ is given by the combined powers of $dx_i, \xi_i,d\xi_i$. As an immediate consequence, note that on the associated graded pieces, the maps $\gr_{\Fil}^p \DR(\widehat{\Pol}_{\pi}^{nc}(B,n-1)) \to \gr_F^p\DR(B)= \Omega^p_B$ are  quasi-isomorphisms. 

Via the equivalence described in the first paragraph, we then let $\cD\Coiso(A,B;n)^{\nondeg}/F^p$ correspond to data  $(\omega,\pi,  \tilde{f},l)$, with  
\[
\omega \in \BiSp(A,n)/F^p, \quad \pi \in \cD\cP(B,n-1)/F^p, \quad \tilde{f} \co A \to \widehat{\Pol}_{\pi}^{nc}(B,n-1)/F^{p-1}
\]
lifting $f$ and inducing a quasi-isomorphism $f^*\Omega^1_A \simeq r^*\Omega^1_{\widehat{\Pol}_{\pi}^{nc}(B,n-1)}$, and $l$ a homotopy  $\tilde{f}^*\omega \sim \psi$ in $F^2_{\DR}\DR(\widehat{\Pol}_{\pi}^{nc}(B,n-1))/(F^2_{\DR} \cap \Fil^p)$. There is then a natural map 
\[
\cD\Coiso(A,B;n)^{\nondeg}/F^p \to \BiLag(A,B;n)/F^{p}
\]
sending $(\omega,\pi,  \tilde{f},l)$ to $(\omega, r^*l)$, and we will show by induction on $p$ that these are equivalences.

In the base case $p=3$,  for this model of $\cD\Coiso(A,B;n)^{\nondeg}/F^3$, an element $(\omega,\pi,  \tilde{f},l) \in \cD\Coiso(A,B;n)^{\nondeg}/F^3$ consists of a non-degenerate element $\omega \in \z_{-n}\Omega^2_{A,\cyc}$, an element 
\[
\pi \in \z_{-n-1}\HHom_{(B^{e})^{\ten 2}}(((\Omega_B^1)_{[-n]})^{\ten 2}, \{B^{\ten 2}\})^{C_2},
\]
a quasi-isomorphism 
\[
d\tilde{f} \co f^*\Omega^1_A \to \cocone(\mu(-,\pi)\co \Omega^1_B \to  \DDer(B,B^{e})[1-n])
\]
lifting  the canonical map $ f^*\Omega^1_A \to \Omega^1_B$ and a homotopy $l$ in $(r^*\Omega^2_{\widehat{\Pol}_{\pi}^{nc}(B,n-1)})_{\cyc}$ between $\phi$ and $\tilde{f}^*\omega$. We can summarise all this data by saying that once we fix $\omega$, we have a homotopy $(\pi, d\tilde{f}, l)$ killing $\phi$  in the complex
\begin{align*}
\HHom_{(B^{e})^{\ten 2}}(((\Omega_B^1)_{[-n]})^{\ten 2}, \{B^{\ten 2}\})^{C_2}
\xra{r^*} &\HHom_{A^{\oL,e}}( \Omega^1_A,  \DDer(B,B^{e}))[-2n]\\ 
&\xra{\omega^{\sharp}}(r^*\Omega^2_{\widehat{\Pol}_{\pi}^{nc}(B,n-1)})_{\cyc},
\end{align*}
satisfying a non-degeneracy condition.
Using the quasi-isomorphism between $f^*\Omega^1_A$ and $\cocone(\pi^{\flat}\co \Omega^1_B \to \DDer(B,B^{\oL,e})[1-n])$ provided by 
non-degeneracy, this complex is just quasi-isomorphic to   $ \gr_{\Fil}^2\DR(\widehat{\Pol}_{\pi,\cyc}^{nc}(B,n-1)$, which as above is quasi-isomorphic to  $\Omega^2_{B,\cyc}$, ensuring that our map $\cD\Coiso(A,B;n)^{\nondeg}/F^3 \to \BiLag(A,B;n)/F^3$ 
 is  
a weak equivalence.

For the higher cases, we observe that the same reasoning adapts to show that the natural map $\Coiso(A,B;n)^{\nondeg}/F^{p+1} \to \Coiso(A,B;n)^{\nondeg}/F^p $ is the homotopy fibre of an obstruction map to a shift of the complex
\begin{align*}
 \gr_F^{p}\widehat{\Pol}^{nc}(B,n-1)
\xra{(r^*,0)} &\HHom_{A^{\oL,e}}( \Omega^1_A,  \gr_F^{p-1}\widehat{\Pol}^{nc}(B,n-1))[-n]  \oplus \Omega^p_{A,\cyc}\\ 
&\xra{(\omega^{\sharp}, \tilde{f}_2^* )} F^2_{\DR}\gr_{\Fil}^p\DR(\widehat{\Pol}_{\pi,\cyc}^{nc}(B,n-1))),
\end{align*}
and that this is quasi-isomorphic to
$\cone ( \Omega^2_{A,\cyc} \to\gr_{\Fil}^p\DR(\widehat{\Pol}_{\pi,\cyc}^{nc}(B,n-1)))$, so the quasi-isomorphism $\gr_{\Fil}^p\DR(\widehat{\Pol}_{\pi}^{nc}(B,n-1)) \simeq \Omega^p_{B} $ gives us a long exact sequence
\begin{align*}
\ldots\to &\H_{p+i-n-1}\cone(\Omega^{p}_{A/R,\cyc} \to\Omega^{p}_{B/R,\cyc} ) 
\to \pi_i (\Coiso(A,B;n)^{\nondeg}/F^{p+1})\\
&\to  \pi_i(\Coiso(A,B;n)^{\nondeg}/F^{p}) \to \H_{p+i-n-2}\cone (\Omega^{p}_{A/R,\cyc} \to\Omega^{p}_{B/R,\cyc}) 
\to\ldots 
\end{align*}
of homotopy groups. Since the spaces $\BiLag(A,B;n)/F^p$ fit in to identical long exact sequences, compatible with the natural maps $\Coiso(A,B;n)^{\nondeg}/F^{p}\to \BiLag(A,B;n)/F^p$, we deduce by induction  that 
the latter maps are
weak equivalences. 
\end{proof}

In order to proceed further, we need to extend the definition of shifted double co-isotropic structures to apply to morphisms $f \co A \to B$ of stacky DGAAs. Interpreting shifted double Poisson structures as algebras for the double Poisson protoperad is not straightforward in this stacky setting, because they are defined using a mixture of  strict and Tate dg category structures, as in the  commutative setting of \cite[\S \ref{DQpoisson-Artinsn}]{DQpoisson}. However, we can still study non-degenerate double co-isotropic structures without needing to introduce these subtleties, because shifted double Poisson structures satisfy homotopy \'etale functoriality.  

We begin by letting $P(B,n)=\Lim_r P(B,n)/F^{r+1}$ be the inverse system  
\[
 \{\prod_{0 \le p \le r} \cHom_{(B^{e})^{\ten p}}(((\Omega^1_{B/R})_{[-n-1]})^{\ten p},[B^{\ten (p+1)}])\}_r
 \]
 of double complexes, so $\widehat{\Pol}_{\pi}^{nc}(B,n) = \Lim_r \hatTot( P(B)/F^{r+1}) $. An $(n-1)$-shifted double Poisson structure $\pi$ on $B$ then defines an element $[\pi,-] \in \Lim_r \hatHHom(P(B)/F^{r}, P(B)/F^{r+1} )$, which we can think of as  a deformation $P_{\pi}(B)$ of $P(B)$.  Although $P_{\pi}(B,n-1)$ is not usually a stacky DGAA, since $\pi$ can contain terms with both chain and  cochain degrees being non-zero,  complexes such as $\hatTot \Omega^1_{P_{\pi}(B,n-1) }$ still exist. Similarly, we can define associative maps $A \to P_{\pi}(B,n-1)$ from a cofibrant stacky DGAA $A$ as consisting of elements of $\Lim_r (\hatHHom_R(A, P(B,n)/F^{r+1}), \delta \pm \pd + [\pi,-])$ respecting the respective multiplicative structures.
 
 We can then let a non-degenerate $n$-shifted double  co-isotropic structure on a morphism $f \co A \to B$ of stacky DGAAs in $dg_+DG^+\Alg(R)$ consist of  an $n$-shifted double Poisson structure $\varpi$ on $A$, an $(n-1)$-shifted  double Poisson structure $\pi$ on $B$, a homotopy formally \'etale  associative morphism $\tilde{f} \co A \to P_{\pi}(B,n-1) $ lifting  $f \co A \to    P(B,n)/F^{1}=B$, and a homotopy between the canonical $n$-shifted double Poisson structure on $P_{\pi}(B,n-1) $ and the structure induced from $\varpi$ by \'etale functoriality.
 
Inserting $\hatTot$s in the relevant places, Proposition \ref{coisoprop} then adapts to stacky DGAAs.   Since everything in sight satisfies homotopy \'etale functoriality, these definitions  also extend to  morphisms  $\eta \co G \to F$ of  homogeneous derived NC prestacks with bounded below cotangent complexes, similarly to  Definition \ref{PreBiSphgsdef}, so the proof of Corollary \ref{compatcor2} (and its generalisation in \S \ref{Artincompat})  still applies. The resulting comparison statement is then:
 
 \begin{corollary}\label{coisocor}
 On a morphism $\eta \co G \to F$ of  homogeneous derived NC prestacks with perfect cotangent complexes,
   there is a natural equivalence between the space $\BiLag(F,G;n)$ of $n$-shifted  bi-Lagrangian structures  and the space $\cD\Coiso(F,G;n)$ of non-degenerate $n$-shifted double co-isotropic structures.
\end{corollary}

\begin{remarks}\label{CostelloRozenblyumRmk}
Corollary \ref{coisocor} implies that  an  $n$-shifted  bi-Lagrangian structure on a morphism $F \to G$ induces an $(n-1)$-shifted double Poisson structure on $F$. Applied to Theorems \ref{Perfthm2Lag} and \ref{MorthmLag}, this implies that for any   $d$-dimensional right (resp. left) Calabi-Yau structure on a dg functor $\theta \co \cB \to  \cA$  (resp. $\phi \co \cA \to \cB$), there is a canonical $(2-d)$-shifted double Poisson structure on the derived NC prestack $\Perf_{\cB}$ of perfect $\cB$-modules (resp. $\Mor_{\cB}$ of derived Morita morphisms from $\cB$), which we can interpret as in Corollary \ref{Perfpoissoncor} (resp. Corollary \ref{Morpoissoncor}). In particular, this applies to $\Perf_B$ when $B$ is a pre-Calabi--Yau algebra, so effectively generalises \cite{IyuduKontsevichPreCYNCPoisson,FernandezHerscovichCyclicAinftyDoublePoisson}. 

In the commutative setting, there is an unpublished observation of Costello and Rozenblyum that a form of converse is true, with every 
$k$-shifted Poisson structure arising as a formal $(k+1)$-shifted Lagrangian in a twist of the shifted cotangent bundle. In our setting, this amounts to the observation that the natural map $r \co P_{\pi}(B,k) \to B$ is $(k+1)$-shifted bi-Lagrangian for any $k$-shifted double Poisson structure $\pi$ on a cofibrant stacky DGAA $B$; the formal $(k+1)$-shifted bi-symplectic structure $\phi$ on $P_{\pi}(B,k)$ satisfies $r^*\phi=0$, so the bi-Lagrangian structure is simply $(\phi,0)$.

We can then apply completed versions  of Propositions \ref{weilArtinprop} and \ref{weilhgsprop} to deduce that the map $\Pi_{S/R}r \co \Pi_{S/R} P_{\pi}(B,k) \to \Pi_{S/R}B$ of Weil restrictions  is $(k+1-d)$-shifted bi-Lagrangian for any  $R$-DGAA $S \in dg_+\Alg(R)$ which is perfect as an  $R$-module and equipped with a non-degenerate $R$-linear chain  map $\tr \co S/^{\oL}[S,S] \to R_{[-d]}$. The formal generalisation of Proposition \ref{coisoprop} then gives us a map
\[
 \cD\cP(F,k) \to \cD\cP( \Pi_{S/R}F,k-d)
\]
from the space of $k$-shifted double Poisson structures on a homogeneous NC prestack $F$ to the space
of $(k-d)$-shifted double Poisson structures on  $\Pi_{S/R}F$. 

In particular, this gives a  map  $\cD\cP(F,k) \to \cD\cP( \Pi_{\Mat_n}F,k)$ from the space of  $k$-shifted double Poisson structures on  $F$  to the space of  $k$-shifted double Poisson structures on the derived  NC prestack $\Pi_{\Mat_n}F$ of $n$-dimensional representations, and hence via Remark \ref{poissoncommrmk} to the space $\cP( (\Pi_{\Mat_n}F)^{\comm, \sharp},k)$ of $k$-shifted Poisson structures on the (commutative) derived stack of $n$-dimensional representations of $F$.

\end{remarks}



\bibliographystyle{alphanum}
\bibliography{references.bib}

\end{document}